\documentclass[11pt]{amsart}

\usepackage{amsmath,amsthm,amsfonts,amssymb,amscd,euscript,enumerate}
\usepackage{tikz}
\usepackage{tikz-cd}
\tikzcdset{arrow style= tikz, diagrams={>=stealth}}
\usetikzlibrary{arrows}
\usepackage[all,arc]{xy}
\usepackage{enumerate}
\usepackage{mathrsfs} 
\usepackage{thmtools}
\usepackage{datetime}
\usepackage[colorlinks, citecolor = blue, linkcolor = melon]{hyperref}
\usepackage{graphicx}
\usepackage{color}
\usepackage{cleveref}
\usepackage{comment}
\usepackage{float}
\input{xy}
\xyoption{all}
\usetikzlibrary{shapes}
\usetikzlibrary{tikzmark}
\usepackage{mathtools}
\usepackage{wrapfig}
\usepackage{stmaryrd}
\usepackage{fp}

\definecolor{melon}{rgb}{0.3, 0.2, 1}
\definecolor{amethyst}{rgb}{0.6, 0.4, 1}
\definecolor{darkspringgreen}{rgb}{0.09, 0.45, 0.27}

\calclayout

\newtheorem{thm}{Theorem}[section]
\newtheorem{cor}[thm]{Corollary}
\newtheorem{prop}[thm]{Proposition}
\newtheorem{lem}[thm]{Lemma}

\theoremstyle{definition}
\newtheorem{defn}[thm]{Definition}

\newtheorem{exmp}[thm]{Example}

\newtheorem{defnlem}[thm]{Definition-Lemma}

\theoremstyle{remark}
\newtheorem{rem}[thm]{Remark}

\makeatletter
\let\c@equation\c@thm
\makeatother
\numberwithin{equation}{section}

\makeatletter
\newcommand*\bigcdot{\mathpalette\bigcdot@{.5}}
\newcommand*\bigcdot@[2]{\mathbin{\vcenter{\hbox{\scalebox{#2}{$\m@th#1\bullet$}}}}}
\makeatother

\makeatletter
\def\subsection{\@startsection{subsection}{3}
 \z@{.5\linespacing\@plus.7\linespacing}{.1\linespacing}
 {\bfseries}}
\makeatother

\newcommand{\Z}{\mathbb{Z}}
\newcommand{\N}{\mathbb{N}}
\newcommand{\R}{\mathbb{R}}

\newcommand{\calA}{\mathcal{A}}

\newcommand{\rngB}{\mathbb{Z}[\mathbb{P}]}
\newcommand{\calC}{\mathcal{C}}
\newcommand{\calD}{\mathcal{D}}
\newcommand{\calF}{\mathcal{F}}

\newcommand{\calP}{\mathcal{P}}
\newcommand{\mndP}{\mathscr{P}}

\newcommand{\calT}{\mathcal{T}}

\newcommand{\frakd}{\mathfrak{d}}
\newcommand{\frakm}{\mathfrak{m}}
\newcommand{\frakp}{\mathfrak{p}}
\newcommand{\frakD}{\mathfrak{D}}
\renewcommand{\phi}{\varphi}

\DeclareMathOperator{\sh}{sh}
\DeclareMathOperator{\Supp}{Supp}

\DeclareMathOperator{\cut} { \setminus}

\DeclareMathOperator{\BL}{BL}
\DeclareMathOperator{\CG}{CG}
\DeclareMathOperator{\Scat}{Scat}
\DeclareMathOperator{\Sing}{Sing} % singular locus of sd
\DeclareMathOperator{\wt}{wt}
\newcommand{\east}{\mathbf{E}}
\newcommand{\north}{\mathbf{N}}
\newcommand{\aone}{d_1}
\newcommand{\atwo}{d_2}
\newcommand{\clco}[3]{\lambda(#1, #2; #3)} % change-of-lattice coefficients
\newcommand{\bline}{\beta} % broken line
\newcommand{\dvect}{{\mathbf d}} % d-vector for rk 2 cluster monomials
\newcommand{\mnd}{\mathbb M} % monoid rk 2
\newcommand{\cpv}{\widetilde x} % cluster pre-variable
\newcommand{\Din}{\mathfrak{D}_{\mathrm{in}}}

\usepackage{etoolbox}
\makeatletter
\patchcmd{\@maketitle}
 {\ifx\@empty\@dedicatory}
 {\ifx\@empty\@date \else {\vskip3ex \centering\footnotesize\@date\par\vskip1ex}\fi
  \ifx\@empty\@dedicatory}
 {}{}
\patchcmd{\@adminfootnotes}
 {\ifx\@empty\@date\else \@footnotetext{\@setdate}\fi}
 {}{}{}
\makeatother

\newcommand{\vrec}[4]{
  \FPeval{\newz}{trunc(#3,0)-1} 
  \foreach \i in {0,...,\newz} {
    \fill[cyan] (#1 + 0.2, #2  + 0.2 + #4 * \i) -- (#1 + 0.2, #2 + #4 - 0.2 + #4 * \i) -- (#1 + 0.8, #2 + #4 - 0.2 + #4 * \i) -- (#1 + 0.8, #2 + 0.2 + #4 * \i) -- (#1 + 0.2, #2 + 0.2 + #4 * \i);
  }
}

\newcommand{\hrec}[4]{
  \FPeval{\newz}{trunc(#3,0)-1} 
  \foreach \i in {0,...,\newz} {
    \fill[red] (#1 - 0.2 - #4 * \i, #2 + 0.2) -- (#1 - #4 + 0.2 - #4 * \i, #2 + 0.2) -- (#1 - #4 + 0.2 - #4 * \i, #2 + 0.8) -- (#1 - 0.2 - #4 * \i, #2 + 0.8) -- (#1 - 0.2 - #4 * \i, #2 + 0.2);
  }
}

\title{Positivity of generalized cluster scattering diagrams}

\author{Amanda Burcroff}
\address{Department of Mathematics, Harvard University, Cambridge, MA 02138, USA}
\email{\href{mailto:aburcroff@math.harvard.edu}{aburcroff@math.harvard.edu}}

\author{Kyungyong Lee}
\address{Department of Mathematics, University of Alabama, Tuscaloosa, AL 35487, USA and Korea
Institute for Advanced Study, Seoul 02455, Republic of Korea}
\email{\href{mailto:kyungyong.lee@ua.edu}{kyungyong.lee@ua.edu} \& \href{mailto:klee1@kias.re.kr}{klee1@kias.re.kr}}

\author{Lang Mou}
\address{Department of Mathematics, University of Cologne, 50923 Köln, Germany}
\email{\href{mailto:langmou@math.uni-koeln.de}{langmou@math.uni-koeln.de}}

\begin{document}

\begin{abstract}
  We introduce a new class of combinatorial objects, named \emph{tight gradings}, which are certain nonnegative integer-valued functions on maximal Dyck paths. Using tight gradings, we derive a manifestly positive formula for any wall-function in a rank-2 generalized cluster scattering diagram. We further prove that any consistent rank-2 scattering diagram is positive with respect to the coefficients of initial wall-functions. Moreover, our formula yields explicit expressions for relative Gromov--Witten invariants on weighted projective planes and the Euler characteristics of moduli spaces of framed stable representations on complete bipartite quivers. Finally, by leveraging the rank-2 positivity, we show that any higher-rank generalized cluster scattering diagram has positive wall-functions, which leads to a proof of the positivity of the Laurent phenomenon and the strong positivity of Chekhov--Shapiro's generalized cluster algebras.
\end{abstract}

\maketitle

\setcounter{tocdepth}{1}
\vspace{-.3in}
\tableofcontents

\section{Introduction} \label{sec: intro}

In Fomin--Zelevinsky's cluster algebras \cite{FZ}, cluster variables are iteratively generated rational functions in $n$ algebraically independent variables that are proven to remain Laurent polynomials, a property known as the \emph{Laurent phenomenon} \cite{FZlaurent}. It was conjectured originally in \cite{FZ} that the non-trivial coefficients in these Laurent polynomials are positive integers. The positivity of the Laurent phenomenon turned out to be a deep property that has driven plentiful research employing wide range of techniques across combinatorics, representation theory and algebraic geometry including \cite{Nak11, Qin12, MSW, Davison}. In this paper, we obtain new positivity results for the generalized cluster algebras of Chekhov and Shapiro \cite{CS} where polynomial type exchange relations (as opposed to binomials) are used to generate cluster variables. Our proofs are built upon a novel counting formula for wall-functions to rank-2 generalized cluster scattering diagrams and the subsequent positivity of wall-functions in higher ranks.

Lee and Schiffler \cite{LSpositive} first proved the positivity for all skew-symmetric cluster algebras. Their proof is of an algebraic combinatorial nature and uses an explicit rank-2 formula of cluster variables in terms of combinatorics on Dyck paths \cite{LS}. This breakthrough led to Lee--Li--Zelevinsky's construction of \emph{greedy bases} of all rank-2 cluster algebras \cite{LLZ} and the subsequent generalization by Rupel \cite{Rupgengreed}. Cluster monomials, which are monomials of cluster variables in one cluster, form only a proper subset of the greedy basis except for finite-type cases. So it is a significant feature of greedy bases that they give explicit combinatorial information even outside the discrete cluster dynamics.

In their seminal work \cite{GHKK}, Gross, Hacking, Keel, and Konstevich solved a number of conjectures fundamental to the structures of cluster algebras by bringing in novel techniques from log Calabi--Yau mirror symmetry \cite{GHK}. Scattering diagrams \cite{KS,GS}, also known as wall-crossing structures, play a crucial role in this study. In particular, the authors constructed \emph{theta functions} through cluster scattering diagrams and showed that the theta functions include all cluster monomials. The positivity of the Laurent phenomenon follows from the positivity of wall-functions in cluster scattering diagrams.

A main theme of \cite{GHKK} is to address Fock--Goncharov's cluster duality \cite{FG06,FG09}, in which it was conjectured that the algebra of functions on a cluster variety $V$ has a canonical basis containing global monomials parametrized by the tropical points $V^\vee(\Z^T)$ of the mirror cluster variety $V^\vee$. Although this conjecture is false in general \cite{GHKcluster}, theta functions do serve as a topological basis that exists in a certain formal limit. For $V$ an $\mathcal A$-type cluster variety, global monomials are precisely cluster monomials and they are parametrized by the tropical points inside the cluster complex $\Delta^+$, which can be realized as a simplicial cone complex in the ambient real vector space of a cluster scattering diagram. The cluster complex $\Delta^+$ is usually infinite but not complete. As noted in \cite{GHKK}, it is precisely the rich structure outside $\Delta^+$ that scattering diagram technique brings into the the cluster algebra picture.

Our discoveries stem from an effort to bridge the explicit combinatorial terms developed in \cite{LS, LSpositive, LLZ, Rupgengreed} and the more abstract constructions of \cite{GHKK} through scattering diagrams. Such a precedent is that greedy bases and theta bases are actually the same in rank 2 \cite{CGM}. However, the explicit scattering diagram structure outside $\Delta^+$ is still elusive, even in rank 2. While one expects a non-trivial wall in every rational direction to appear in this region (\Cref{app: Badlands nonvanishing}), the wall-function in general is notoriously hard to compute. Our formula presented in \Cref{thm: tight grading formula intro} computes any such wall-function in terms of combinatorics on Dyck paths. In fact, the formula itself is far more general than the ordinary cluster situation. As explained in \Cref{subsec: formula wall-function intro}, the formula essentially computes the commutator of two arbitrary generators of the tropical vertex group \cite{KS, GPS}. As rank-2 scattering diagrams are already linked with other interesting objects \cite{GPS,RW}, our formula invites immediate applications to Gromov--Witten invariants and quiver representations; see \Cref{sec: application}.

An immediate consequence of \Cref{thm: tight grading formula intro} is that any wall-function has nonnegative integer coefficients. We extend this positivity to wall-functions in any rank-2 consistent scattering diagram (\Cref{thm: positivity rk 2 intro}) and in any higher-rank generalized cluster scattering diagram (\Cref{thm: positive wall-function intro}). The latter scattering diagrams encode structures of the associated generalized cluster algebras \cite{Mou,CKM}. The positivity of wall-functions allows us to prove the positivity of the Laurent phenomenon for any generalized cluster algebra, previously conjectured by Chekhov and Shapiro \cite[Section 5]{CS}. The \emph{strong positivity}, which says that the structure constants in the multiplication of theta functions are nonnegative, also follows.

We now give more detailed description of the main results in \Cref{subsec: formula wall-function intro} through \ref{subsec: positivity cluster intro}.

\subsection{Formula for wall-functions}\label{subsec: formula wall-function intro}

Let $\Bbbk$ be a field of characteristic zero. Following \cite{GPS}, we consider the automorphism group
\[
	A \coloneqq \operatorname{Aut}_{\Bbbk\llbracket t\rrbracket }(\Bbbk[x^{\pm 1}, y^{\pm 1}]\llbracket t\rrbracket )
\]
of the $1$-parameter formal family of the two-dimensional torus. For a non-zero vector $(a, b)\in \Z^2$ and a function $f\in \Bbbk[x^{\pm 1}, y^{\pm 1}]\llbracket t\rrbracket $ of the form
\[
    f = 1 + tx^ay^b \cdot g(x^ay^b, t), \quad g(z, t)\in \Bbbk[z]\llbracket t\rrbracket ,
\]
we define $T_{(a, b), f}\in A$ by
\[
    T_{(a, b), f}(x) = f^{-b}\cdot x \quad \text{and} \quad T_{(a, b), f}(y) = f^a\cdot y.
\]
Clearly $T_{(a, b),f}^{-1} = T_{(a, b), f^{-1}}\in A$.

The \emph{tropical vertex group} is the completion with respect to the maximal ideal $(t)\subseteq \Bbbk\llbracket t\rrbracket $ of the subgroup of $A$ generated by all elements of the form $T_{(a, b), f}$. This group and its close variants arose in mirror symmetry \cite{KS, GS}, and subsequently also play important roles in wall-crossing identities of invariants in triangulated categories \cite{KS08,KSsummary}.

One initial task to understand the tropical vertex group is to compute the ordered factorization of the commutator of two generators. This turns out to be a rather hard question, but investigations into it have led to links to interesting objects including relative Gromov--Witten invariants \cite{GPS,GP} and moduli spaces of stable quiver representations \cite{Rei08,Rei11,RW}. Our approach to this problem is combinatorial and is largely motivated by Fomin--Zelevinsky's cluster structure.

Let $P_1 = 1 + \sum_{k\geq 1}p_{1, k} x^k$ and $P_2 = 1 + \sum_{k\geq 1} p_{2, k} y^{k}$ be elements in $\Bbbk[x^{\pm 1}, y^{\pm 1}]\llbracket t\rrbracket $. Here $p_{i,k}\in \Bbbk\llbracket t\rrbracket $, $i = 1, 2$ are regarded as coefficients. As observed by Kontsevich and Soibelman \cite{KS}, there exists a unique ordered factorization of the commutator
\begin{equation}\label{eq: factorization tropical vertex}
    T_{(0, 1), P_2}^{-1} \cdot T_{(1, 0), P_1} \cdot T_{(0, 1), P_2} \cdot T_{(1, 0), P_1}^{-1} = \prod_{b/a\ \text{decreasing}}^{\rightarrow} T_{(a, b), f_{(a, b)}}.
\end{equation}
The right-hand side of \eqref{eq: factorization tropical vertex} is an infinite product ranging over all positive coprime pairs $(a, b)$. It is understood as the limit of finite products modulo each $(t^n)$ for $n\geq 1$.

\begin{thm}[Tight grading formula]\label{thm: tight grading formula intro}
    Fix coprime positive integers $(a,b)$.  For each $k \geq 1$, choose integers $(d_{1,k},d_{2,k})$ such that $ad_{2,k} - bd_{1,k} = \pm 1$, $ka \leq d_{1,k}$, and $kb \leq d_{2,k}$.  Then
    \begin{equation}\label{eq: tight grading formula intro}
        f_{(a, b)} = 1 + \sum_{k\geq 1}\sum_{\omega} \mathrm{wt}(\omega) x^{ka}y^{kb},  
    \end{equation}
    where the second sum is over all tight gradings $\omega$ on the Dyck path $\mathcal P(d_{1,k}, d_{2,k})$ of total horizontal weight $kb$ and total vertical weight $ka$, and $\mathrm{wt}(\omega)\in \Bbbk\llbracket t\rrbracket $ is the weight monomial in $\{p_{i, k}\mid i = 1, 2, k\geq 1\}$ associated to $\omega$.
\end{thm}

Tight gradings, the main combinatorial ingredients in \Cref{thm: tight grading formula intro}, have simple descriptions in terms of combinatorics on the Dyck path $\mathcal P(d_{1, k}, d_{2, k})$ given in \Cref{subsec: compatible grading} and \Cref{subsec: shadow and tight grading}. They form a subset of Rupel's compatible gradings \cite{Rupgengreed}, whose definition was inspired by the compatible pairs defined by Lee--Li--Zelevinsky in \cite{LLZ}. The formula \eqref{eq: tight grading formula intro} was announced in \cite[Theorem 3.5]{BLMshort} with a proof sketch. We give two different proofs in \Cref{subsec: first proof} and \Cref{subsec: second proof}.

\Cref{thm: tight grading formula intro} is a special case of \Cref{thm: power shadow grading formula}, where we give a combinatorial formula of the $m$-th power of $f_{(a, b)}$ for any $m\geq 1$ using more general \emph{shadowed gradings}. The more general formula will be used to compute Euler characteristics of quiver moduli spaces; see \Cref{cor: euler char by shadowed grading}. It also yields explicit expressions of relative Gromov--Witten invariants on weighted projective planes; see \Cref{thm: gw by shadowed grading}.

We give in \Cref{thm: positivity rk 2 non-standard} a formula for the commutator of two arbitrary generators $T_{(a_1, b_1), f_1}$ and $T_{(a_2, b_2), f_2}$ in the tropical vertex group. The difference with \Cref{thm: tight grading formula intro} is that the vectors $(a_1, b_1)$ and $(a_2, b_2)$ may only generate a proper rank-2 sub-lattice of $\Z^2$. One may compare this formula with Gross--Pandharipande--Siebert's \emph{full commutator formula} \cite[Section 5]{GPS}, which involves relative Gromov--Witten invariants defined through orbifold blowups of toric surfaces and a process factorizing $f_1,f_2$ into products of binomials. Our formula is direct in terms of the coefficients of $f_1, f_2$ and by counting combinatorial objects. It is thus much simpler from the perspective of computation and can be efficiently implemented through computer programs.

\subsection{Positivity of scattering diagrams}\label{subsec: positivity sd intro}

In the main body of the paper, the factorization \eqref{eq: factorization tropical vertex} is packed in a structure called a rank-2 \emph{consistent scattering diagram} determined by two initial \emph{lines}. More generally such a structure $\frakD = \Scat(\frakD_\mathrm{in})$ can start with a countable collection $\frakD_\mathrm{in}$ of lines in $\R^2$ and result in adding a collection $\frakD\cut\frakD_\mathrm{in}$ of \emph{rays} in $\R^2$ to achieve consistency. Lines and rays are uniformly called \emph{walls} and are attached with \emph{wall-functions}. Each $f_{(a, b)}$ in \eqref{eq: factorization tropical vertex} is the wall-function associated to the ray $\R_{\leq 0}(a, b)$ while $P_1$ and $P_2$ are associated to the coordinate axes. We will soon see this description in \Cref{subsec: gen cluster sd rk 2}.

The starting point of the remaining discussion is the positivity demonstrated by the tight grading formula \eqref{eq: tight grading formula intro}. The combined coefficient for each term $x^{ka}y^{kb}$ is a polynomial with nonnegative integer coefficients, that is, an element in
\[
    \N[p_{i,k}\mid i = 1, 2, k\geq 1],
\]
where $\N$ denotes $\Z_{\geq 0}$ throughout. It turns out that this positivity is shared by any consistent scattering diagram in rank 2.

\begin{thm}[\Cref{thm: full positivity rk 2}]\label{thm: positivity rk 2 intro}
    Let $\frakD_\mathrm{in}$ be an initial scattering diagram consisting of lines. Any coefficient of any wall-function in $\Scat(\frakD_\mathrm{in})$ can be expressed as a polynomial in the coefficients of wall-functions in $\frakD_\mathrm{in}$ with nonnegative integer coefficients.
\end{thm}

We next utilize \Cref{thm: positivity rk 2 intro} to show the positivity of wall-functions of certain scattering diagrams living in higher dimensions. There will be immediate implications on several positivity phenomena of generalized cluster algebras.

Briefly, a scattering diagram in $\R^N$ is a (countable) collection of codimension one rational polyhedral cones, each with an attached power series, referred to as \emph{walls} and their \emph{wall-functions}. The (ordinary) cluster scattering diagram of Gross--Hacking--Keel--Kontsevich \cite{GHKK} is determined by a finite collection of initial walls through a higher dimensional consistency condition that generalizes the equation \eqref{eq: factorization tropical vertex}. These scattering diagrams encode the information of Fomin--Zelevinsky's cluster algebras.

In more concrete terms, a \emph{generalized cluster scattering diagram} $\frakD_\mathbf s$ is the unique consistent completion of an initial collection of walls
\[
    \frakD_{\mathbf s, \mathrm{in}} = \{(H_i, f_i)\mid i = 1, \dots, n\}
\]
such that $\frakD_\mathbf s\cut \frakD_{\mathbf s, \mathrm{in}}$ consists only of non-incoming walls, a notion in higher dimensions generalizing rays in rank 2. Each $H_i$ is the coordinate hyperplane in $\R^{N}$ (for $n\leq N$) with vanishing $i$-th coordinate and
\[
    f_i = 1 + p_{i, 1}x^{w_i} + \dots + p_{i, \ell_i} x^{\ell_iw_i},
\]
where each $w_i$ is a (column) vector in $\Z^N$, $\ell_i$ is a positive integer, and $p_{i,k}$ are formal variables. The matrix $(\ell_1w_1, \dots, \ell_nw_n)$ is assumed to have full rank. Its top $n\times n$ submatrix needs to be skew-symmetrizable, as required for exchange matrices that govern mutations in cluster algebras. These scattering diagrams, constructed and studied in \cite{Mou,CKM}, generalize the cluster scattering diagrams of \cite{GHKK}, where each $P_i$ is a binomial $1 + x^{\ell_iw_i}$. They encode the information of Chekhov--Shapiro's generalized cluster algebras \cite{CS}.

A significant property of ordinary cluster scattering diagrams is the positivity of wall-functions \cite[Theorem 1.24]{GHKK}, which says that any wall-function is a (possibly infinite) product of binomials of the form $(1 + x^w)$. Thus, as a power series, every wall-function has nonnegative integer coefficients. The proof in \emph{loc.~cit} heavily relies on the assumption that the initial wall-functions are binomials. Our next main result shows the positivity of wall-functions in the polynomial case.

\begin{thm}[\Cref{thm: positive wall-function}]\label{thm: positive wall-function intro}
    Any wall-function in a generalized cluster scattering diagram $\frakD_\mathbf s$ has coefficients in $\N[p_{i, s}\mid 1\leq i \leq n, 1\leq s \leq \ell_i]$.
\end{thm}

The proof of \Cref{thm: positive wall-function intro} relies on \Cref{thm: positivity rk 2 intro}, the rank-2 positivity in full generality. In \Cref{subsec: positivity higher rank}, we give an order-by-order construction of $\frakD_\mathbf s$ analogous to the algorithm in \cite[Appendix C]{GHKK}, where local scatterings are essentially governed by rank-2 behaviors. Then the proof undergoes a careful induction that ensures positivity at each order.

\begin{rem}
    It would be interesting to explore quiver representation-theoretic interpretations of $\frakD_\mathbf s$ beyond the rank-2 case (\Cref{subsec: euler char quiver}). Following Bridgeland \cite{Bri17}, stability scattering diagrams equivalent to a particular class of $\frakD_\mathbf s$ are constructed in \cite{LFM} through representations of Jacobian algebras whose quivers have loops. \Cref{thm: positive wall-function intro} then implies that the Euler characteristic type wall-function coefficients in \cite{LFM} are all nonnegative integers.
\end{rem}

\subsection{Positivity of generalized cluster algebras}\label{subsec: positivity cluster intro}

The remaining structural properties regarding the positivity of generalized cluster algebras follow from \Cref{thm: positive wall-function intro}, summarized below. This is a series of results extended from the ordinary case \cite{GHKK}.

\begin{thm}\label{thm: struct thm intro}
    Let $\mathcal A$ (resp. $\mathcal U$) denote the generalized cluster algebra (resp. upper cluster algebra) associated to a generalized cluster structure (with coefficients and frozen directions) of Chekhov--Shapiro as defined in \emph{\Cref{sec: gen cluster algebra}}.
    \begin{enumerate}
        \item[\emph{(1)}] Every cluster monomial in $\mathcal A$ is a Laurent polynomial in the cluster variables of any cluster with coefficients in $\N[\mathbb P]$ where $\mathbb P$ is the coefficient group. \emph{(\Cref{thm: positive Laurent})}
        \item[\emph{(2)}] When the cluster structure is generated by a seed data $\mathbf s$ \emph{(\Cref{subsec: gen cluster sd higher rk})} that determines $\frakD_\mathbf s$, we have the so-called theta functions $\{\vartheta_m\mid m\in \mathbb Z^N\}$ constructed from $\frakD_\mathbf s$ in certain formal completion $\widehat{\mathcal U^\oplus}$ (dependent on $\mathbf s$) of $\mathcal U$. Each $\vartheta_m$ is a formal Laurent series with coefficients in $\N[\mathbb P]$. The set of theta functions contain all cluster monomials. \emph{(\Cref{thm: cluster monomial theta})}  
        \item[\emph{(3)}] Theta functions form a topological basis of $\,\widehat{\mathcal U^\oplus}$ and satisfy the formal multiplicative rule
        \[
            \vartheta_{p_1}\cdot \vartheta_{p_2} = \sum_{q\in \mathbb Z^n} \alpha(p_1, p_2; q) \vartheta_{q}\in \widehat{\mathcal U^\oplus}, 
        \]
        where each structure constant $\alpha(p_1, p_2;q)$ is defined through broken lines and is in $\N[\mathbb P]$. \emph{(\Cref{thm: strong positivity})}
        \item[\emph{(4)}] The $\Z[\mathbb P]$-module
        \[
            \mathcal M\coloneqq \operatorname{span}_{\Z[\mathbb P]}\{\vartheta_m\mid m\in \Z^N, \vartheta_m\in \mathcal U\}
        \]
        is closed under multiplication, hence a $\Z[\mathbb P]$-algebra. There are natural inclusions $\mathcal A \subseteq \mathcal M \subseteq \mathcal U$. \emph{(\Cref{thm: middle cluster algebra})}
        \item[\emph{(5)}] Theta functions are characterized as the atomic elements in $\widehat{\mathcal U^\oplus}$ with respect to the scattering atlas. \emph{(\Cref{thm: atomicity})}
    \end{enumerate}
\end{thm}

\Cref{thm: struct thm intro}(1) confirms the \emph{positivity of the Laurent phenomenon} for generalized cluster algebras, previously conjectured by Chekhov and Shapiro \cite{CS}. It is proven in \Cref{subsec: theta and positive} through \Cref{thm: struct thm intro}(2). We refer to \Cref{thm: struct thm intro}(3) as the \emph{strong positivity}.  There is an extensive discussion in \cite{GHKK} on when any of the inclusions $\mathcal A\subseteq \mathcal M \subseteq \mathcal U$ in the ordinary case becomes equality, which we expect to naturally carry over to the generalized case but do not discuss in this paper. \Cref{thm: struct thm intro}(5) is suggested to the authors by Ben Davison. The proof follows Mandel's \cite{Man17} in the ordinary case.

\subsection{Structure of the paper}

The remainder of the paper is organized as follows. In \Cref{sec: rk 2 sd and theta functions}, we review rank-2 scattering diagrams and important constructions such as broken lines and theta functions. Generalized cluster scattering diagrams in rank 2 are studied in detail in \Cref{subsec: gen cluster sd rk 2}, whose cluster complex structure leads to the definition of rank-2 generalized cluster algebras in \Cref{subsec: bl and theta}.

In \Cref{subsec: greedy elements}, we define greedy elements for rank-2 generalized cluster algebras through a positive recurrence of coefficients. Compatible gradings are defined in \Cref{subsec: compatible grading} and are used to express greedy elements in \Cref{subsec: greedy by cg}.

We establish the equality between greedy elements and theta functions in \Cref{sec: greedy = theta}, generalizing \cite{CGM}. This serves as a main ingredient in our first proof of \Cref{thm: power shadow grading formula} given in \Cref{subsec: first proof}.

In \Cref{sec: shadowed and tight grading}, we introduce the main combinatorial objects \emph{shadowed gradings} and \emph{tight gradings} and prove \Cref{thm: power shadow grading formula}, which is a more general formula covering \Cref{thm: tight grading formula intro}. Definitions and examples of shadowed and tight gradings are given in \Cref{subsec: shadow and tight grading}. Another important ingredient towards \Cref{thm: power shadow grading formula} is the \emph{one-bending property} regarding combinatorics of broken lines, introduced in \Cref{subsec: one-bending}. Two different proofs for \Cref{thm: power shadow grading formula} are given respectively in \Cref{subsec: first proof,subsec: second proof}.

In \Cref{sec: application}, we apply \Cref{thm: power shadow grading formula} to obtain combinatorial formulas for Euler characteristics of quiver moduli spaces and relative Gromov--Witten invariants.

\Cref{sec: gen cluster algebra,sec: gen cluster sd higher rk} deal with positivity properties of generalized cluster algebras. We review Chekhov--Shapiro's generalized cluster structure in \Cref{sec: gen cluster algebra} and state the positivity of Laurent phenomenon in \Cref{thm: positive Laurent}, whose proof is given in \Cref{subsec: theta and positive}. Generalized cluster scattering diagrams are defined in \Cref{subsec: gen cluster sd higher rk}. The cluster complex structure is described in \Cref{subsec: cluster complex high rk}. Other structural results regarding generalized cluster algebras such as the strong positivity and atomicity of theta functions are discussed in \Cref{subsec: strong positivity,subsec: atomicity}. 

In \Cref{sec: proof of positivity}, we fill in rather technical results on positivity of scattering diagrams. We establish the full positivity in rank 2 in \Cref{subsec: non-standard,subsec: full positivity rk 2}. The proof of \Cref{thm: positive wall-function}, the positivity of higher-rank generalized cluster scattering diagrams, is completed in \Cref{subsec: positivity higher rank}.

\Cref{app: generalized greedy element formula} contains the proof of the compatible grading formula for greedy elements in generalized cluster algebras.  In \Cref{app: Badlands nonvanishing}, we give a combinatorial proof that the Badlands wall-function coefficients are non-vanishing in rank-$2$ generalized cluster scattering diagrams whose initial wall-functions have positive coefficients.

\addtocontents{toc}{\protect\setcounter{tocdepth}{-1}}
\section*{Acknowledgments}
\addtocontents{toc}{\protect\setcounter{tocdepth}{1}}
AB was supported by the Jack Kent Cooke Foundation and NSF grant DGE 2140743. KL was supported by the University of Alabama, Korea Institute for Advanced Study, and the NSF grants DMS 2042786 and DMS 2302620. LM was supported by the Deutsche Forschungsgemeinschaft (DFG, German Research Foundation) -- Project-ID 281071066 -- TRR 191. The authors thank Gregg Musiker and Lauren Williams for helpful suggestions. Special thanks to Bernd Siebert for valuable discussions, and to Markus Reineke for pointing out a minor mistake on an earlier version. The authors are grateful to Ryota Akagi, H\"ulya Arguz, Ben Davison, Sergey Fomin, Michael Gekhtman, Mark Gross, Sean Keel, Fang Li, Li Li, Siyang Liu, Travis Mandel, Tomoki Nakanishi, Jie Pan, Fan Qin, Michael Shapiro and Salvatore Stella for their correspondences and helpful discussions.

\section{Rank 2 scattering diagrams and theta functions}\label{sec: rk 2 sd and theta functions}

In this section, we review rank-2 scattering diagrams, broken lines and theta functions. Generalized cluster scattering diagrams in rank 2 are studied in detail in \Cref{subsec: gen cluster sd rk 2}, whose cluster complex structure leads to the definition of rank-2 generalized cluster algebras in \Cref{subsec: bl and theta}.

\subsection{Scattering diagrams in rank 2}\label{subsec: sd in rk 2}

We introduce basic notions of rank-2 scattering diagrams following \cite[Section 6.3]{G11}. Fix a rank-$2$ lattice $M \cong \Z^2$ and let $N = \operatorname{Hom}(M, \Z)$. Let $\mndP$ be a finitely generated monoid equipped with a map
\[
  r \colon \mndP \rightarrow M.
\]

Denote by $\mndP^\times$ the group of units in $\mndP$. Let $\mathfrak m$ denote both the monoid ideal $\mndP\cut \mndP^\times$ in $\mndP$ and the monomial ideal $\Bbbk[\mndP\cut \mndP^\times]$ in the monoid algebra $\Bbbk[\mndP]$ over a ground field $\Bbbk$ of characteristic zero. Let 
\[
  \widehat{\Bbbk[\mndP]}\coloneqq \lim_{\longleftarrow} \Bbbk[\mndP]/\mathfrak m^k
\]
be the $\mathfrak m$-adic completion of $\Bbbk[\mndP]$.

We say a monomial ideal $I$ in $\Bbbk[\mndP]$ is \emph{$\mathfrak m$-primary} if its radical $\sqrt{I}$ is $\mathfrak m$.

\begin{defn}\label{def: ray and line}
  Fix $r\colon \mndP \rightarrow M$ and let $I$ be an $\mathfrak m$-primary ideal.
  A \emph{ray} or \emph{line}, uniformly referred to as a \emph{wall}, over $\Bbbk[\mndP]/I$ or $\widehat{\Bbbk[\mndP]}$ is the data $\frakd$ of
      \begin{itemize}
        \item the \emph{support} $S_\frakd\subseteq M_\mathbb R \coloneqq M\otimes \mathbb R$ given by
        \[
          S_\frakd = b - \mathbb R_{\geq 0}m_{\frakd}
          \quad \text{or} \quad
          S_\frakd = b - \mathbb R m_{\frakd}
        \]
        for some $b \in M_\mathbb R$ and the \emph{direction} $-m_{\frakd}\in M\cut \{0\}$;
        \item the \emph{wall-function} $f_\frakd$ is an element in $\Bbbk[\mndP]/I$ or in $\widehat{\Bbbk[\mndP]}$ such that
        \[
          f_\frakd = 1 + \sum_{m\in \mndP({m_\frakd})} c_m x^m 
        \]
        where $\mndP({m_\frakd}) \coloneqq \{m\in \mndP \mid r(m)\in \mathbb R_{>0}m_\frakd\}$ and $f_\frakd \equiv 1 \mod \frakm$.
      \end{itemize}
\end{defn}

\begin{rem}
  A ray does not start at the origin when $b\neq 0$ and a line does not cross the origin when $b\notin \mathbb R m_\frakd$. A wall $\frakd$ is sometimes represented by the pair $(S_\frakd, f_\frakd)$, but two walls can have the same pair of support and wall-function.
\end{rem}

\begin{defn}\label{def: sd rk 2}
  A \emph{scattering diagram $\mathfrak D$ over $\Bbbk[\mndP]/I$} is a finite set of walls over $\Bbbk[\mndP]/I$. A \emph{scattering diagram $\mathfrak D$ over $\widehat{\Bbbk[\mndP]}$} is a countable set of walls over $\widehat{\Bbbk[\mndP]}$ such that for each $k\geq 0$, the subset
  \[
    \mathfrak D_k \coloneqq \{\mathfrak d\in \mathfrak D\mid f_\mathfrak d\not\equiv 1\mod \mathfrak m^{k+1}\}
  \]
  is a scattering diagram over $\Bbbk[\mndP]/\mathfrak m^{k+1}$.
\end{defn}

\begin{defn}
  The \emph{support} of $\frakD$ is defined as
  \[
    \operatorname{Supp}(\frakD) \coloneqq \bigcup_{\mathfrak d\in \mathfrak D} S_{\mathfrak d} \subseteq M_\mathbb R.
  \]
  The \emph{singular locus} $\operatorname{Sing}(\frakD)$ of $\frakD$ is the union of the boundary point of every ray and the zero-dimensional intersection locus of any pair of walls.
\end{defn}

Associated to a wall $\mathfrak d$ over $\Bbbk[\mndP]/I$ or $\widehat{\Bbbk[\mndP]}$ and a vector $v\in M_\mathbb R$ transversal to $S_\mathfrak d$ is the \emph{wall-crossing automorphism} $\mathfrak p_{v, \mathfrak d}\in \operatorname{Aut}(\Bbbk[\mndP]/I)$ or $\operatorname{Aut}(\widehat{\Bbbk[\mndP]})$. It is an algebra automorphism defined by
\begin{equation}\label{eq: wcs automorphism}
    \mathfrak p_{v, \frakd}(x^m) \coloneqq x^m f_\frakd^{n\cdot r(m)} \quad \text{for}\quad m\in \mndP,
\end{equation}
where $n\in \operatorname{Hom}(M, \mathbb Z)$ is primitive and orthogonal to $S_\frakd$ such that $n\cdot v<0$, where $\cdot$ denotes the natural pairing between $\operatorname{Hom}(M, \mathbb Z)$ and $M$. Clearly $\mathfrak p_{v, \frakd}^{-1} = \mathfrak p_{-v, \frakd}$.

A path (a smooth immersion) $\gamma \colon [0, 1]\rightarrow M_\mathbb R$ is called \emph{regular} with respect to a scattering diagram $\frakD$ over $\Bbbk[\mndP]/I$ if it is a smooth immersion such that
  \begin{itemize}
    \item $\gamma([0, 1])\cap \operatorname{Sing}(\frakD) = \varnothing$ and $\{\gamma(0), \gamma(1)\} \cap \operatorname{Supp}(\frakD) = \varnothing$;
    \item it crosses walls only transversally.
  \end{itemize}
For a regular path $\gamma$, let
\[
  0 < t_1 \leq \cdots \leq t_s < 1
\]
be a sequence of maximal length such that $\gamma(t_i)\in S_{\mathfrak d_i}$ for some wall $\mathfrak d_i\in \frakD$ and that if $t_i = t_j$, then $\mathfrak d_i$ and $\mathfrak d_j$ are two distinguished elements in $\frakD$.

  The \emph{path-ordered product} of $\gamma$ is defined to be
  \[
    \mathfrak p_{\gamma, \mathfrak D} \coloneqq \mathfrak p_{\dot{\gamma}(t_s), \mathfrak d_s} \circ \cdots \circ \mathfrak p_{\dot{\gamma}(t_1), \mathfrak d_1} \in \operatorname{Aut}(\Bbbk[\mndP]/I)\,,
  \]
  where $\dot{\gamma}(t)$ denotes the tangent vector $\frac{d\gamma(t)}{dt}$. The order of $\mathfrak p_{\dot{\gamma}(t_i), \mathfrak d_i}$ and $\mathfrak p_{\dot{\gamma}(t_j), \mathfrak d_j}$ in the product when $t_i = t_j$ is irrelevant because they commute in $\operatorname{Aut}(\Bbbk[\mndP]/I)$.

  A path $\gamma$ is \emph{regular} with respect to a scattering diagram $\mathfrak D$ over $\widehat{\Bbbk[\mndP]}$ if it is regular with respect to $\mathfrak D_k$ for each $k\geq 0$. The \emph{path-ordered product} of $\gamma$ is defined to be
  \[
    \mathfrak p_{\gamma, \mathfrak D} \coloneqq \lim_{\longleftarrow} \mathfrak p_{\gamma, \mathfrak D_k} \in \operatorname{Aut}(\widehat {\Bbbk[\mndP]})\,,
  \]
  where each $\mathfrak p_{\gamma, \mathfrak D_k}$ is considered to be in $\operatorname{Aut}(\Bbbk[\mndP]/\mathfrak m^{k+1})$.

\begin{defn}\label{def: equiv sd rk 2}
  Two scattering diagrams $\mathfrak D$ and $\mathfrak D'$ (both over $\Bbbk[\mndP]/I$ or $\widehat{\Bbbk[\mndP]}$) are said to be \emph{equivalent} and denoted as $\mathfrak D \equiv \mathfrak D'$ if 
  \[
    \mathfrak p_{\gamma, \mathfrak D} = \mathfrak p_{\gamma, \mathfrak D'}
  \]
  for any path $\gamma$ that is regular to both.
\end{defn}

In view of the above equivalence, it is convenient to consider a ``wall'' with support $S_\frakd = b + \mathbb R_{\geq 0}m_\frakd$ and wall-function $f_\frakd$. It does not qualify either as a line or a ray in \Cref{def: ray and line}, but is equivalent to the union of a line and a ray
\begin{equation}\label{eq: incoming part of line}
    \{(b + \mathbb R m_\frakd, f_\frakd)\} \cup \{(b + \mathbb R_{\leq 0}m_\frakd, f_\frakd^{-1}).
\end{equation}

\begin{defn}\label{def: consistent sd rk 2}
  A scattering diagram $\mathfrak D$ is called \emph{consistent} if the path-ordered product $\mathfrak p_{\gamma, \mathfrak D}$ is the identity for any regular path $\gamma$ that is a closed loop, or, equivalently, $\mathfrak p_{\gamma, \mathfrak D}$ only depends on the endpoints $\gamma(0)$ and $\gamma(1)$.
\end{defn}

\begin{rem}\label{rem: D1 is consistent}
  It is straightforward to see that $\frakD$ is consistent if and only if each $\frakD_k$ is consistent. By definition, $\frakD_0$ is empty. If $\frakD_1$ only consists of lines, then it is consistent because wall-crossing automorphisms in $\operatorname{Aut}(\Bbbk[\mndP]/\frakm^2)$ commute.
\end{rem}

The following theorem is due to Kontsevich and Soibelman \cite{KS}. The version adopted below can be found in \cite[Theorem 6.38]{G11} where an order-by-order algorithmic construction is given; see also \cite[Theorem 1.4]{GPS}.

\begin{thm}\label{thm: rank 2 consistent}
  Let $\mathfrak D$ be a scattering diagram (over $\Bbbk[\mndP]/I$ or $\widehat{\Bbbk[\mndP]}$). There always exists a consistent scattering diagram $\Scat(\mathfrak D)$ unique up to equivalence such that $\Scat(\mathfrak D)\cut \mathfrak D$ consists only of rays.
\end{thm}

The above theorem implies that $\Scat(\frakD)$ only depends on the lines in $\frakD$ because for any ray in $\frakD$ one can always add another ray with the exact same support but with the inverse wall-function to cancel out their contribution to any path-ordered product.

\Cref{thm: rank 2 consistent} is functorial in the following sense.

In the complete case, suppose that we have $\varphi \colon \widehat{\Bbbk[\mndP_1]}\rightarrow \widehat{\Bbbk[\mndP_2]}$ such that $\varphi(\widehat{\Bbbk[\mndP_1(m)]}) \subseteq \widehat{\Bbbk[\mndP_2(m)]}$ for any $m\in M\cut\{0\}$.
For $\mathfrak D$ a scattering diagram over $\widehat{\Bbbk[\mndP_1]}$, we have
\begin{equation}\label{eq: functoriality complete}
  \Scat(\varphi(\mathfrak D)) \equiv \varphi(\Scat(\mathfrak D)),
\end{equation}
as scattering diagrams over $\widehat{\Bbbk[\mndP_2]}$ by \Cref{thm: rank 2 consistent}, where the operation $\varphi$ keeps the support of a wall and applies the map $\varphi$ to the attached wall-function.

\subsection{Generalized cluster scattering diagrams}\label{subsec: gen cluster sd rk 2}

In this subsection, we study a particular class of scattering diagrams that are relevant to generalized cluster algebras in rank 2. They are fundamental to our study of positivity in more general cases.

Let $M = \mathbb Z^2$ with the standard unit vectors $e_1 = (1, 0)$ and $e_2 = (0, 1)$. Let $\ell_1$, $\ell_2$ be positive integers. Choose the monoid 
\[
    \mndP = \mndP_{\ell_1, \ell_2} = M\oplus \mnd_1 \oplus \mnd_2
\]
with $r\colon \mndP \rightarrow M$ the projection to $M$, where $\mnd_i$ is the (multiplicative) monoid of monomials in variables $\{p_{i, k}\mid 0\leq k \leq \ell_i\}$. We set conventionally $p_{i,0} = 1$, thus $\mnd_i \cong \mathbb N^{\ell_i}$. The monoid $\mnd_i$ embeds in $\mathbb P_i \cong \Z^{\ell_i}$ the group of Laurent monomials. Denote 
\[
    \mnd = \mnd_1 \oplus \mnd_2 \quad \text{and} \quad \mathbb P = \mathbb P_1 \oplus \mathbb P_2.
\]

To express monomials in $\Bbbk[\mndP]$, let $x$ and $y$ denote respectively the monomials associated to $e_1$ and $e_2$ in $M$.

\begin{defn}\label{def: gen cluster sd rk 2}
  The \emph{generalized cluster scattering diagram} associated to the pair $(\ell_1, \ell_2)$ is the consistent scattering diagram $\Scat(\mathfrak D_\mathrm{in})$ over $\widehat{\Bbbk[\mndP]}$ (up to equivalence) where
  \[
    \mathfrak D_\mathrm{in} \coloneqq \{(\mathbb Re_1, P_1), (\mathbb Re_2, P_2)\},
  \]
  \begin{equation}\label{eq: initial function rk 2}
    P_1 = 1 + \sum_{k = 1}^{\ell_1} p_{1, k}x^k \quad \text{and} \quad P_2 = 1 + \sum_{k = 1}^{\ell_2} p_{2, k}y^k.
  \end{equation}
\end{defn}

This scattering diagram (and its close variants) has been considered many times in different contexts, for example in \cite{GPS, GP, RW, Mou, CKM}. Since it is determined by the initial wall-functions $P_1$ and $P_2$, we denote it by $\Scat(P_1, P_2)$. This notation will be further extended to $\Scat(f_i\mid i\in I)$ for a collection of wall-functions $\{f_i\mid i\in I\}$, which represent the lines (through the origin) they determine.  This specializes to the \emph{ordinary cluster scattering diagram} $\frakD_{(\ell_1,\ell_2)}$ (cf. \cite{GHKK}) when $P_1 = 1 + x^{\ell_1}$ and $P_2 = 1 + y^{\ell_2}$.  

\begin{rem}
Note that we are using the $d$-vector version of the scattering diagram rather than the $g$-vector version, see \cite{CGM} for a comparison.
\end{rem}

Following the Kontsevich--Soibelman algorithm explained in the proof of \cite[Theorem 6.38]{G11}, one can choose $\mathfrak D = \Scat(P_1, P_2)$ (in its equivalent class) so that any added ray in $\mathfrak D\cut \mathfrak D_\mathrm{in}$ has support in the strict third quadrant and is pointed at the origin, i.e., of the form $\mathbb R_{\leq 0}(a, b)$ for $(a, b)\in \mathbb Z_{>0}^2$. Further, rays with identical support can be combined into one ray by multiplying their wall-functions together. Therefore, we can assume that $\mathfrak D$ has distinct rays with non-trivial wall-functions $f_{\mathbb R_{\leq 0}(a,b)}$. It is convenient to assign degrees $\deg p_{1, k} = \deg p_{2, k} = k$, and thus a multi-degree in $\mathbb Z^2$ to any (Laurent) monomials in $\mathbb P$ so that for a monomial term $px^{ka}y^{kb}$ in $f_{\mathbb R_{\leq 0}(a,b)}$, the multi-degree of $p$ is $(ka, kb)$ by homogeneity constraints.

A first step towards understanding the structure of $\frakD$ is the following \emph{mutation} operation. Take a representative of $\frakD$ as
\[
    \{(\mathbb R e_1, P_1)\}\cup\{(\mathbb R_{\geq 0} e_2, P_2)\} \cup \{(\mathbb R_{\leq 0}(a, b), f_{\mathbb R_{\leq 0}(a, b)})\mid a\geq 0, b>0\},
\]
where $(a, b)$ is over all primitive vectors such that $f_{\mathbb R_{\leq 0}(a, b)}$ is non-trivial. Namely we break the wall $(\mathbb Re_2, P_2)$ into two halves. The positive half is a wall in the sense of \eqref{eq: incoming part of line} and the negative half is included in the third subset of the above union.

Define the piecewise linear transformation $T\colon M_\mathbb R \rightarrow M_\mathbb R$ by for $(a_1, a_2)\in M_\mathbb R$,
\begin{equation}\label{eq: pl map T}
    T(m) \coloneqq \begin{cases}
        (a_1, a_2) \quad & a_2\geq 0, \\
        (a_1-\ell_1a_2, a_2) \quad & a_2\leq 0.
    \end{cases}
\end{equation}
Define the monoid $\mndP' \coloneqq M \oplus {\mnd}'_1 \oplus \mnd_2$ with the projection $r'\colon \mndP' \rightarrow M$, where ${\mnd}'_1$ is the monoid of monomials in $\{p_{1,k}/p_{1,\ell_1}\mid 0\leq k\leq \ell_1\}$. Let $P'_1 = P_1/(p_{1,\ell_1}x^{\ell_1})\in \widehat{\Bbbk[\mndP']}$.

Define the collection
\[
    T(\frakD) \coloneqq \{(\mathbb R(-e_1), P'_1)\} \cup \{(\mathbb R_{\geq 0} e_2, P_2)\} \cup \{(\mathbb R_{\leq 0}T((a, b)), T(f_{\mathbb R_{\leq 0}(a, b)}))\mid a\geq 0, b>0\},
\]
where $T(f_{\mathbb R_{\leq 0}(a, b)})(x, y) = f_{\mathbb R_{\leq 0}(a, b)}\left(x, y/(p_{1,\ell_1}x^{\ell_1})\right)$.

\begin{thm}\label{thm: mutation of sd rk 2}
    The set $T(\frakD)$ is a scattering diagram over $\widehat{\Bbbk[\mndP']}$. Furthermore it is equivalent to $\Scat(P'_1, P_2)$ as scattering diagrams over $\widehat{\Bbbk[\mndP']}$.
\end{thm}

\begin{proof}
    In the context of cluster scattering diagrams of Gross--Hacking--Keel--Kontsevich, this result is well-known as the \emph{mutation invariance} \cite[Theorem 1.24]{GHKK}. In principle one obtains a similar proof which involves generalizing the very hard Theorem 1.28 in \emph{loc.~cit}. Here we are in rank 2 so a simpler argument is available although more complex initial wall-functions are allowed.

    It suffices to show that (1) $T(\frakD)$ is indeed a scattering diagram over $\widehat{\Bbbk[\mndP']}$ and then (2) it is consistent and contains the same collection of lines as $\Scat(P'_1, P_2)$. To show (1), we check that each pair
    \[
        \left(\mathbb R_{\leq 0}T((a,b)), T(f_{\mathbb R_{\leq 0}(a, b)})\right)
    \]
    qualifies as a wall over $\widehat{\Bbbk[\mndP']}$. Let $px^{ka}y^{kb}$ be a monomial in a non-trivial $f_{\mathbb R_{\leq 0}(a, b)}$. It gets transformed by $T$ into 
    \[
        p p_{1,\ell_1}^{-kb}x^{ka-\ell_1kb}y^{kb},
    \]
    which can be induced by the map
    \begin{equation}\label{eq: monoid map T}
        T\colon M \oplus \mathbb P \rightarrow M \oplus \mathbb P, \quad (m, p_1, p_2) \mapsto (T(m), p_1p_{1,\ell_1}^{-\deg p_2}, p_2).
    \end{equation}
    It is enough to show that the first entry of the multi-degree of $pp_{1,\ell_1}^{-kb}$ is always non-positive. This is equivalent to 
    \[
        ka-\ell_1kb\leq 0 \iff a/b\leq \ell_1,
    \]
    which follows directly from \cite[Corollary 6]{GP}. This also implies that every ray of $T(\frakD)$ lies in the fourth quadrant. Notice that as in (\ref{eq: incoming part of line}) the wall $(\mathbb R_{\geq 0}e_2, P_2)$ in $T(\frakD)$ is equivalent to
    \[
        \{(\mathbb Re_2, P_2)\} \cup \{(\mathbb R_{\leq 0}e_2, P_2^{-1})\}.
    \]
    Hence $T(\frakD)$ and $\Scat(P'_1, P_2)$ can be regarded to have the same set of lines.

    It only remains to show the consistency, that is, $\mathfrak p_{\gamma, T(\frakD)} = \operatorname{id}$ for a simple loop $\gamma$ around the origin. This can be checked analogously as in \emph{Step II} of the proof of \cite[Theorem 1.24]{GHKK}, which essentially follows from the consistency of $\frakD$.
\end{proof}

One immediate consequence of \Cref{thm: mutation of sd rk 2} is 
\[
    f_{\mathbb R_{\leq 0}(\ell_1, 1)} = 1 + \sum_{k=1}^{\ell_2} p_{1, \ell_1}^{k}p_{2,k}x^{\ell_1k}y^{k},
\]
which follows from $T(f_{\mathbb R_{\leq 0}(\ell_1, 1)}) = P_2$. 

We also know from the proof of \Cref{thm: mutation of sd rk 2} that $\frakD$ has no ray in the interior of
\[
    \operatorname{span}_{\mathbb R_{\geq 0}}\{(-1, 0), (-\ell_1, -1)\}.
\]
Such a region will be called a ``chamber'' and many more can be obtained by applying iterative mutations of scattering diagrams as follows.

Let ${\mnd}'_2$ be the monoid of monomials in $\{p_{2,k}/p_{2,\ell_2}\mid 0\leq k\leq \ell_2\}$. Let $P'_2 \coloneqq P_2/(p_{2,\ell_2}y^{\ell_2})$. Denote for $i\in \mathbb Z$,
\[  \frakD^{(i)} \coloneqq \begin{cases}
        \Scat(P_1, P_2)\\
        \Scat(P'_1, P_2)\\
        \Scat(P'_1, P'_2)\\
        \Scat(P_1, P'_2)
    \end{cases}
    \ \text{over} \quad
    \mndP^{(i)}\coloneqq \begin{cases}
        M\oplus \mnd_1 \oplus \mnd_2 \\
        M\oplus \mnd_1' \oplus \mnd_2  \\
        M\oplus \mnd_1' \oplus \mnd_2' \\
        M\oplus \mnd_1 \oplus \mnd_2'
    \end{cases}
    \  \text{for} \quad
    i = \begin{cases}
        0 \mod 4\\
        1 \mod 4\\
        2 \mod 4\\
        3 \mod 4.
    \end{cases}
\]
This periodic sequence of scattering diagrams can be seen as obtained by iteratively applying mutations in \Cref{thm: mutation of sd rk 2}. We denote by $T_{i}$ the mutation operation turning $\frakD^{(i)}$ into $\frakD^{(i+1)}$ analogous to the mutation $T$ defined earlier from $\frakD^{(0)}$ to $\frakD^{(1)}$. It involves a piecewise linear map $T_i\colon M_\mathbb R \rightarrow M_\mathbb R$ as in (\ref{eq: pl map T}) and a group homomorphism $T_i\colon M\oplus \mathbb P \rightarrow M\oplus \mathbb P$ as in (\ref{eq: monoid map T}). The composition $T_{i+3}T_{i+2}T_{i+1}T_i$ does not act on $M_\mathbb R$ by the identity map. Hence it sends a wall $\frakd\in \frakD^{(i)}$ to another one with different support, giving non-trivial implication on the structure of $\frakD^{(i)}$.

We use the above idea to generate a $\mathbb Z$-sequence of ``chambers'' in $\frakD = \frakD^{(0)}$. A connected component in the complement of the closure of a scattering diagram's support is called a \emph{chamber}. Each $\frakD^{(i)}$ has a \emph{positive chamber} $\operatorname{int}(\mathcal C^{(i)})$, the interior of the cone $\mathcal C^{(i)}$ non-negatively spanned by the \emph{incoming} halves of the two initial walls, that is, the $(i+1)$-th (modulo $4$) quadrant in $M_\mathbb R = \mathbb R^2$. Define piecewise linear maps on $M_\mathbb R$ for $i\geq 1$
\[
    T^{(i)} \coloneqq T_{i-1}\cdots T_1T_0,\quad T^{(-i)} = T_{-i}^{-1}\cdots T_{-1}^{-1}, \quad \text{and} \quad T^{(0)} = \mathrm{id}.
\]
Define for $i\in \mathbb Z$ cones
\[
    \mathcal C_i \coloneqq \left(T^{(i)}\right)^{-1}\left(\mathcal C^{(i)}\right) \subseteq M_\mathbb R.
\]
Namely, we pull back $\operatorname{int}(\mathcal C^{(i)})$ via the piecewise linear transform $T^{(i)}$ to obtain a chamber $\operatorname{int}(\mathcal C_i)$ in $\frakD$. Every $\mathcal C_i$ and $\mathcal C_{i+1}$ intersect at their common boundary, which is either the support of a ray or a half line in $\frakD$. The associated wall-function can also be obtained by pulling back from the initial wall-functions in $\frakD^{(i)}$. The cones $\{\mathcal C_i\mid i\in \mathbb Z\}$ form a discrete part of $\frakD$, which is referred to as the \emph{cluster complex} of $\frakD$ where $\mathcal C_i$ is called a \emph{cluster cone} while its interior is called a \emph{cluster chamber}. The map $T^{(4)}\colon M_\mathbb R\rightarrow M_\mathbb R$ acts on the cluster complex such that $T^{(4)}(\mathcal C_{i+4}) = \mathcal C_i$.

When $\ell_1\ell_2<4$, there are only finitely many $\mathcal C_i$ and their union is $M_\mathbb R$. Thus there are finitely many rays in $\mathfrak D\cut \mathfrak D_\mathrm{in}$. We depict in \Cref{fig: G2} the case $\ell_1 = 3$ and $\ell_2 = 1$. All the remaining cases can be obtained by specializing certain $p_{i,j}$ to zero (by the functoriality) and possibly switching $\ell_1$ and $\ell_2$. In \Cref{fig: G2}, the added rays are
  \begin{align*}
    & \left( \mathbb R_{\leq 0}(3, 1),\ 1 + p_{1,3}p_{2,1} x^3y \right), \\
    & \left( \mathbb R_{\leq 0}(2, 1),\ 1 + p_{1,2}p_{2,1} x^2y + p_{1,1}p_{1,3}p_{2,1}^2 x^4 y^2 + p_{1,3}^2p_{2,1}^3 x^6y^3 \right), \\
    & \left( \mathbb R_{\leq 0}(3, 2),\ 1 + p_{1,3}p_{2,1}^2 x^3y^2 \right), \\
    & \left( \mathbb R_{\leq 0}(1, 1),\ 1 + p_{1,1}p_{2,1} xy + p_{1,2}p_{2,1}^2 x^2y^2 + p_{1,3}p_{2,1}^3 x^3 y^3 \right).
  \end{align*}

\begin{figure}
\begin{tikzpicture}[scale=.3]

\draw[] (-10,0)--(10,0);
\draw[] (0,-10)--(0,10);
\draw[] (-10, -10) -- (0, 0);
\draw[] (-10, -5) -- (0, 0);
\draw[] (-10, -20/3) -- (0, 0);
\draw[] (-10, -10/3) -- (0, 0);

\end{tikzpicture}

\caption{$\Scat(P_1, P_2)$, where $P_1 = 1 + p_{1,1}x + p_{1,2}x^2 + p_{1,3}x^3$ and $P_2 = 1 + p_{2,1}y$.}
\label{fig: G2}
\end{figure}
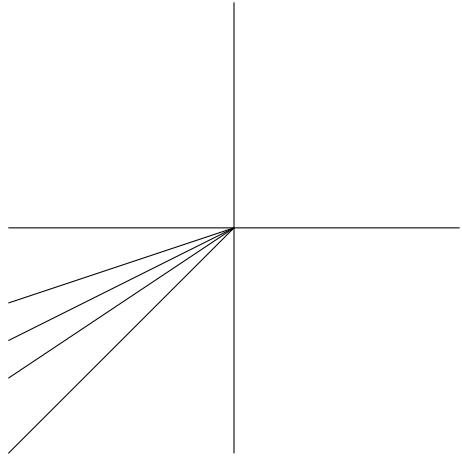

When $\ell_1\ell_2\geq 4$, the structure of the cluster complex is summarized in \Cref{thm: cluster complex struct rk 2}, for which we need some preparation. Define vectors $\dvect_n = (d_{1;n}, d_{2;n}) \in M = \Z^2$ by setting initially
\[
    \dvect_0 = e_2, \quad \dvect_1 = -e_1, \quad \dvect_2 = -e_2, \quad \dvect_3 = e_1,
\]
and for $n\neq 1, 2$ 
\begin{equation}\label{eq: recursion d vect}
    \dvect_{n+1} + \dvect_{n-1} = \begin{cases}
        \ell_1 \dvect_n \quad & \text{when $n$ is odd},\\
        \ell_2 \dvect_n \quad & \text{when $n$ is even}.
    \end{cases}
\end{equation}

The following properties of $d_n$ are easily obtained by induction.

\begin{lem}\label{lem: properties of dn}
    We have \emph{(1)} $\dvect_n\in \Z_{>0}^2$ for $n\neq 0, 1, 2, 3$; \emph{(2)} $\{\dvect_n, \dvect_{n+1}\}$ is a basis of $M$; \emph{(3)} for $n\neq 0, 1, 2, 3$,
    \[
        \begin{cases}
            d_{1;n}/\ell_1,\ (d_{2;n}-1)/\ell_2\in \N, \quad & n\equiv 0 \mod 4,\\
            (d_{1;n}+1)/\ell_1,\ d_{2;n}/\ell_2\in \N, \quad & n\equiv 1 \mod 4,\\
            d_{1;n}/\ell_1,\ (d_{2;n}+1)/\ell_2\in \N, \quad & n\equiv 2 \mod 4,\\
            (d_{1;n}-1)/\ell_1,\ d_{2;n}/\ell_2\in \N, \quad & n\equiv 3 \mod 4.
        \end{cases}
    \]
\end{lem}

We define 
\begin{equation}\label{overlineP12}
  \overline{P}_1(z):=\frac{1}{p_{1,\ell_1}} z^{\ell_1}P_1(z^{-1})\quad \text{and} \quad \overline{P}_2(z):=\frac{1}{p_{2,\ell_2}} z^{\ell_2}P_2(z^{-1}).
\end{equation}
Notice that they both have constant term $1$. For $(u,v)\in \Z^2$, let $W((u,v))=p_{1,\ell_1}^{u/\ell_1} p_{2,\ell_2}^{v/\ell_2}$.

\begin{thm}\label{thm: cluster complex struct rk 2}
  Outside the closed cone ``Badlands''
  \[
    C_{\ell_1, \ell_2} \coloneqq \operatorname{span}_{\mathbb R_{\leq 0}}\left\{\left(2\ell_1, \ell_1 \ell_2 + \sqrt{\ell_1^2\ell_2^2 - 4\ell_1 \ell_2} \right), \left(\ell_1 \ell_2 + \sqrt{\ell_1^2\ell_2^2 - 4\ell_1 \ell_2}, 2\ell_2 \right)\right\},
  \]
  the rays in $\mathfrak D\cut \mathfrak D_\mathrm{in}$ are precisely
  \[
    \left\{\left(\R_{\geq 0}d_n, \zeta_n(x^{d_{1;n}}y^{d_{2;n}}) \right) \,\middle\vert\, n\in \Z\cut\{0,1,2,3\} \right\}
  \]
  where for $n\neq 1, 2$,
  \begin{equation*}
      \zeta_n(z) \coloneqq \begin{cases}
          P_2(W(\dvect_n-e_2)z) \quad & n \equiv 0 \mod 4, \\
          \overline P_1(W(\dvect_n + e_1)z) \quad & n \equiv 1 \mod 4, \\
          \overline P_2(W(\dvect_n + e_2)z) \quad & n \equiv 2 \mod 4, \\
          P_1(W(\dvect_n-e_1)z) \quad & n \equiv 3 \mod 4.
      \end{cases}
  \end{equation*}
\end{thm}

\begin{proof}
    When $P_1$ and $P_2$ are binomials, this structure is well-known from \cite[Example 1.15]{GHKK}; see also \cite[Section 4]{CGM}. When $P_1$ and $P_2$ are monic polynomials, this result can already be extracted from \cite[Section 5]{GP}. In the current setting, the stated description of rays is obtained by working out the cones $\mathcal C_n$, $n\in \mathbb Z$ and the wall-functions on their boundary rays. Inductively we have
    \[
        \mathcal C_n = \operatorname{span}_{\R_{\leq 0}}\{\dvect_{n+1}, \dvect_{n+2}\}.  
    \]
    The attached wall-functions on $\R_{\leq 0}\dvect_n$ are also obtained inductively from iterative mutations in \Cref{thm: mutation of sd rk 2}. Notice that there is no ray in the interior of $\mathcal C_n$. The limits of $\R_{\leq 0}\dvect_n$ when $n\rightarrow +\infty$ and $n\rightarrow - \infty$ are respectively the two boundaries of $C_{\ell_1, \ell_2}$ (which degenerates to a single ray when $\ell_1\ell_2 = 4$). Therefore we have exhausted all added rays outside $C_{\ell_1, \ell_2}$.
\end{proof}

The next two examples are when $\ell_1\ell_2 = 4$. In these cases, the cluster complex is infinite but the cone $C_{\ell_1, \ell_2}$ is one-dimensional.

\begin{exmp}\label{exmp: affine 2 2 wall-function RW}
  Let $\ell_1 = \ell_2 = 2$. We take the initial wall-functions
  \[
    P_1 = (1 + s_1x)(1 + s_2x) \quad \text{and} \quad P_2 = (1 + t_1y)(1 + t_2y),
  \]
  which are related to the polynomials in (\ref{eq: initial function rk 2}) by the substitution
  \[
    p_{1,1} = s_1 + s_2, \quad p_{2, 1} = t_1 + t_2, \quad p_{1,2} = s_1s_2, \quad p_{2,2} = t_1t_2.
  \]
  The closed cone $C_{2, 2}$ is a single ray $\mathbb R_{\leq 0}(1, 1)$. The associated wall-function is computed by Reineke and Weist \cite{RW} to be
  \begin{align*}
    f_{\mathbb R_{\leq 0}(1, 1)} & = \frac{(1 + s_1t_1xy)(1 + s_1t_2xy)(1 + s_2t_1xy)(1 + s_2t_2xy)}{(1 - s_1s_2t_1t_2x^2y^2)^4} \\
    & = \frac{(p_{1,1}p_{2,2}xy^2 + p_{2,1} y)(p_{1,2}p_{2,1}x^2y + p_{1,1}x)}{(1 - p_{1,2}p_{2,2} x^2y^2)^4} + \frac{1}{(1 - p_{1,2}p_{2,2} x^2y^2)^2}.
  \end{align*}
  We see that in the expansion (where $x$ and $y$ are omitted), every coefficient is a polynomial of $p_{1,1}, p_{1,2}, p_{2,1}, p_{2,2}$ with nonnegative integer coefficients:
  \begin{multline}\label{eq: affine 2 2 refined}
      \sum_{k\geq 0} (k+1) p_{1,2}^kp_{2,2}^k + \left[\binom{k+3}{3} + \binom{k+2}{3}\right] p_{1,1}p_{2,1}p_{1,2}^kp_{2,2}^k \\
      + \binom{k+3}{3}p_{2,1}^2p_{1,2}^{k+1}p_{2,2}^k + \binom{k+3}{3} p_{1,1}^2p_{1,2}^{k}p_{2,2}^{k+1}.
  \end{multline}
A combinatorial interpretation of each term is given in \Cref{thm: power shadow grading formula}. The objects corresponding to $(k+1)p_{1,2}^kp_{2,2}^k$ are classified in \Cref{exmp: affine 2 2 gradings}. These are precisely the terms that appear in the ordinary cluster case.
\end{exmp}

\begin{exmp}\label{exmp: affine 4 1 wall-function RW}
  Let $\ell_2= 4$ and $\ell_1 = 1$. We take the initial wall-functions
  \[
    P_1 = \prod_{i=1}^4(1 + s_ix) = 1 + p_{1,1}x + p_{1, 2}x^2 + p_{1,3}x^3 + p_{1, 4}x^4 \quad \text{and} \quad P_2 = 1 + p_{2,1}y.
  \]
  The closed cone $C_{4, 1}$ is a single ray $\mathbb R_{\leq 0}(2,1)$. The wall-function is again computed in \cite{RW} to be
  \[
    f_{\mathbb R_{\leq 0}(2,1)} = \dfrac{\prod_{1\leq i<j\leq 4}(1 + s_is_jp_{2,1}x^2y)}{(1 - p_{1,4}p_{2,1}^2x^4y^2)^4}.
  \]
  To simplify notations we let $x = y = p_{2,1} = 1$ as they can be recovered from homogeneity constraints. We can rewrite
  \[
    \frac{\prod_{1\leq i<j \leq 4}(1 + s_is_j)}{(1 - p_{1,4})^4} = \frac{1+p_{1,2}+p_{1,4}}{(1-p_{1,4})^2} + \frac{(p_{1,1}p_{1,4} + p_{1,3})(p_{1,1} + p_{1, 3})}{(1-p_{1,4})^4}.
  \]
  The power series expansion of $f_{\mathbb R_{\leq 0}(2, 1)}$ then has nonnegative integer coefficients. Explicitly the expansion (with $x$ and $y$ omitted) is
  \begin{multline}\label{eq: affine 4 1 refined}
      \sum_{k\geq 0} (2k+1) p_{1,4}^kp_{2,1}^{2k} +(k+1)p_{1,2}p_{1,4}^kp_{2,1}^{2k+1} + \binom{k+3}{3} p_{1,1}^2p_{1,4}^{k+1}p_{2,1}^{2k+3}\\
      +\left[\binom{k+3}{3} + \binom{k+2}{3} \right] p_{1,1}p_{1,3}p_{1,4}^kp_{2,1}^{2k+2} + \binom{k+3}{3} p_{1,3}^2p_{1,4}^kp_{2,1}^{2k+3}.
  \end{multline}
  A combinatorial interpretation of each term is given in \Cref{thm: power shadow grading formula}. The objects corresponding to $(2k+1)p_{1,4}^kp_{2,1}^{2k}$ are classified in
  \Cref{exmp: affine 4 1 gradings}, which are precisely the terms that appear in the ordinary cluster case.
\end{exmp}

When $\ell_1\ell_2>4$, the structure of $\frakD$ within $C_{\ell_1,\ell_2}$ is rather elusive. One aim of this paper is to understand the rays and their wall-functions within $C_{\ell_1,\ell_2}$. \Cref{thm: power shadow grading formula} in particular provides a combinatorial formula for the wall-function of every such ray. We further discuss in \Cref{app: Badlands nonvanishing} that every ray within $C_{\ell_1,\ell_2}$ is non-trivial and give a lower bound for wall-function coefficients.

\subsection{Broken lines and theta functions}\label{subsec: bl and theta}

Let $\mathfrak D$ be a scattering diagram over $\widehat{\Bbbk[\mndP]}$ as in \Cref{subsec: sd in rk 2}.

\begin{defn}
   Let $m_0 \in \mndP$ and $Q \in M_\R \cut \Supp(\frakD)$. A \emph{broken line} $\bline$ for $m_0$ with endpoint $Q$ is a piecewise linear continuous map 
   \[
    \bline\colon (-\infty,0] \to M_\R \cut \operatorname{Sing}(\mathfrak D)
   \]
   with values $-\infty = \tau_0 < \tau_1 < \cdots < \tau_{\ell+1} = 0$ and an associated monomial $c_i x^{m_i}$ (with $c_i\in \Bbbk$ and $m_i\in \mndP$) for each domain of linearity $(\tau_i, \tau_{i+1})$, $i = 0, \dots, \ell$ such that 
  \begin{enumerate}
    \item $\bline(0) = Q$ and $c_0 = 1$;
    \item $\dot{\bline}(\tau) = -r(m_i)$ for any $\tau \in (\tau_{i},\tau_{i+1})$ for each $i = 0, \dots, \ell$;
    \item $\bline$ transversally crosses (the support of) some wall at $\tau_i$ for $i = 1, \dots, \ell$, that is, neither $-r(m_{i-1})$ nor $-r(m_{i})$ is tangent to the wall.
    \item for $i = 0, \dots, \ell -1$, $c_{i+1}x^{m_{i+1}}$ is a monomial term (other than $c_ix^{m_i}$) of the power series
    \[
      c_ix^{m_i} \prod_{\{\mathfrak d\in \mathfrak D\mid \bline(\tau_i)\in S_\mathfrak d\}} f_\mathfrak d^{n_i\cdot r(m_i)},
    \]
    where $n_i$ is primitive in $N$ and orthogonal to any $\frakd$ with $\bline(\tau_i)\in S_\mathfrak d$ such that $n_i \cdot r(m_i)>0$.
  \end{enumerate}
\end{defn}

When $r(m_0) = 0$, there is only the ``static'' broken line at $Q$ for $m_0$ and any endpoint $Q$.

\begin{defn}\label{def: theta function rank 2}
  The \emph{theta function} associated to $m_0$ and $Q$ is
  \[
    \vartheta_{Q, m_0} \coloneqq \sum_\bline \operatorname{Mono}(\beta)
  \]
  where the (formal) sum is over all broken lines for $m_0$ with endpoint $Q$ and $\operatorname{Mono}(\beta) = c_\ell x^{m_\ell}$ is the last monomial of $\beta$.
\end{defn}

\begin{rem}
    The theta function $\vartheta_{Q, m_0}$ can be viewed as an element in $x^{m_0}\widehat{\Bbbk[\mndP]}$. In fact, broken lines and theta functions can be defined for $\frakD_k$ over $\Bbbk[\mndP]/\mathfrak m^{k+1}$. Then $\vartheta_{Q, m_0}$ is the limit of theta functions in $\frakD_k$ where each one is necessarily a finite sum in $x^{m_0}\Bbbk[\mndP]/\mathfrak m^{k+1}$. More detail is given in the proof of \cite[Proposition 3.4]{GHKK}.
\end{rem}

\begin{thm}[{\cite{CPS}, \cite[Theorem 3.5]{GHKK}}]\label{thm: CPS consistency}
  Let $\mathfrak D$ be a consistent scattering diagram, and let $m_0\in \mndP$ and $Q, Q'\in M_\mathbb R\cut \operatorname{Supp}(\mathfrak D)$ be two points. Suppose further that the coordinates of $Q$ are linearly independent over $\mathbb Q$, and the same is true for $Q'$. Then for any regular path $\gamma$ with endpoints $\gamma(0) = Q$ and $\gamma(1) = Q'$, we have
  \[
    \vartheta_{Q', m_0} = \mathfrak p_{\gamma, \mathfrak D}(\vartheta_{Q, m_0}).
  \]
\end{thm}

Let $\frakD$ be $\Scat(P_1, P_2)$ as in \Cref{subsec: gen cluster sd rk 2}. Define the \emph{weight} of a broken line $\beta$ to be
\begin{equation}\label{eq: weight of broken line}
    c(\beta) \coloneqq c_\ell x^{m_\ell}/x^{r(m_\ell)},
\end{equation}
a monomial in $\Bbbk[\mnd]$. Let $m(\beta) \coloneqq r(m_\ell)$. Thus $c(\beta)x^{m(\beta)} = \operatorname{Mono}(\beta)$.

\begin{cor}\label{cor: CPS consistency non strict}
    Let $m_0\in \mndP$. For any $m_\mathrm{fin}\in M$, the sum
    \[
        \sum_{m(\beta) = m_\mathrm{fin}} c(\beta)x^{m(\beta)}
    \]
    over all broken lines $\beta$ for $m_0$ with endpoint $Q$ such that $m(\beta) = m_\mathrm{fin}$ is independent of $Q$ chosen generally in the same $\mathcal C_n$ for $n\in \mathbb Z$. If $m_\mathrm{fin}$ is tangent to the ray $\mathcal C_n\cap \mathcal C_{n+1}$, then the sum is independent of $Q$ chosen generally in $\mathcal C_n\cup \mathcal C_{n+1}$.
\end{cor}

\begin{proof}
    The first half is a direct corollary of \Cref{thm: CPS consistency}. To see the second half, we choose general $Q\in \mathcal C_n$, $Q'\in \mathcal C_{n+1}$, and a regular path $\gamma$ with $\gamma(0) = Q$ and $\gamma(1) = Q'$. The path-ordered product $\mathfrak p_{\gamma, \frakD}$ is a single wall-crossing automorphism associated to the wall $\frakd$ whose support is $\mathcal C_n\cap \mathcal C_{n+1}$. Then any monomial $x^{m'}$ in the expansion of $\mathfrak p_{\gamma, \frakD}(x^m)$ must have that $r(m') = r(m) + km_\frakd$ for some $k\in \mathbb N$. In particular, $r(m')$ is tangent to $\mathcal C_n\cap \mathcal C_{n+1}$ if and only if  $r(m)$ is. Then the result follows by comparing the monomials containing $x^m$ with $r(m) = m_\mathrm{fin}$ on the two sides of $\vartheta_{Q', m_0} = \mathfrak p_{\gamma, \frakD}(\vartheta_{Q, m_0})$.
\end{proof}

\begin{prop}[cf. {\cite[Proposition 3.6]{GHKK}}]\label{prop: mutation broken line}
    The mutation $T$ in \Cref{thm: mutation of sd rk 2} defines a one-to-one correspondence $\beta\mapsto T(\beta)$ between broken lines for $m_0\in \mndP$ with endpoint $Q = (q_1, q_2)\in M_\mathbb R$ for $\frakD = \Scat(P_1, P_2)$ and broken lines for $T(m_0)$ with endpoint $T(Q)$ for $\frakD' = \Scat(P_1', P_2)$ such that
    \[
        \operatorname{Mono}(T(\beta)) = \begin{dcases}
            T(\operatorname{Mono}(\beta)) \quad & q_2\leq 0, \\
            \operatorname{Mono}(\beta) \quad & q_2\geq 0.
        \end{dcases}
    \]
\end{prop}

\begin{proof}
    One constructs the bijection in the same way as in the proof of \cite[Proposition 3.6]{GHKK}.
\end{proof}

\begin{prop}\label{prop: theta function as monomial}
  Let $m_0\in \mathcal C_n\cap M$ for some $n\in \mathbb Z$ and $Q$ be in the interior of $\mathcal C_n$. Then
  \[
    \vartheta_{Q, m_0} = x^{m_0}.
  \]
\end{prop}

\begin{proof}
    There is always the straight broken line without bending which contributes to $x^{m_0}$ in $\vartheta_{Q, m_0}$. We need to show that this is the only broken line.

    When $n = 0$, $\mathcal C_0$ is the first quadrant. For a broken line $\beta$, if there is $t\in (-\infty, 0)$ such that $\dot{\beta}(t)\in \mathbb Z_{\leq 0}^2$ and $\beta(t)\notin \mathbb R_{\geq 0}^2$, then $\beta$ can never end at $Q$ in the interior of $\mathcal C_0$. This is because $\beta$ can further bend only at the negative $x$-axis, $y$-axis, or the rays in the third quadrant. So $\dot{\beta}(t)$ (whenever defined) always stays in $\mathbb Z_{\leq 0}^2$ and thus $\beta$ can never reach the first quadrant. Therefore we have proven the case $n = 0$.

    For any other $n\in \mathbb Z$, by applying \Cref{prop: mutation broken line} to a sequence of mutations we have a bijection between the broken lines for $m_0$ with endpoint $Q$ for $\frakD$ and the broken lines for $T^{(n)}(m_0)$ with endpoint $T^{(n)}(Q)$ for $\frakD^{(n)}$. Both $T^{(n)}(Q)$ and $T^{(n)}(m_0)$ are in the positive chamber of $\frakD^{(n)}$. There is thus only one straight broken line by the case $n=0$. It is pulled back to be the straight broken line $\beta$ in $\frakD$ with $\operatorname{Mono}(\beta) = x^{m_0}$. 
\end{proof}

Let $Q_+$ denote a general point in the first quadrant $\mathbb R_{\geq 0}^2$. It follows directly from \Cref{prop: theta function as monomial} and \Cref{thm: CPS consistency} that for any $m$ in $\mathcal C_n\cap M$,
\begin{equation}\label{eq: theta function path product}
  \vartheta_{Q_+, m} = \mathfrak p_{\gamma, \mathfrak D}(x^m),
\end{equation}
where $\gamma$ is a regular path from $Q\in \operatorname{int}(\mathcal C_n)$ to $Q_+$.

We now consider a sequence of theta functions known as \emph{cluster variables}. This is an unorthodox way to introduce cluster variables, but it will soon be clear that they satisfy (generalized) cluster exchange relations \cite{FZ, CS}. Denote $\vartheta_m = \vartheta_{Q_+, m}$.

\begin{defn}\label{defn: cluster variables}
    The \emph{cluster variables} $\{x_n\mid n\in \mathbb Z\}$ are defined to be
    \[
        x_n \coloneqq \vartheta_{-\dvect_n}.
    \]
\end{defn}

\begin{lem}\label{lem: exchange relation rk 2}
    The cluster variables satisfy the following cluster exchange relations
    \begin{enumerate}
        \item[\emph{(1)}] $x_0x_2 = P_1(x_1)$,
        \item[\emph{(2)}] $x_1x_3 = P_2(x_2)$,
        \item[\emph{(3)}] $x_{n-1}x_{n+1} = \bar \zeta_{n}(x_n)$ for $n\neq 1, 2$,
    \end{enumerate}
    where $\bar \zeta_n(z) \coloneqq z^{\deg(\zeta_n)}\zeta_n(1/z)$.
\end{lem}

\begin{proof}
  We first check (1) and (2). By \Cref{prop: theta function as monomial}, we have $x_1 = x$ and $x_2 = y$. Notice $x_0 = \vartheta_{(0, -1)}$ and $x_3 = \vartheta_{(-1, 0)}$. By \eqref{eq: theta function path product}, we have $x_0 = x_2^{-1}P_1(x_1)$ and $x_3 = x_1^{-1}P_2(x_2)$.

  For any $n\in \Z$, let $\gamma_n$ be a regular path going from the interior of $\mathcal C_{n}$ to the first quadrant $\mathcal C_0$. For some $n\neq 1, 2$, let $\gamma$ be a regular path going from $\mathcal C_{n-1}$ to $\mathcal C_{n-2}$ by crossing the wall $(\R_{\leq 0}d_n, \zeta_n)$. Then we have
  \begin{align*}
    \frakp_{\gamma, \frakD}(x^{-d_{1;n+1}}y^{-d_{2;n+1}}) & = x^{-d_{1;n+1}}y^{-d_{2;n+1}} \zeta_n(x^{d_{1;n}}y^{d_{2;n}}) \\
    & = x^{\deg (\zeta_n) d_{1; n} - d_{1; n+1}} y^{\deg (\zeta_n) d_{2; n}- d_{2; n+1}} \bar \zeta_n(x^{-d_{1; n}} y^{-d_{2;n}}) \\
    & = x^{d_{1;n-1}} y^{d_{2;n-1}} \bar \zeta_n(x^{-d_{1; n}} y^{-d_{2;n}}).
  \end{align*}
  Applying $\frakp_{\gamma_{n-2}, \frakD}$ to both sides of the above equation, we obtain
  \[
    \frakp_{\gamma_{n-1}, \frakD}\left(x^{-d_{1;n+1}}y^{-d_{2;n+1}}\right) = \frakp_{\gamma_{n-2}, \frakD}\left(x^{d_{1;n-1}}y^{d_{2;n-1}}\right) \bar\zeta_n\left(\frakp_{\gamma_{n-2},\frakD}\left(x^{-d_{1;n}}y^{-d_{2;n}}\right)\right).
  \]
  Now by \eqref{eq: theta function path product}, we have for $n\geq 3$,
  \[
    x_{n+1} = x_{n-1}^{-1} \bar \zeta_n(x_n).\qedhere
  \]
\end{proof}

\subsection{Theta functions are universal Laurent}\label{subsec: theta universal Laurent}

Recall the notations in \Cref{subsec: gen cluster sd rk 2}. In that situation the monoid $\mndP$ decomposes as $M \oplus \mnd$, where $\mnd = \mnd_1\oplus \mnd_2$ embeds in $\mathbb P$ the group of Laurent monomials in $\{p_{i,j}\mid i=1, 2, 1\leq j \leq \ell_i\}$. Elements in $\rngB$ are regarded as coefficients.

Let $\mathcal F$ denote the field of rational functions $\operatorname{Frac}(\Bbbk[\mndP])$. The generalized \emph{cluster algebra} $\mathcal A = \mathcal A(P_1, P_2)$ is defined to be the $\mathbb Z[\mathbb P]$-subalgebra in $\mathcal F$ generated by all cluster variables $\{x_n\mid n\in \mathbb Z\}$. The generalized \emph{upper cluster algebra} is defined to be 
\begin{equation}\label{eq: def upper algebra}
  \mathcal U = \mathcal U(P_1, P_2) \coloneqq \bigcap_{n\in \mathbb Z} \rngB[x_n^{\pm 1}, x_{n+1}^{\pm 1}] \subseteq \mathcal F.
\end{equation}
An element in $\mathcal U$ is called \emph{universal Laurent} as it belongs to every $\rngB[x_n^{\pm 1}, x_{n+1}^{\pm 1}]$.

\begin{prop}\label{prop: theta universal Laurent}
    For any $m\in M$, the theta function $\vartheta_m$ belongs to ${\mathcal U}$. Hence $\mathcal A\subseteq \mathcal U$.
\end{prop}

\begin{proof}
    This is a slight generalization from \cite[Theorem 3.9(b)(c)]{CGM} to our setting. For each $\mathcal C_n$, choose a general point $Q_n$ in its interior. If $m\in \mathbb N^2$, we have $\vartheta_m = \vartheta_{Q_0, m} = x^m$ by \Cref{prop: theta function as monomial}. If $m\notin \mathbb N^2$, for any broken line $\beta$ for $m$ with endpoint $Q_0$, the vector $m_\beta\in M = \mathbb Z^2$ has at least one negative coordinate. Then there are finitely many walls where any such $\beta$ can bend. Moreover only finitely many terms on these wall-functions are relevant. Then $\vartheta_m$ can be equivalently computed in $\frakD_k$ for a large $k$. Therefore, $\vartheta_m$ is a Laurent monomials. This shows $\vartheta_m\in \rngB[x_1^{\pm 1}, x_2^{\pm 1}]$ for any $m\in M$.

    For any $n\in \Z$, we first show that $\vartheta_{Q_n, m}$ is in $\rngB[x_1^{\pm 1}, x_2^{\pm 1}]$. By \Cref{prop: mutation broken line}, there is a piecewise linear transformation $T^{(n)}\colon M_\mathbb R\rightarrow M_\mathbb R$ that takes $\mathcal C_n$ to the positive chamber of $\frakD^{(n)}$, which induces a one-to-one correspondence between broken lines in $\frakD$ and those in $\frakD^{(n)}$. Since $T^{(n)}(Q_n)$ lies in the positive chamber of $\frakD^{(n)}$, there are finitely many broken lines by the first paragraph. Thus there are too only finitely many broken lines for $\vartheta_{Q_n, m}$, which means $\vartheta_{Q_n, m}\in \rngB[x_1^{\pm 1}, x_2^{\pm 1}]$.

    Each cluster variable $x_n$ corresponds to the ray $\mathcal C_{n-2}\cap \mathcal C_{n-1}$ in the cluster complex. The primitive integral generator of this ray is $-\dvect_n$ (see \Cref{defn: cluster variables}). Any pair $\{\dvect_n, \dvect_{n+1}\}$ is a basis of $M$ (\Cref{lem: properties of dn}). Then for any $m\in M$ we have
    \[
        \vartheta_{Q_n, m}\in \rngB[x^{\pm \dvect_{n+1}}, x^{\pm \dvect_{n+2}}].
    \]
    Therefore for a regular path $\gamma$ from $Q_n$ to $Q_0$, we have by \Cref{thm: CPS consistency} that
    \[
        \vartheta_{m} = \mathfrak p_{\gamma, \mathfrak D}(\vartheta_{Q_n, m})\in \rngB[x_{n+1}^{\pm 1}, x_{n+2}^{\pm 1}]. \qedhere
    \]
\end{proof}

\begin{defn}
    The set of theta functions $\{\vartheta_m\mid m\in M\}$ is called the \emph{generalized theta basis} of $\mathcal U$.
\end{defn}

\begin{rem}
    One can prove that theta functions $\{\vartheta_m \mid m\in M\}$ form a $\rngB$-basis of $\mathcal U$ analogously as in \cite{GHKK}. This also follows from the fact that theta functions equal so-called \emph{greedy elements}, proven in \Cref{sec: greedy = theta}, and that the greedy elements form a basis of $\mathcal U$.
\end{rem}

\section{Greedy bases for rank-2 generalized cluster algebras}\label{sec: greedy bases}

Following Lee--Li--Zelevinsky \cite{LLZ} in the cluster algebra setting and Rupel \cite{Rupgengreed} in the palindromic generalized cluster algebra setting, we define the generalized greedy basis for rank-$2$ generalized cluster algebras. We will then, following these previous works, show that the coefficients can be realized as counts of combinatorial objects called \emph{compatible gradings}.

\subsection{Generalized greedy elements}\label{subsec: greedy elements}

Let $\calT_k = \rngB[x_k^{\pm 1}, x_{k+1}^{\pm 1}]$ for $k\in \Z$. Therefore $\mathcal U = \bigcap_{k\in \mathbb Z} \calT_k$ from \eqref{eq: def upper algebra}.

\begin{defn}\label{def: pointed element}
  An element $x \in \calT_k$ is \emph{pointed} at $(\aone,\atwo) \in \Z^2$ in the $k^{\text{th}}$ cluster if it can be expressed in the form 
  $$x = x_k ^{-\aone} x_{k+1}^{-\atwo } \sum_{p,q \geq 0} c(p,q) x_k^p x_{k+1}^q$$
  for some $c(p,q) \in \rngB$ and $c(0,0) = 1$.
\end{defn}

An element pointed in the initial cluster $\calT_1$ is simply called \emph{pointed}.

\begin{defn}\label{defn: pi zeta}
  For a sequence of variables $p_\bullet = (p_{k})_{k=1}^{\ell}$ where $\ell\in \Z_{\geq 1}\cup \{\infty\}$ with $\deg p_{k} = k$, we write for any $m\geq 0$
  \[
    \left(1 + \sum_{k\geq 1}p_k x^k \right)^m = \sum_{n\geq 0} \pi_{m, n}(p_\bullet) x^n,
  \]
  defining $\pi_{m, n}(p_\bullet)$ as a homogeneous polynomial in $(p_{k})_{k}$ of degree $n$. We write for $m\geq 0$
  \[
    \left(1 + \sum_{k\geq 1}p_k x^k \right)^{-m} = \sum_{n\geq 0} \varsigma_{m, n}(p_\bullet) x^n,
  \]
   defining $\varsigma_{m, n}(p_\bullet)$ as a homogeneous polynomial in $(p_{k})_k$ of degree $n$.
\end{defn}

We take the conventions that $\pi_{m, 0}(p_\bullet) = \varsigma_{m, 0}(p_\bullet) = 1$ for any $m\in \mathbb Z$ and that $\pi_{m,n}(p_\bullet) = \varsigma_{m,n}(p_\bullet) = 0$ when $m<0$ and $n\neq 0$.

\begin{lem}
  For any $s\geq 1$
  \[
    \sum_{k = 0}^s \pi_{m, s-k}(p_\bullet) \varsigma_{m, k}(p_\bullet) = 0.
  \]
\end{lem}

\begin{proof}
    By \Cref{defn: pi zeta}, we have $\left(\sum_{n\geq 0} \pi_{m, n}(p_\bullet) x^n\right)\left(\sum_{n\geq 0} \varsigma_{m, n}(p_\bullet) x^n\right) = 1$.  The statement then follows from isolating the $s^{\text{th}}$ coefficient of this product.
\end{proof}

The above lemma implies
\begin{equation}\label{eq: 1 dim recursion}
  \pi_{m, s}(p_\bullet) = \sum_{k = 1}^s - \pi_{m, s-k}(p_\bullet) \varsigma_{m, k}(p_\bullet).
\end{equation}
This can be seen as a recursive definition of $\pi_{m, s}(p_\bullet)$ for $m, s\geq 1$. The fact that any $\pi_{m, s}(p_\bullet)$ has nonnegative integer coefficients is less obvious from this definition. Next we consider a two-dimensional recursion generalizing (\ref{eq: 1 dim recursion}).

For an element $f$ in any polynomial ring $Y=\mathbb Z[y_1,...,y_n]$ where $y_1,...,y_n$ are algebraically independent, and for $e=(e_1,...,e_n)\in \N^n$, let $[y^e]f=[y_1^{e_1}\cdots y_n^{e_n}]f$ be the coefficient of $y_1^{e_1}\cdots y_n^{e_n}$ in $f$.   For a pair of elements $f$ and $g$ in  $Y$, we define
\[
    \operatorname{T-max}(f,g) \coloneqq \sum_{e\in \N^n} \max([y^e]f, [y^e]g)y^e.
\]

Consider two sequences of variables $p_{i,\bullet} = (p_{i,k})_{k=1}^{\ell_i}$ for $i = 1, 2$ and $\ell_i \in \Z_{\geq 1} \cup \{\infty\}$.

\begin{defn}\label{def: cpq}
    Fix $(\aone, \atwo )\in \mathbb Z^2$. Define $c(p, q)\in \mathbb Z[p_{1,\bullet}\cup p_{2,\bullet}]$ for any $(p, q)\in \mathbb N^2$ by the base case $c(0, 0) = 1$ and the recursion
  \[
    c(p, q) \coloneqq \operatorname{T-max} \left(\sum_{k=1}^p -c(p-k, q)\varsigma_{\atwo -q, k}(p_{1, \bullet}),\ \sum_{k = 1}^q -c(p, q-k) \varsigma_{\aone-p, k}(p_{2, \bullet}) \right).
  \]
\end{defn}

\begin{rem}\label{rem: cpq easy case}
    It is not hard to determine all $c(p, q)$ unless both $d_1$ and $d_2$ are positive. In fact, we have by \eqref{eq: 1 dim recursion} that
    \begin{enumerate}
        \item when $d_1\leq 0, d_2\leq 0$, every $c(p, q) = 0$ except $c(0, 0) = 1$;
        \item when $d_1\leq 0, d_2\geq 0$, every $c(p, q) = 0$ except $c(p, 0) = \pi_{d_2,p}(p_{1,\bullet})$;
        \item when $d_1\geq 0, d_2\leq 0$, every $c(p, q) = 0$ except $c(0, q) = \pi_{d_1,q}(p_{2, \bullet})$.
    \end{enumerate}
\end{rem}

\begin{defn}\label{def: dpq}
    Define $d_+(0, 0) = d_-(0, 0) \coloneqq c(0, 0) = 1$, and for $(p, q)\neq (0, 0)$,
    \begin{align*}
        d_+(p, q) \coloneqq c(p, q) - \sum_{k=1}^p -c(p-k, q)\varsigma_{\atwo-q, k}(p_{1, \bullet}), \\
        d_-(p, q) \coloneqq c(p, q) - \sum_{k = 1}^q -c(p, q-k) \varsigma_{\aone-p, k}(p_{2, \bullet}).
    \end{align*}
\end{defn}

By definition, both $d_\pm(p, q)$ are polynomials with nonnegative integer coefficients.

\begin{lem}\label{lem: cpq by dpq}
    We have
    \begin{align*}
        c(p, q) & = d_+(p, q) + \sum_{s = 1}^p  \pi_{\atwo-q, s}(p_{1,\bullet}) d_+(p-s, q)\\
        & = d_-(p, q) + \sum_{s=1}^q \pi_{\aone-p, s}(p_{2, \bullet}) d_-(p, q-s).
    \end{align*}
\end{lem}

\begin{proof}
    We prove the first equality by induction on $p$. By the definition of $d_+(p, q)$, it is equivalent to show
    \begin{equation}\label{eq: c(p,q) and d+(p,q)}
        \sum_{k=1}^p -c(p-k, q)\varsigma_{d_2-q, k}(p_{1, \bullet}) = \sum_{s=1}^p \pi_{\atwo-q, s}(p_{1, \bullet})d_+(p-s, q).
    \end{equation}
    We only need to deal with the case when $d_2 - q\geq 1$. Assume that the first equality is true for $c(0, q),\dots, c(p-1, q)$. Then we can write the left-hand side of \eqref{eq: c(p,q) and d+(p,q)} as
    \begin{align*}
        & \sum_{k=1}^p -\varsigma_{\atwo-q, k}(p_{1,\bullet})\sum_{t=0}^{p-k}\pi_{\atwo-q, t}(p_{1, \bullet}) d_+(p-k-t, q) \\
        = &\sum_{s=1}^p d_+(p-s, q) \sum_{k=1}^s-\pi_{\atwo-q, s-k}(p_{1, \bullet}) \varsigma_{\atwo-q, k}(p_{1,\bullet}).
    \end{align*}
    Applying \eqref{eq: 1 dim recursion}, the above equals the right-hand side of \eqref{eq: c(p,q) and d+(p,q)}.

    The second equality of the statement can be proven similarly by induction on $q$.
\end{proof}

A consequence of \Cref{lem: cpq by dpq} is that each $c(p, q)$ has nonnegative integer coefficients.

\begin{defn}\label{}
    For $(\aone ,\atwo )\in \Z^2$, we define the \emph{greedy element}
    \[
        x[\aone, \atwo] \coloneqq x_1^{-\aone}x_2^{-\atwo}\sum_{(p, q)\in \N^2} c(p, q)x_1^{p}x_2^q \in x_1^{-\aone}x_2^{-\atwo}\N[p_{1,\bullet}\cup p_{2,\bullet}]\llbracket x_1, x_2\rrbracket .
    \]
\end{defn}

According to \Cref{rem: cpq easy case}, the greedy elements are clear in the following cases
\begin{equation}\label{eq: cpq easy case}
    x[\aone, \atwo] = \begin{cases}
        x_1^{-\aone}x_2^{-\atwo} \quad & \aone\leq 0, \atwo \leq 0,\\
        x_1^{-\aone}x_2^{-\atwo}\left(1 + \sum_{k\geq 1}p_{1,k} x_1^k \right)^{\aone} \quad & \aone\geq 0, \atwo \leq 0,\\
        x_1^{-\aone}x_2^{-\atwo}\left(1 + \sum_{k\geq 1}p_{2,k} x_2^k \right)^{\atwo} \quad & \aone\leq 0, \atwo \geq 0.
    \end{cases}
\end{equation}

When $p_{1, \bullet}$ and $p_{2, \bullet}$ are both finite sequences, in the cases of \eqref{eq: cpq easy case}, $x[d_1,d_2]$ is a finite sum, thus a pointed element in $\mathcal T_1$ (\Cref{def: pointed element}). When $d_1$ and $d_2$ are positive, $x[d_1, d_2]$ is also a finite sum. This fact is non-trivial from the definition and is a consequence of the combinatorial formula of $x[d_1, d_2]$ given in \Cref{thm: generalized greedy compatible gradings}, where each $c(p, q)$ is interpreted as the weighted count of combinatorial objects called \emph{compatible gradings}.

\subsection{Compatible gradings}\label{subsec: compatible grading}

\begin{defn}
 A \emph{North-East (NE) path} $\calP$ is a lattice path in $(\mathbb{Z}\times \mathbb{R})\cup(\mathbb{R}\times \mathbb{Z})\subset \mathbb{R}^2$ starting at $(0,0)$ and ending at $(\aone , \atwo )\in \N^2$, proceeding by only unit north and east steps. We also view $\mathcal P$ as the set of its unit north and east steps, where we refer to each step as an \emph{edge}.
 \begin{itemize}
     \item Given a set $C$ of edges in $\calP$, we denote the set of horizontal edges (east steps) by $C_\east$ and the set of vertical edges (north steps) by $C_\north$.  We let $|C|$ denote the number of edges in $C$.
     \item For edges $e$ and $f$ in $\calP$, let $\overrightarrow{ef}$ denote the subpath proceeding east from $e$ to $f$ (including both $e$ and $f$), continuing cyclically around $\calP$ if $e$ is to the north or east of $f$.
 \end{itemize}
\end{defn}

Note that if $e = f$, then $\overrightarrow{ef} = \{e\}$.  If $e$ is the edge adjacent to $f$ and to the northeast of $f$, then $\overrightarrow{ef} = \calP$.

Consider the rectangle $R(d_1, d_2)$ with vertices $(0,0)$, $(0,\atwo )$, $(\aone ,0)$, and $(\aone , \atwo )$ with a main diagonal from $(0,0)$ to $(\aone , \atwo )$. A \emph{Dyck path} is an NE path from $(0,0)$ to $(\aone,\atwo)$ that never passes strictly above the main diagonal of $R(d_1, d_2)$.  The Dyck paths from $(0,0)$ to $(\aone ,\atwo )$ form a partially ordered set by comparing the heights at all vertices. The \emph{maximal Dyck path} $\calP(\aone ,\atwo )$ is the maximal element under this partial order. When $\aone$ and $\atwo$ are relatively prime, the maximal Dyck path $\calP(\aone ,\atwo )$ corresponds to the lower Christoffel word of slope $\atwo/\aone$. We label the horizontal edges from left to right by $u_1,u_2,\dots,u_{\aone}$ and the vertical edges from bottom to top by $v_1,v_2,\dots,v_{\atwo}$.

\begin{exmp} In \Cref{fig: compatible grading example}, the maximal Dyck path $\calP(6,4)$ is shown in the top left and $\calP(7,4)$ is shown in the top right.
\end{exmp}

Motivated by Lee--Schiffler \cite{LS},  Lee, Li, and Zelevinsky \cite{LLZ} introduced combinatorial objects called \emph{compatible pairs} to construct the \emph{greedy basis} for rank-$2$ cluster algebras, consisting of indecomposable positive elements including the cluster monomials. Rupel \cite{Rupgengreed, Rup2} extended this construction to the setting of \emph{generalized} rank-$2$ cluster algebras by defining \emph{compatible gradings}.

A \emph{grading} $\omega: \mathcal{P} \to \N$ assigns a nonnegative integer to every edge of the NE path $\mathcal{P}$.  For any grading $\omega$ and for any set of edges $C \subseteq \mathcal{P}$, let $\displaystyle \omega(C) \coloneqq \sum_{e\in C} \omega(e)$. 

\begin{defn}
Let $\calP$ be an NE path.  A grading $\omega: \calP \rightarrow \N$ is called \emph{compatible} if for every pair of $u\in \calP_\east$ and $v\in \calP_\north$ with $\omega(u)\omega(v)>0$, there exists an edge $e$ along the subpath $\overrightarrow{uv}$ so that at least one of the following holds:
  \begin{equation}\label{0407df:comp}
   \aligned
&  e\in \calP_\north \setminus\{v\} \quad \text{and} \quad |\overrightarrow{ue}_\north|= \omega\left( \overrightarrow{ue}_\east\right);\\
 &  e\in \calP_\east \setminus\{u\} \quad \text{and} \quad |\overrightarrow{ev}_\east|=\omega\left( \overrightarrow{ev}_\north\right).
  \endaligned
  \end{equation}

\end{defn} 
\begin{exmp}\label{main_exmp}
For each $i\in\{1,2,3\}$, let $\omega_i: \calP(i+5,4)\rightarrow \N$ be the grading given by $\omega_i(u_1)=\omega_i(u_2)=2$, $\omega_i(v_3)=\omega_i(v_4)=3$, and $\omega_i(e)=0$ for every edge $e$ in $\calP(i+5,4)\setminus\{u_1,u_2,v_3,v_4\}$ (see \Cref{fig: compatible grading example} for $i = 1,2$). Then $\omega_1$ is not compatible, but $\omega_2$ and $\omega_3$ are compatible. The main difference between $\omega_1$ and $\omega_2$ is that the edge $e=u_2$ in $\calP(7,4)$ satisfies the second condition in (\ref{0407df:comp}) for $u = u_1$ and $v = v_4$, as both sides of the equation equal $6$.
\end{exmp}

\subsection{Generalized greedy element formula}\label{subsec: greedy by cg}

We now give a combinatorial interpretation of the coefficients of generalized greedy elements using compatible gradings, as presented in \Cref{thm: generalized greedy compatible gradings}. This generalizes the work of Rupel \cite[Theorem 2.15]{Rup} to the non-palindromic setting, which will be needed in our discussion on the positivity of higher-rank scattering diagrams in \Cref{sec: proof of positivity}. The proof of \Cref{thm: generalized greedy compatible gradings}, which is inspired by Rupel's work, will be given in \Cref{app: generalized greedy element formula}.

Let $\mnd$ denote the multiplicative monoid of monomials in $p_{1,\bullet}\cup p_{2,\bullet}$.

\begin{defn}\label{def: weigth poly of grading}
  The \emph{weight} of a grading $\omega:\calP(\aone ,\atwo )\rightarrow\mathbb{Z}_{\ge0}$ is the monomial
  \[
    \text{wt}(\omega) \coloneqq \prod_{i=1}^{\aone} p_{2,\omega(u_i)} \prod_{j=1}^{\atwo} p_{1,\omega(v_j)}\in \mnd\,.
  \]
\end{defn}

Fix $(d_1, d_2)\in \mathbb Z^2$. Let $\CG(d_1, d_2, p, q) = \CG(p,q)$ denote the set of all compatible gradings $\omega$ on $\mathcal P([d_1]_+, [d_2]_+)$ such that $\omega(\calP_\east) = q$ and $\omega(\calP_\north) = p$. Define the weighted sum
\[
    |\CG(p, q)| \coloneqq \sum_{\omega\in \CG(p, q)} \text{wt}(\omega) \in \mathbb N[\mnd]\,.
\]

The coefficients $c(p, q)$ of a generalized greedy element are then given as weighted sums of compatible gradings on a fixed Dyck path with a fixed total vertical and horizontal grading.

\begin{thm}\label{thm: generalized greedy compatible gradings}
For any $(\aone ,\atwo ) \in \Z^2$, the greedy element $x[\aone ,\atwo ]$ can be computed as 
\begin{equation}\label{eq: greedy expansion} x[\aone ,\atwo ] = x_1^{-\aone }x_2^{-\atwo }\sum_{\omega} \wt(\omega) x_1^{\omega(\calP_\north)}x_2^{\omega(\calP_\east)}\,,\end{equation}
where the sum is over all compatible gradings $\omega$ on $\calP([\aone]_+ ,[\atwo]_+ )$. In other words, for any $(p, q)\in \N^2$, we have 
\[
    c(p, q) = |\CG(p, q)|\,.
\]
\end{thm}

When $\ell_1, \ell_2 < \infty$, we set in \Cref{def: weigth poly of grading} that $p_{1,k} = 0$ for any $k>\ell_1$ and $p_{2, k} = 0$ for any $k>\ell_2$. Then in this case, it follows from \eqref{eq: greedy expansion} that $x[d_1, d_2]$ is always a finite sum as there are only finitely many gradings with non-zero weights.

We will show in \Cref{cor: gen greedy containment} that $x[d_1, d_2]$ is not only a pointed element in $\calT_1$, but also belongs to any other $\calT_k$ (hence in $\mathcal U$ which actually equals $\mathcal A(P_1, P_2)$). Moreover, the set of all greedy elements
\[
    \{x[d_1, d_2]\mid (d_1, d_2)\in \Z^2\}
\]
is in fact a $\rngB$-basis of $\mathcal A(P_1, P_2)$. This statement generalizes \cite[Theorem 1.7(c)]{LLZ} and can be proven analogously as in \cite[Section 6]{LLZ}. This basis will be referred to as the \emph{generalized greedy basis}.

\section{The generalized greedy basis equals the generalized theta basis}\label{sec: greedy = theta}

It was shown by Cheung, Gross, Muller, Musiker, Rupel, Stella, and Williams that, in the setting of cluster algebras, the greedy basis equals the theta basis \cite{CGM}. Their proof is algebraic, using a characterization of greedy elements by their support and properties of rank-$2$ scattering diagrams. We extend their results to the generalized setting. 

\begin{thm}\label{thm: gen greedy equals theta}
The generalized greedy basis equals the generalized theta basis. More precisely, $\vartheta_{-(d_1, d_2)} = x[d_1, d_2]$ for any $(d_1, d_2)\in M$.
\end{thm}

Before presenting the proof, we establish a couple of lemmas. First, we give a support condition that can be used to identify generalized greedy elements, following \cite[Scholium 2.6]{CGM}. A crucial part of this argument is the compatible grading formula for coefficients of generalized greedy elements, which is developed in the previous section (see \Cref{thm: generalized greedy compatible gradings}).

\begin{defn}
Let $R^\circ_{\aone ,\atwo }$ be a region in $M_\R = \R^2$ defined as follows:
\begin{enumerate}[\normalfont(1)]
  \item If $\aone ,\atwo \leq 0$, then $R^\circ_{\aone ,\atwo } = \{(-\aone ,-\atwo )\}$,
  \item If $\aone  \leq 0 < \atwo $, then $R^\circ_{\aone ,\atwo } = \{(m_1,-\atwo ): -\aone \leq m_1 \leq -\aone + \ell_1\atwo \}$,
  \item If $\atwo  \leq 0 < \aone$, then $R^\circ_{\aone ,\atwo } = \{(-\aone ,m_2): -\atwo \leq m_2 \leq -\atwo + \ell_2\aone \}$,
  \item Otherwise, consider the vertices $O = (0,0)$, $A = (-\aone + \ell_1 \atwo , -\atwo )$, $B = (-\aone ,-\atwo )$, and $C = (-\aone , -\atwo + \ell_2 \aone)$. Then $R^\circ_{\aone ,\atwo }$ is the interior of the polygon $OABC$ along with the closed line segments $AB$ and $BC$ (see \cite[Figure 1]{CGM}).
\end{enumerate}
\end{defn}

Given a Laurent polynomial $f = \sum_{m \in \Z^2} c_m x^m$ in $\Bbbk [\Z^{\ell_1} \oplus \Z^{\ell_2}][x_1^{\pm 1}, x_2^{\pm 1}]$, the support of $f$ is the set $\{m \in \Z^2 : c_m \neq 0\}$. Rupel showed that the region $R^\circ_{\aone ,\atwo }$ contains the support of the generalized greedy element $x[\aone ,\atwo ]$ when $P_1$ and $P_2$ satisfy certain restrictions. In fact, this result also holds with these restrictions removed.

\begin{prop}[{\cite[Proposition 4.22]{Rupgengreed}}]
  The support of the generalized greedy element $x[\aone ,\atwo ]$ is contained in $R^\circ_{\aone ,\atwo }$. 
\end{prop}
\begin{proof}
  Rupel showed that this holds when $P_1$ and $P_2$ are palindromic polynomials by considering the generalized greedy coefficients as counts of compatible gradings ({\cite[Theorem 2.15]{Rupgengreed}). In the previous section, we extend this interpretation to the non-palindromic setting (see \Cref{thm: generalized greedy compatible gradings}). Rupel's argument then also holds without requiring $P_1$ and $P_2$ to be  palindromic.}
\end{proof}

We then obtain the following characterization of generalized greedy elements. 

\begin{lem}[{\cite[Scholium 2.6]{CGM}}]\label{lem: greedy support containment}
If $z \in \calA(P_1,P_2)$ is any element containing the monomial $x_1^{-\aone }x_2^{-\atwo }$ with coefficient $1$ and whose support is contained in $R^\circ_{\aone ,\atwo }$, then $z = x[\aone ,\atwo ]$.
\end{lem}
\begin{proof}
This follows directly from the proof of \cite[Theorem 2.5]{CGM} and the proof of \cite[Scholium 2.6]{CGM}, both of which carry through in the generalized cluster algebra setting.
\end{proof}

Thus, it is enough to show that the support of each generalized theta basis element is contained in $R^\circ_{\aone ,\atwo }$ for a corresponding choice of $\aone ,\atwo $. This is shown in the classical setting in \cite[Section 5]{CGM}, and their methods work essentially verbatim in the generalized case. Their results do not rely on the structure of the scattering diagram, other than that all walls are along the coordinate axes or in the third quadrant, which holds for the generalized cluster scattering diagrams. This culminates in bounds on the final exponent of the broken lines contributing to the generalized theta basis, yielding the desired bounds on the support of the theta basis elements.

\begin{defn}
Given a broken line $\beta$, fix a point $q = (q_1,q_2) \in M_\R$ on the linear portion $\beta|_{(\tau_{i-1},\tau_i)}$. The \emph{angular momentum of $\beta$ (at $q$)} is given by the dot product $(q_2, -q_1) \cdot r(m_i)$.  
\end{defn}

The angular momentum of a broken line does not depend on the choice of $q$, as shown in \cite[Lemma 5.3]{CGM}.  Broken lines of positive (resp. negative) angular momentum proceed counter-clockwise (resp. clockwise) around the origin. The behavior of broken lines terminating in the first quadrant can be characterized based on the sign of their angular momentum \cite[Lemma 5.4]{CGM}.  If $\beta$ is a broken line of positive (resp. negative) angular momentum, then the slopes of the linear domains of $\beta$ decrease (resp. increase) at each bend, except possibly at the boundary of the first quadrant.

Fix a general point $Q_+$ in the first quadrant of $M_{\R} = \R^2$, and denote $\vartheta_{Q_+,m} = \vartheta_m$ for $m \in M$.

\begin{lem}[{\cite[Lemma 5.5]{CGM}}]\label{lem: broken line slope bounds}
Let $\gamma$ be a broken line in $\Scat(P_1,P_2)$ which begins in the third quadrant, with endpoint $Q_+$ in the first quadrant. Denote the initial exponent by $(-\aone ,-\atwo )$ and the final exponent by $m(\gamma) = (m_1,m_2)$. 
\begin{enumerate}
  \item If $\gamma$ has positive angular momentum, then $-\atwo \leq m_2 < 0$ and 
  $$-\aone \leq m_1 \leq \left(\frac{\aone }{\atwo } - \ell_1\right)m_2\,,$$
  where the upper bound is achieved only when $(m_1,m_2) = (-\aone +\ell_1\atwo , -\atwo )$.
  \item If $\gamma$ has negative angular momentum, then $-\aone \leq m_1 < 0$ and 
  $$-\atwo \leq m_2 \leq \left(\frac{\atwo }{\aone } - \ell_2\right)m_1\,,$$
  where the upper bound is achieved only when $(m_1,m_2)= (-\aone , -\atwo + \ell_2\aone )$.
\end{enumerate}
\end{lem}

We now have all the tools necessary to prove \Cref{thm: gen greedy equals theta}. 

\begin{proof}[Proof of \Cref{thm: gen greedy equals theta}]
We proceed by showing that the theta function $\vartheta_{(-\aone ,-\atwo )}$ satisfies the conditions of \Cref{lem: greedy support containment}, following the approach of \cite[Proof of Theorem 5.1]{CGM}. First, note that a broken line with the same initial and terminal momentum in $\Scat(P_1,P_2)$ must not bend, as any bending would increase one of the entries of the exponent. There is a unique broken line terminating at $Q_+$ that does not bend and has terminal momentum $(-\aone ,-\atwo )$, which contributes the term $x_1^{-\aone }x_2^{-\atwo }$ to $\vartheta_{(-\aone ,-\atwo )}$. Hence $x_1^{-\aone }x_2^{-\atwo }$ appears with coefficient $1$ in $\vartheta_{(-\aone ,-\atwo )}$. For the support condition, we split into four cases.
\begin{enumerate}[\normalfont(1)]
  \item If $\aone ,\atwo \leq 0$, then by \Cref{prop: theta function as monomial} we have $\vartheta_{(-\aone ,-\atwo )} = x_1^{-\aone }x_2^{-\atwo }$, so the support is simply $\{(-\aone ,-\atwo )\} = R^\circ_{\aone ,\atwo }$.
  \item If $\aone  \leq 0 < \atwo $, then a broken line with initial momentum $(-\aone ,-\atwo )$ terminating at $Q_+$ can bend only at the $x$-axis. This follows because any other bending would need to occur to the left of $Q_+$, and such a broken line can never move in the positive $x$ direction to return to $Q_+$. Thus, the terminal momentums of broken lines that contribute to $\vartheta_{(-\aone ,-\atwo )}$ are precisely $(-\aone ,-\atwo ) + i(\ell_1,0)$ for any integer $i \in \{0,1,\dots,\ell_1\}$. 
  \item If $\atwo  \leq 0 < \aone$, then an analogous argument to that in Case (2) applies.
  \item If $\aone ,\atwo > 0$, then we can apply \Cref{lem: broken line slope bounds} to conclude that the final exponent of any broken line $\gamma$ contributing to $\vartheta_{(-\aone ,-\atwo )}$ must have final exponent contained in $R^\circ_{\aone ,\atwo }$. 
\end{enumerate}  
Therefore, the terminal momentum of any broken line contributing to $\vartheta_{(-\aone ,-\atwo )}$, and hence the support of $\vartheta_{(-\aone ,-\atwo )}$, is contained in $R^\circ_{\aone ,\atwo }$. Applying \Cref{lem: greedy support containment}, we can conclude that $\vartheta_{(-\aone ,-\atwo )} = x[\aone ,\atwo ]$.
\end{proof}

\section{Shadowed gradings and tight gradings}\label{sec: shadowed and tight grading}

We define new combinatorial objects called \emph{tight gradings} that give a combinatorial interpretation of the wall-function coefficients in rank-$2$ generalized cluster scattering diagrams. By considering a larger class of gradings called \emph{shadowed gradings}, we also give a combinatorial interpretation of the coefficients of powers of the wall-functions. The latter is useful for computing Euler characteristics of quiver moduli (see \Cref{subsec: euler char quiver}) and relative Gromov--Witten invariant in toric surfaces (see \Cref{subsec: GW invariants}). 

In this section, we begin by defining shadowed gradings and tight gradings. We then describe certain conditions under which we can isolate a unique broken line that corresponds to a set of shadowed gradings. Utilizing these conditions, we relate the combinatorial approach of Lee--Li--Zelevinsky and the scattering diagram approach of Gross--Hacking--Keel--Kontsevich \cite{LLZ, GHKK}. This yields a directly computable, manifestly positive, and elementary but highly nontrivial formula describing rank-2 consistent scattering diagrams. 

\subsection{Shadowed gradings and tight gradings}\label{subsec: shadow and tight grading}

We now introduce shadowed gradings and tight gradings, which are subsets of the compatible gradings defined by Rupel \cite{Rupgengreed} and inspired by the compatible pairs of Lee--Li--Zelevinsky \cite{LLZ} (see \Cref{subsec: compatible grading}).  We first require the notion of the ``shadow" of a set of edges in an NE path.

\begin{defn}\label{def: shadow}
Fix an NE path $\calP$ and a grading $\omega \colon \calP \rightarrow \N$.
For each edge $e$ in $\calP$, we define its \emph{shadow}, denoted by $\sh(e)$ or $\sh(e;\omega)$, as follows. 
\begin{itemize}
    \item If $e$ is horizontal, then its shadow is $(\overrightarrow{ev})_\north$, where $v \in \calP_\north$ is chosen such that $|(\overrightarrow{ev})_\north| = \omega((\overrightarrow{ev})_\east)$ and $\overrightarrow{ev}$ has minimal length under this condition.  If no such $v$ exists, let $\sh(e) = \calP_\north$.
    \item If $e$ is vertical, then its shadow is $(\overrightarrow{ue})_\east$, where $u \in \calP_\east$ is chosen such that $|(\overrightarrow{ue})_\east| = \omega((\overrightarrow{ue})_\north)$ and $\overrightarrow{ue}$ has minimal length under this condition.  If no such $u$ exists, let $\sh(e) = \calP_\east$.

\end{itemize}  

For $S \subset \calP$, let the \emph{shadow} of $S$ be $\sh(S)=\bigcup_{e\in S} \sh(e)$.  
\end{defn}

Note that if $\omega(e) = 0$, then $\sh(e) = \emptyset$.

\begin{exmp}
  Consider $\omega_2$ as in \Cref{main_exmp}. Then $\sh(u_2) = \{v_1,v_2\}$ and $\sh(u_1)=\{v_1,\dots,v_4\} = \sh(\{u_1,u_2\})$. We also have $\sh(v_3)=\{u_4,u_5,u_6\}$ and $\sh(v_4)=\{u_2,u_3,\dots,u_7\}=\sh(\{v_3,v_4\})$.
\end{exmp}

We then define a subset of compatible gradings where the vertical edges with nonzero grading value are contained in the shadow of the horizontal edges, or vice versa.

\begin{defn}\label{main_def}
A \emph{shadowed grading} on $\calP$ is a compatible grading satisfying either
$$S_\east\subseteq\sh(\calP_\north) \quad \text{or} \quad S_\north\subseteq\sh(\calP_\east)\,,$$
where $S$ is the set of edges $e$ in $\calP$ with $\omega(e)>0$. 
\end{defn}

The notion of shadowed gradings generalizes the notion of tight gradings, as first defined in \cite{BLMshort} and partially motivated by \cite{BL}.

\begin{defn}
Let $\omega$ be a grading on $\calP$ with $p = \omega(\calP_\north)$ and $q = \omega(\calP_\east)$.
The grading $\omega$ is a \emph{tight grading} if $\omega$ is a shadowed grading satisfying $p \leq d_1$, $q \leq d_2$, and
$$pd_2 - qd_1 = \pm \gcd(p,q)\,.$$
\end{defn}

We now explain how to represent shadowed gradings as rectangular tilings. This is important for the illustration of shadowed and tight gradings. First, rotate (i.e. cyclically shift) the maximal Dyck path so that the shadow of each horizontal (resp. vertical) edge does not extend beyond the left (resp. top) boundary of the rotated path. Such a rotation always exists provided $\omega(\calP_\north)\leq d_1$ and $\omega(\calP_\east) \leq d_2$, which always holds for tight gradings. Then draw blue rectangles above each horizontal edge $e$ with total height equal to $|\sh(e)|
$, partitioned into the vertical grading values contributing to the shadow. Similarly, we draw red rectangles to the left of each edge in $\calP_\north$. The grading is compatible only if the resulting rectangles are non-overlapping. The condition on the shadows can be easily read off by seeing if every blue rectangle has a red rectangle above, or if every red rectangle has a blue rectangle to the left.

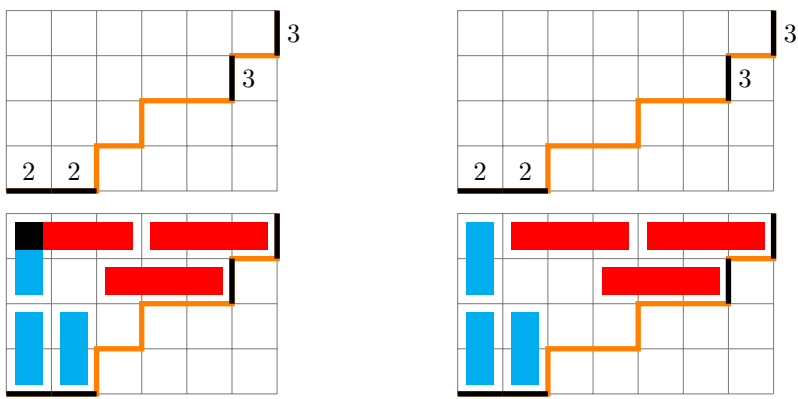
\begin{figure}[hb]
\centering
\begin{tikzpicture}[scale=.6]
\draw[step=1,color=gray] (0,0) grid (6,4);
\draw[line width=2,color=orange] (0,0)--(2,0)--(2,1)--(3,1)--(3,2)--(5,2)--(5,3)--(6,3)--(6,4);
\draw[line width=2pt] (0,0)--(1,0);
\draw (0.5,0.85) node[anchor=north] {\small $2$};
\draw[line width=2pt] (1,0)--(2,0);
\draw (1.5,0.85) node[anchor=north] {\small $2$};
\draw[line width=2pt] (6,4)--(6,3);
\draw (6,3.5) node[anchor=west] {\small $3$};
\draw[line width=2pt] (5,3)--(5,2);
\draw (5,2.5) node[anchor=west] {\small $3$};
\begin{scope}[shift={(10,0)}]
\draw[step=1,color=gray] (0,0) grid (7,4);
\draw[line width=2,color=orange] (0,0)--(2,0)--(2,1)--(4,1)--(4,2)--(6,2)--(6,3)--(7,3)--(7,4);
\draw[line width=2pt] (0,0)--(1,0);
\draw (0.5,0.85) node[anchor=north] {\small $2$};
\draw[line width=2pt] (1,0)--(2,0);
\draw (1.5,0.85) node[anchor=north] {\small $2$};
\draw[line width=2pt] (7,4)--(7,3);
\draw (7,3.5) node[anchor=west] {\small $3$};
\draw[line width=2pt] (6,3)--(6,2);
\draw (6,2.5) node[anchor=west] {\small $3$};
\end{scope}
\begin{scope}[shift={(0,-4.5)}]
\draw[step=1,color=gray] (0,0) grid (6,4);
\fill[cyan] (0.2,0.2)--(0.2,1.8)--(0.8,1.8)--(0.8,0.2)--(0.2,0.2);
\fill[cyan] (1.2,0.2)--(1.2,1.8)--(1.8,1.8)--(1.8,0.2)--(1.2,0.2);
\fill[cyan] (0.2,2.2)--(0.2,2+1.8)--(0.8,2+1.8)--(0.8,2+0.2)--(0.2,2+0.2);
\fill[red] (3.2-1,3.2-1)--(3.2-1,3.8-1)--(5.8-1,3.8-1)--(5.8-1,3+0.2-1)--(3.2-1,3.2-1);
\fill[red] (3.2+1-1,3.2)--(3.2+1-1,3.8)--(5.8+1-1,3.8)--(5.8+1-1,3+0.2)--(3.2+1-1,3.2);
\fill[red] (3.2+1-3-1,3.2)--(3.2+1-3-1,3.8)--(5.8+1-3-1,3.8)--(5.8+1-3-1,3+0.2)--(3.2+1-3-1,3.2);
\fill[black] (3.2+1-3-1,3.2)--(3.2+1-3-1,3.8)--(3.8+1-3-1,3.8)--(3.8+1-3-1,3+0.2)--(3.2+1-3-1,3.2);
\draw[line width=2,color=orange] (0,0)--(2,0)--(2,1)--(3,1)--(3,2)--(5,2)--(5,3)--(6,3)--(6,4);
\draw[line width=2pt] (0,0)--(1,0);
\draw[line width=2pt] (1,0)--(2,0);
\draw[line width=2pt] (6,4)--(6,3);
\draw[line width=2pt] (5,3)--(5,2);
\begin{scope}[shift={(10,0)}]
\draw[step=1,color=gray] (0,0) grid (7,4);
\draw[line width=2,color=orange] (0,0)--(2,0)--(2,1)--(4,1)--(4,2)--(6,2)--(6,3)--(7,3)--(7,4);
\draw[line width=2pt] (0,0)--(1,0);
\draw[line width=2pt] (1,0)--(2,0);
\draw[line width=2pt] (7,4)--(7,3);
\draw[line width=2pt] (6,3)--(6,2);
\fill[cyan] (0.2,0.2)--(0.2,1.8)--(0.8,1.8)--(0.8,0.2)--(0.2,0.2);
\fill[cyan] (1.2,0.2)--(1.2,1.8)--(1.8,1.8)--(1.8,0.2)--(1.2,0.2);
\fill[cyan] (0.2,2.2)--(0.2,2+1.8)--(0.8,2+1.8)--(0.8,2+0.2)--(0.2,2+0.2);
\fill[red] (3.2,3.2-1)--(3.2,3.8-1)--(5.8,3.8-1)--(5.8,3+0.2-1)--(3.2,3.2-1);
\fill[red] (3.2+1,3.2)--(3.2+1,3.8)--(5.8+1,3.8)--(5.8+1,3+0.2)--(3.2+1,3.2);
\fill[red] (3.2+1-3,3.2)--(3.2+1-3,3.8)--(5.8+1-3,3.8)--(5.8+1-3,3+0.2)--(3.2+1-3,3.2);
\end{scope}
\end{scope}
\end{tikzpicture}
\caption{In the top images, we depict gradings $\omega_1$ and $\omega_2$ on the Dyck paths $\calP(6,4)$ and $\calP(7,4)$ from \Cref{main_exmp}, where edges with no weight shown are assigned weight $0$. Below, we depict the corresponding rectangular tilings. The grading $\omega_1$ has overlapping rectangles and hence is not compatible, while the grading $\omega_2$ is.
} 
\label{fig: compatible grading example}
\end{figure}

\begin{rem}
The word ``tight" is coined by the tight space between blue and red rectangles.
\end{rem}

\begin{exmp}\label{tight_exmp}
(1) The grading $\omega_2$ as in \Cref{main_exmp} (and displayed on the right of \Cref{fig: compatible grading example}) is shadowed.  However, it is not tight because $(\aone ,\atwo )=(7,4)$ does not satisfy $pd_2 - qd_1 = \pm\gcd(p,q)$ for $(p,q)=(6,4)$.

\noindent (2) Let $(p,q)=(2,1)$ and $(\aone ,\atwo )=(3,1)$. Let $\omega \colon \mathcal P(3, 1)\rightarrow \N$ be a grading such that $\omega(u_1)=1$, $\omega(u_2)=\omega(u_3)=0$, and $\omega(v_1)=2$. Then $\omega$ is tight as displayed on the left of \Cref{fig: tight grading g2}.

\noindent (3) Let $(p,q)=(4,2)$ and $(\aone ,\atwo )=(5,2)$. Let $\omega \colon \mathcal{P}(5,2) \rightarrow \N$ be a grading such that $\omega(u_1)=\omega(u_2)=\omega(v_1)=1$, $\omega(v_2)=3$, and $\omega(u_3)=\omega(u_4)=\omega(u_5)=0$. Then $\omega$ is tight as displayed in the middle of \Cref{fig: tight grading g2}.

\noindent (4) Let $(p,q)=(6,3)$ and $(\aone ,\atwo )=(7,3)$. Let $\omega \colon \mathcal P(7, 3) \rightarrow \N$ be a grading such that $\omega(v_2)=\omega(v_3)=3$, $\omega(u_1)=\omega(u_2)=\omega(u_3)=1$, and $\omega(v_1)=\omega(u_4)=\omega(u_5)=\omega(u_6)=\omega(u_7)=0$. Then $\omega$ is tight as displayed on the right of \Cref{fig: tight grading g2}.

\noindent (5) Let $(p,q)=(12,8)$ and $(\aone ,\atwo )=(14,9)$. There are a total of 14 tight gradings $\omega$ such that $\omega(h)=2$ for exactly four horizontal edges $h$, $\omega(v)=3$ for exactly four vertical edges $v$, and $\omega(e)=0$ for all other edges on $\mathcal{P}(14,9)$. 
\end{exmp}

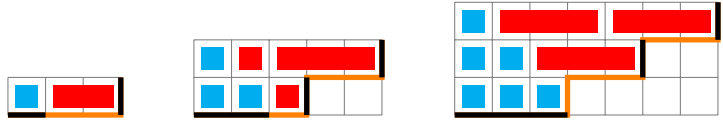
\begin{figure}[ht]
    \centering
    \begin{tikzpicture}[scale=.5]
        \draw[step=1,color=gray] (0,0) grid (3,1);
        \draw[line width=2,color=orange] (0,0)--(3,0)--(3,1);
        \fill[cyan] (0.2,0.2)--(0.2,0.8)--(0.8,0.8)--(0.8,0.2)--(0.2,0.2);
        \fill[red] (1.2,0.2)--(1.2,0.8)--(2.8,0.8)--(2.8,0.2)--(1.2,0.2);
        \draw[line width=2pt] (0,0)--(1,0);
        \draw[line width=2pt] (3,0)--(3,1);
    \end{tikzpicture}
    \quad\quad
    \begin{tikzpicture}[scale=.5]
        \draw[step=1,color=gray] (0,0) grid (5,2);
        \draw[line width=2,color=orange] (0,0)--(3,0)--(3,1)--(5,1)--(5,2);
        \fill[cyan] (0.2,0.2)--(0.2,0.8)--(0.8,0.8)--(0.8,0.2)--(0.2,0.2);
        \fill[cyan] (1+0.2,0.2)--(1+0.2,0.8)--(1+0.8,0.8)--(1+0.8,0.2)--(1+0.2,0.2);
        \fill[cyan] (0.2,1+0.2)--(0.2,1+0.8)--(0.8,1+0.8)--(0.8,1+0.2)--(0.2,1+0.2);
        \fill[red] (2.2,0.2)--(2.2,0.8)--(2.8,0.8)--(2.8,0.2)--(2.2,0.2);
        \fill[red] (2.2-1,1+0.2)--(2.2-1,1+0.8)--(2.8-1,1+0.8)--(2.8-1,1+0.2)--(2.2-1,1+0.2);
        \fill[red] (2.2,1+0.2)--(2.2,1+0.8)--(4.8,1+0.8)--(4.8,1+0.2)--(2.2,1+0.2);
        \draw[line width=2pt] (0,0)--(2,0);
        \draw[line width=2pt] (3,0)--(3,1);
        \draw[line width=2pt] (5,1)--(5,2);
    \end{tikzpicture}
    \quad\quad
    \begin{tikzpicture}[scale=.5]
        \draw[step=1,color=gray] (0,0) grid (7,3);
        \draw[line width=2,color=orange] (0,0)--(3,0)--(3,1)--(5,1)--(5,2)--(7,2)--(7,3);
        \fill[cyan] (0.2,0.2)--(0.2,0.8)--(0.8,0.8)--(0.8,0.2)--(0.2,0.2);
        \fill[cyan] (1+0.2,0.2)--(1+0.2,0.8)--(1+0.8,0.8)--(1+0.8,0.2)--(1+0.2,0.2);
        \fill[cyan] (2+0.2,0.2)--(2+0.2,0.8)--(2+0.8,0.8)--(2+0.8,0.2)--(2+0.2,0.2);
        \fill[cyan] (0.2,1+0.2)--(0.2,1+0.8)--(0.8,1+0.8)--(0.8,1+0.2)--(0.2,1+0.2);
        \fill[cyan] (1+0.2,1+0.2)--(1+0.2,1+0.8)--(1+0.8,1+0.8)--(1+0.8,1+0.2)--(1+0.2,1+0.2);
        \fill[cyan] (0.2,2+0.2)--(0.2,2+0.8)--(0.8,2+0.8)--(0.8,2+0.2)--(0.2,2+0.2);
        \fill[red] (2.2,1+0.2)--(2.2,1+0.8)--(4.8,1+0.8)--(4.8,1+0.2)--(2.2,1+0.2);
        \fill[red] (2.2-1,1+1+0.2)--(2.2-1,1+1+0.8)--(4.8-1,1+1+0.8)--(4.8-1,1+1+0.2)--(2.2-1,1+1+0.2);
        \fill[red] (2.2-1+3,1+1+0.2)--(2.2-1+3,1+1+0.8)--(4.8-1+3,1+1+0.8)--(4.8-1+3,1+1+0.2)--(2.2-1+3,1+1+0.2);
        \draw[line width=2pt] (0,0)--(3,0);
        \draw[line width=2pt] (5,1)--(5,2);
        \draw[line width=2pt] (7,2)--(7,3);
    \end{tikzpicture}
    \caption{The three tight gradings in \Cref{tight_exmp}(2)(3)(4).}
    \label{fig: tight grading g2}
\end{figure}

We now discuss two families of tight gradings with relevance to the rank-2 cluster scattering diagrams of affine type, i.e., $\Scat(1 + x^{\ell_1}, 1 + y^{\ell_2})$ where $\ell_1\ell_2 = 4$.  This connection is made explicit in \Cref{exmp: affine 2 2 wall-function} and \Cref{exmp: affine 4 1 wall-function}.

\begin{exmp}\label{exmp: affine 2 2 gradings}Let $P_1(x) = 1 + p_{1,2}x^2$ and $P_2(y) = 1 + p_{2,2}y^2$. For $k \in \N$, consider the path $\calP(2k,2k+1)$.  There are $k + 1$ tight gradings of weight $p_{1,2}^kp_{2,2}^k$, that is, tight gradings $\omega$ where $\omega(h) = 2$ for exactly $k$ horizontal edges $h$, $\omega(v) = 2$ for exactly $k$ vertical edges $v$, and $\omega(e) = 0$ for all other edges $e$. 

It is straightforward to describe these gradings explicitly as follows.  Fix $i \in \{0,1,\dots,k\}$.  Let $S_\east = \{u_{1 + i}, u_{2 + i}, \dots, u_{k + i}\}$ and let $S_\north = \{v_1,v_2,\dots,v_i\} \cup \{v_{k + 2 + i}, v_{k + 3 + i},\dots, v_{2k + 1}\}$, where the sets in the union can be empty if $i = 0$ or $i = k$.  Then the grading
$$\omega(e) = \begin{cases} 2 & \text{ if } e \in S_\east \cup S_\north\\ 0 & \text{ otherwise.} \end{cases}$$
is one of the desired tight gradings, and varying the choice of $i$ yields all such tight gradings.   

\begin{figure}[ht]
\centering
\begin{tikzpicture}[scale=.4]
\draw[step=1,color=gray] (0,0) grid (6,7);
\vrec{0}{0}{3}{2}
\vrec{1}{1}{2}{2}
\vrec{2}{2}{1}{2}
\hrec{5}{4}{1}{2}
\hrec{6}{5}{2}{2}
\hrec{6}{6}{3}{2}
\draw[line width=2,color=orange] (0,0)--(1,0)--(1,1)--(2,1)--(2,2)--(3,2)--(3,3)--(4,3)--(4,4)--(5,4)--(5,5)--(6,5)--(6,7);
\draw[line width=2pt] (0,0)--(1,0);
\draw[line width=2pt] (1,1)--(2,1);
\draw[line width=2pt] (2,2)--(3,2);
\draw[line width=2pt] (5,4)--(5,5);
\draw[line width=2pt] (6,5)--(6,7);

\begin{scope}[shift={(8,0)}]
\draw[step=1,color=gray] (0,0) grid (6,7);
\vrec{0}{0}{3}{2}
\vrec{1}{1}{2}{2}
\vrec{2}{2}{1}{2}
\hrec{5}{4}{1}{2}
\hrec{5}{5}{2}{2}
\hrec{6}{6}{3}{2}
\draw[line width=2,color=orange] (0,0)--(1,0)--(1,1)--(2,1)--(2,2)--(3,2)--(3,3)--(4,3)--(4,4)--(5,4)--(5,6)--(6,6)--(6,7);
\draw[line width=2pt] (0,0)--(1,0);
\draw[line width=2pt] (1,1)--(2,1);
\draw[line width=2pt] (2,2)--(3,2);
\draw[line width=2pt] (5,4)--(5,6);
\draw[line width=2pt] (6,6)--(6,7);
\end{scope}

\begin{scope}[shift={(16,0)}]
\draw[step=1,color=gray] (0,0) grid (6,7);
\vrec{0}{0}{3}{2}
\vrec{1}{1}{2}{2}
\vrec{2}{2}{1}{2}
\hrec{4}{4}{1}{2}
\hrec{5}{5}{2}{2}
\hrec{6}{6}{3}{2}
\draw[line width=2,color=orange] (0,0)--(1,0)--(1,1)--(2,1)--(2,2)--(3,2)--(3,3)--(4,3)--(4,5)--(5,5)--(5,6)--(6,6)--(6,7);
\draw[line width=2pt] (0,0)--(1,0);
\draw[line width=2pt] (1,1)--(2,1);
\draw[line width=2pt] (2,2)--(3,2);
\draw[line width=2pt] (4,4)--(4,5);
\draw[line width=2pt] (5,5)--(5,6);
\draw[line width=2pt] (6,6)--(6,7);
\end{scope}

\begin{scope}[shift={(24,0)}]
\draw[step=1,color=gray] (0,0) grid (6,7);
\vrec{0}{0}{3}{2}
\vrec{1}{1}{2}{2}
\vrec{2}{2}{1}{2}
\hrec{4}{4}{1}{2}
\hrec{5}{5}{2}{2}
\hrec{6}{6}{3}{2}
\draw[line width=2,color=orange] (0,0)--(1,0)--(1,1)--(2,1)--(2,2)--(3,2)--(3,4)--(4,4)--(4,5)--(5,5)--(5,6)--(6,6)--(6,7);
\draw[line width=2pt] (0,0)--(1,0);
\draw[line width=2pt] (1,1)--(2,1);
\draw[line width=2pt] (2,2)--(3,2);
\draw[line width=2pt] (4,4)--(4,5);
\draw[line width=2pt] (5,5)--(5,6);
\draw[line width=2pt] (6,6)--(6,7);
\end{scope}
\end{tikzpicture}
\caption{The four tight gradings in the setting of \Cref{exmp: affine 2 2 gradings} when $k = 3$.}
\end{figure}
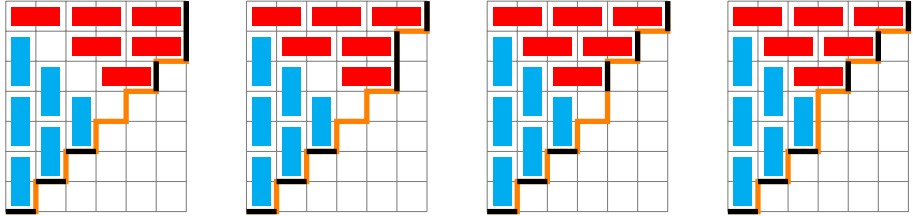
\end{exmp}

\begin{exmp}\label{exmp: affine 4 1 gradings} Let $P_1(x) = 1 + p_{1,4}x^4$ and $P_2(y) = 1 + p_{2,1}y$.
For $n \in \N$, consider the path $\calP(4k+1,2k)$.  There are $2k + 1$ tight gradings of weight $p_{1,4}^k p_{2,1}^{2k}$.  These are the tight gradings $\omega$ such that $\omega(h) = 1$ for exactly $2k$ horizontal edges $h$, $\omega(v) = 4$ for exactly $k$ vertical edges $v$, and $\omega(e) = 0$ for all other edges $e$. 

We can describe these gradings explicitly as follows.  Fix $i \in \{0,1,\dots,k - 1\}$ and $s \in \{0,1\}$, with the restriction that if $i = k$ then $s = 0$.   Let $S_\north = \{v_{k + 1 - i}, v_{k + 2 - i}, \dots, v_{2k - i}\}$ and let $S_\east = \{u_1, u_2, \dots, u_{4k - 2i - 1 }\} \cup \{u_{4k - 2i + s}\} \cup \{u_{4k - 2i + 2}, u_{4k - 2i + 3},\dots, u_{4k + 1}\}$, where the sets in the union can be empty if $i = 0$.  Then the grading
$$\omega(e) = \begin{cases} 2 & \text{ if } e \in S_\east \cup S_\north\\ 0 & \text{ otherwise.} \end{cases}$$
is one of the desired tight gradings, and varying the choice of $i$ yields all but one of the $2k+1$ such tight gradings.  The last tight grading can be obtained by considering $S_\east = \{u_{2k+2}, u_{2k+3},\dots, u_{4k+1}\}$ and $S_\north = \{v_1,v_2,\dots,v_k\}$, with the grading $\omega$ constructed as above.

\begin{figure}[ht]
\centering
\begin{tikzpicture}[scale=.3]
\draw[step=1,color=gray] (0,0) grid (9,4);
\vrec{0}{0}{4}{1}
\vrec{1}{0}{3}{1}
\vrec{2}{0}{1}{1}
\vrec{3}{1}{1}{1}
\hrec{7}{2}{1}{4}
\hrec{9}{3}{2}{4}
\draw[line width=2,color=orange] (0,0)--(3,0)--(3,1)--(5,1)--(5,2)--(7,2)--(7,3)--(9,3)--(9,4);
\draw[line width=2pt] (0,0)--(3,0);
\draw[line width=2pt] (3,1)--(4,1);
\draw[line width=2pt] (7,2)--(7,3);
\draw[line width=2pt] (9,3)--(9,4);

\begin{scope}[shift={(10.5,0)}]
\draw[step=1,color=gray] (0,0) grid (9,4);
\vrec{0}{0}{4}{1}
\vrec{1}{0}{3}{1}
\vrec{2}{0}{1}{1}
\vrec{4}{1}{1}{1}
\hrec{7}{2}{1}{4}
\hrec{9}{3}{2}{4}
\draw[line width=2,color=orange] (0,0)--(3,0)--(3,1)--(5,1)--(5,2)--(7,2)--(7,3)--(9,3)--(9,4);
\draw[line width=2pt] (0,0)--(3,0);
\draw[line width=2pt] (4,1)--(5,1);
\draw[line width=2pt] (7,2)--(7,3);
\draw[line width=2pt] (9,3)--(9,4);
\end{scope}

\begin{scope}[shift={(21,0)}]
\draw[step=1,color=gray] (0,0) grid (9,4);
\vrec{0}{0}{4}{1}
\vrec{1}{0}{1}{1}
\vrec{2}{1}{2}{1}
\vrec{3}{1}{1}{1}
\hrec{7}{2}{1}{4}
\hrec{9}{3}{2}{4}
\draw[line width=2,color=orange] (0,0)--(2,0)--(2,1)--(5,1)--(5,2)--(7,2)--(7,3)--(9,3)--(9,4);
\draw[line width=2pt] (0,0)--(2,0);
\draw[line width=2pt] (2,1)--(3,1);
\draw[line width=2pt] (3,1)--(4,1);
\draw[line width=2pt] (7,2)--(7,3);
\draw[line width=2pt] (9,3)--(9,4);
\end{scope}

\begin{scope}[shift={(31.5,0)}]
\draw[step=1,color=gray] (0,0) grid (9,4);
\vrec{0}{0}{4}{1}
\vrec{1}{0}{1}{1}
\vrec{2}{1}{2}{1}
\vrec{4}{1}{1}{1}
\hrec{7}{2}{1}{4}
\hrec{9}{3}{2}{4}
\draw[line width=2,color=orange] (0,0)--(2,0)--(2,1)--(5,1)--(5,2)--(7,2)--(7,3)--(9,3)--(9,4);
\draw[line width=2pt] (0,0)--(2,0);
\draw[line width=2pt] (2,1)--(3,1);
\draw[line width=2pt] (4,1)--(5,1);
\draw[line width=2pt] (7,2)--(7,3);
\draw[line width=2pt] (9,3)--(9,4);
\end{scope}

\begin{scope}[shift={(42,0)}]
\draw[step=1,color=gray] (0,0) grid (9,4);
\vrec{0}{0}{4}{1}
\vrec{1}{0}{1}{1}
\vrec{2}{1}{2}{1}
\vrec{3}{1}{1}{1}
\hrec{7}{2}{1}{4}
\hrec{9}{3}{2}{4}
\draw[line width=2,color=orange] (0,0)--(2,0)--(2,1)--(4,1)--(4,2)--(7,2)--(7,3)--(9,3)--(9,4);
\draw[line width=2pt] (0,0)--(2,0);
\draw[line width=2pt] (2,1)--(4,1);
\draw[line width=2pt] (7,2)--(7,3);
\draw[line width=2pt] (9,3)--(9,4);
\end{scope}
\end{tikzpicture}
\caption{The five tight gradings in the setting of \Cref{exmp: affine 4 1 gradings} when $k = 2$.}
\end{figure}
\end{exmp}

\subsection{One-bending property}\label{subsec: one-bending}

The \emph{bend-type} of a broken line is the tuple $(r(m_0), \dots, r(m_\ell))\in M^{\ell+1}$. In a scattering diagram $\frakD$, for a fixed endpoint $Q$, every broken line $\beta$ with a given bend-type has identical image in $M_\mathbb R$. This can be seen by working backwards from $Q$ to reconstruct the broken line.

Consider a generalized cluster scattering diagram $\frakD = \Scat(P_1, P_2)$.

\begin{defn}\label{def: one-bending}
    Let $(\aone , \atwo )\in \mathbb N^2$ and coprime $(a, b)\in \mathbb N^2$ such that
    \[
        \aone b - \atwo a \neq 0.
    \]
    Let $\beta$ be a (possible) broken line in $\frakD$ such that it has initial exponent $(-\aone, -\atwo)$, one unique bend at the wall $(\mathbb R_{\leq 0}(a, b), f_{\mathbb R_{\leq 0}(a, b)})$, and final exponent $(ka-\aone, kb-\atwo)$ for $k\geq 1$. We say that the pair $(\aone , \atwo )$ and $(ka, kb)$ satisfies the \emph{one-bending property} (with respect to $\frakD$) if any broken line that
    \begin{align}
        & \text{has initial exponent $(-\aone , -\atwo )$ and final exponent $(-\aone + ka, -\atwo + kb)$, and} \label{eq: one-bend 1}\\
        & \text{does not bend at positive coordinate axes} \label{eq: one-bend 2}
    \end{align}
    has the same bend-type as $\beta$.
\end{defn}

\begin{rem}\label{rem: one-bending nonempty}
Let $(d_1,d_2)$ and $(ka,kb)$ satisfy the one-bending property.  Depending on the choice of terminal point $Q \in M_\mathbb R$, there is either exactly one broken line or no broken lines satisfying \eqref{eq: one-bend 1} and \eqref{eq: one-bend 2} and terminating at $Q$. In particular, there will be one such broken line if 
\[Q \in \begin{cases}\left\{(q_1,q_2) \subseteq \R^2 : q_2 < \min\left(\frac{kb-d_2}{ka - d_1} q_1, \frac{d_2}{d_1} q_1\right)\right\} & \text{ if } d_1b-d_2a > 0,\\[10pt] \left\{(q_1,q_2) \subseteq \R^2 : q_2 > \max\left(\frac{kb-d_2}{ka - d_1} q_1, \frac{d_2}{d_1} q_1\right)\right\} & \text{ if } d_1b-d_2a < 0, \end{cases}\]
and no such broken lines otherwise.
\end{rem}

\begin{lem}\label{lem: one-bending}
  If $(\aone, \atwo)$ and $(ka, kb)$ satisfy
  \begin{equation*}
    \frac{kab+1}{kb^2}\geq \frac{\aone}{\atwo} > \frac{a}{b} \quad \text{or} \quad
    \frac{ka^2}{kab+1}\leq \frac{\aone}{\atwo} < \frac{a}{b},
  \end{equation*}
  then they satisfy the one-bending property.
\end{lem}

\begin{proof}
    We prove for the first case and the second case is similar. It suffices to show that there is no broken line satisfying (\ref{eq: one-bend 1}) and (\ref{eq: one-bend 2}) that bends at least twice. Suppose that there is one such $\beta$. Notice that $\beta$ can only bend at walls
    strictly below $\mathbb R_{\leq 0}(\aone , \atwo )$, that is, walls $\mathbb R_{\leq 0}(a', b')$ such that
    \[
        \frac{d_1}{d_2} > \frac{a'}{b'}.
    \]
    Denote the amount of bending cumulated between $\mathbb R_{\leq 0}(\aone , \atwo )$ and $\mathbb R_{\leq 0}(a, b)$ (inclusive) by $(a_1, b_1)\in \mathbb N^2$. Denote the amount of bending cumulated between $\mathbb R_{\leq 0}(a, b)$ (exclusive) and $\mathbb R_{\leq 0}(0, 1)$ (inclusive) by $(a_2 , b_2)\in \mathbb N^2$. Thus $(a_1 , b_1) + (a_2 , b_2) = (ka, kb)$. Now we have
    \[
        \frac{kab + 1}{kb^2} \geq \frac{\aone}{\atwo} > \frac{a_1}{b_1} > \frac{a}{b} > \frac{a_2}{b_2}.
    \]
    We note that since $\beta$ bends at least twice, the last three inequalities are all strict.
    We derive from $\dfrac{kab + 1}{kb^2} > \dfrac{a_1}{b_1} > \dfrac{a}{b}$ and $b_1 \leq kb$
    that
    \[
        ab_1 + \frac{b_1}{kb} > a_1 b > ab_1
        \implies ab_1 + 1 > a_1 b > ab_1.
    \]
    There is no integer solution for $(a_1 , b_1)$ and hence such a broken line $\beta$ does not exist.
\end{proof}

\begin{cor}\label{cor: vectors one-bending}
Let $(a, b)\in \mathbb N^2$ be coprime and $k\in \mathbb N$.
  \begin{enumerate}
    \item[\emph{(1)}] If $(\aone , \atwo ) = (kab + 1, kb^2)$ or $(ka^2, kab + 1)$, then $(\aone , \atwo )$ and $(ka, kb)$ satisfy the one-bending property.
    \item[\emph{(2)}] If $\aone \geq mka$ when $d_1b-d_2a<0$ or $\atwo \geq mkb$ when $d_1b-d_2a>0$ where $m = |\aone b-\atwo a|$, then $(\aone , \atwo )$ and $(ka, kb)$ satisfy the one-bending property.
    \item[\emph{(3)}] For any $m>0$, there exist infinitely many $(d_1, d_2)$ such that $(d_1, d_2)$ and $(ka, kb)$ satisfy the one-bending property and that $m = |d_1b - d_2a|$.
  \end{enumerate}
\end{cor}

\begin{proof}
  Statement (1) is a direct corollary of \Cref{lem: one-bending}. We prove (2) when $m = \aone b - \atwo a$, and the other case is similar. Again by \Cref{lem: one-bending}, it suffices to check 
  \[
    \frac{kab + 1}{kb^2} = \frac{a}{b} + \frac{1}{kb^2} \geq \frac{\aone}{\atwo} \iff \frac{1}{kb^2} \geq \frac{m}{\atwo b} \iff \atwo \geq mkb.
  \]

  For (3), choose some $(d_1, d_2)\in \mathbb N^2$ so that $m = |d_1b - d_2a|$. By (2), for sufficiently large $N$, $(d_1 + Na, d_2 + Nb)$ and $(ka, kb)$ satisfy the one-bending property.
\end{proof}

Consider the ordinary cluster scattering diagram $\mathfrak D_{(\ell_1, \ell_2)} = \Scat(1 + x^{\ell_1}, 1 + y^{\ell_2})$. More one-bending pairs can be found and will be useful in later applications. Notice that in this case, the total amount of bending of a broken line can only be of the form $(k\ell_1 \bar a, k \ell_2 \bar b)$,
where $k\in \mathbb N$ and $(\bar a, \bar b)\in \mathbb N^2$ is coprime.

\begin{lem}\label{lem: one-bending eg central}
  Let $(\bar a, \bar b)\in \mathbb N^2$ be coprime and $k\geq 1$.
  \begin{enumerate}
    \item If $\bar b \leq \ell_1$, then the vectors $(\aone ,\atwo ) = (k\ell_1 \bar a + 1, k \ell_2 \bar b)$ and $(k\ell_1\bar a, k\ell_2 \bar b)$ satisfy the one-bending property with respect to $\mathfrak D_{(\ell_1, \ell_2)}$.
    \item If $\bar a \leq \ell_2$, then the vectors $(\aone ,\atwo ) = (k\ell_1 \bar a, k \ell_2 \bar b + 1)$ and $(k\ell_1\bar a, k\ell_2 \bar b)$ satisfy the one-bending property with respect to $\mathfrak D_{(\ell_1, \ell_2)}$.
  \end{enumerate}
\end{lem}

\begin{proof}
  We prove (1) and the other case is similar. As in the proof of \Cref{lem: one-bending}, it suffices to show that there is no integer solution $a_1 \leq k\bar a$, $b_1\leq k\bar b$ for the inequality
  \begin{align*}
      \frac{k\ell_1 \bar a + 1}{k \ell_2 \bar b} > \frac{\ell_1 a_1}{\ell_2 b_1} > \frac{\ell_1 \bar a}{\ell_2 \bar b} \quad     \iff \quad b_1\bar a + \frac{b_1}{k\ell_1} > a_1\bar b > b_1\bar a.
  \end{align*}
  Since $b_1\leq k\bar b \leq k\ell_1$, we have that $a_1\bar b$ belongs to an open interval without any integer in it.
\end{proof}

The above lemma applies in particular to the \emph{central ray} $\mathbb R_{\leq 0}(\ell_1\bar a, \ell_2\bar b)$ with $(\bar a, \bar b) = (1, 1)$.

\subsection{Formula for wall-functions}\label{subsec: formula wall function}

Fix two positive integers $\ell_1, \ell_2$. Consider the generalized cluster scattering diagram 
\[
    \frakD = \Scat\left(1 + \sum_{k=1}^{\ell_1} p_{1,k}x^k, 1 + \sum_{k=1}^{\ell_2} p_{2,k}y^k\right).
\]
Let $(a, b)\in \mathbb N^2$ be coprime. Let $m$ be a positive integer. Then we can express the $m$-th power of the wall-function $f_{\mathbb R_{\leq 0}(a, b)}$ by
\begin{equation}\label{eq: m-th power of fab}
    \left(f_{\mathbb R_{\leq 0}(a, b)}\right)^m = 1 + \sum_{k\geq 1} \lambda^{(m)}(ka, kb) x^{ka}y^{kb},
\end{equation}
where each coefficient $\lambda^{(m)}(a,b)$ is a polynomial in $\{p_{i,j}\mid i=1, 2, 1\leq j\leq \ell_i\}$. When $m=1$, we simply denote $\lambda(ka, kb) = \lambda^{(1)}(ka, kb)$.

\begin{thm}[Shadowed grading formula]\label{thm: power shadow grading formula}
    \begin{equation}\label{eq: shadow grading formula}
        \left(f_{\mathbb R_{\leq 0}(a, b)}\right)^m = 1 + \sum_{k\geq 1}\sum_{\omega} \mathrm{wt}(\omega) x^{ka}y^{kb},  
    \end{equation}
    where the second sum is over all shadowed gradings on the Dyck path $\mathcal P(d_{1,k}, d_{2, k})$ of total horizontal weight $kb$ and total vertical weight $ka$ for any $(d_{1, k}, d_{2, k})\in \N^2$ such that
    \begin{itemize}
        \item $d_{1, k}\geq ka$ and $d_{2, k}\geq kb$,
        \item $(d_{1, k}, d_{2, k})$ and $(ka, kb)$ satisfy the one-bending property,
        \item $m = |ad_{2,k} - bd_{1,k}|$.
    \end{itemize}
\end{thm}

For any tuple $\mathbf Q = (q_j)_{j=1}^\ell\in \mathbb N^\ell$, define the weighted sum $w(\mathbf Q) \coloneqq \sum_{j=1}^\ell jq_j$. For tuples $\mathbf Q_i = (q_{i, j})_{j=1}^{\ell_i}\in \mathbb N^{\ell_i}$, $i = 1, 2$, denote $p_i^{\mathbf Q_i} = \prod_{j=1}^{\ell_i} p_{i,j}^{q_{i,j}}$. Then the expression \eqref{eq: m-th power of fab} can be refined as
\begin{equation}
    \left(f_{\mathbb R_{\leq 0}(a, b)}\right)^m = \sum_{k\geq 0}\sum_{\substack{w(\mathbf Q_1)=ka \\ w(\mathbf Q_2) = kb}}\lambda^{(m)}_{a, b}(\mathbf Q_1, \mathbf Q_2)p_1^{\mathbf Q_1}p_2^{\mathbf Q_2} x^{ka} y^{kb}.
\end{equation}
It follows from \Cref{thm: power shadow grading formula} that for each $(\mathbf Q_1, \mathbf Q_2)$ such that $w(\mathbf Q_1) = ka, w(\mathbf Q_2) = kb$,
\[
    \lambda^{(m)}_{a, b}(\mathbf Q_1, \mathbf Q_2) = \left|\left\{\omega \,\middle\vert\, \text{$\omega$ is shadowed},\ \mathrm{wt}(\omega) = p_1^{\mathbf Q_1}p_2^{\mathbf Q_2}\right\}\right|,
\]
that is, the number of shadowed gradings on $\mathcal P(d_{1,k}, d_{2, k})$ whose weight polynomial is $p_1^{\mathbf Q_1}p_2^{\mathbf Q_2}$. The fact that this number computes $\lambda_{a, b}^{(m)}(\mathbf Q_1, \mathbf Q_2)$ implies it does not depend on the choice of $(d_{1, k}, d_{2, k})$. This turns out to be a rather delicate property of shadowed gradings from the view of combinatorics.

\begin{rem}\label{rem: tight grading case}
    When $m = 1$, any $d_{1, k}\geq ka$, $d_{2, k}\geq kb$ such that $|ad_{2,k}-bd_{2,k}| = 1$ satisfy the one-bending property by \Cref{cor: vectors one-bending}(2). In this case, shadowed gradings are precisely tight gradings and the formula \eqref{eq: shadow grading formula} was announced in \cite[Theorem 3.5]{BLMshort}.
\end{rem}

We can use the formula \eqref{eq: shadow grading formula} to directly calculate the wall-function of the single wall comprising the Badlands in an affine cluster scattering diagram.

\begin{exmp}\label{exmp: affine 2 2 wall-function}
Consider the cluster scattering diagram $\frakD_{(2,2)}$. It follows from \Cref{thm: power shadow grading formula} and \Cref{exmp: affine 2 2 gradings} that 
$$f_{\R_{\leq 0}(1,1)} = \sum_{k \geq 0} (k+1) x^{2k}y^{2k}\,.$$
This was known to experts \cite{KS08} and was shown using quiver representations by Reineke \cite{Rei08}. 
\end{exmp}

\begin{exmp}\label{exmp: affine 4 1 wall-function}
Consider the cluster scattering diagram $\frakD_{(4,1)}$. It follows from \Cref{thm: power shadow grading formula} and \Cref{exmp: affine 4 1 gradings} that 
$$f_{\R_{\leq 0}(2,1)} = \sum_{k \geq 0} (2k+1) x^{4k}y^{2k}\,.$$
This has been previously calculated by Reading \cite{Rea} using limits of $F$-polynomial expressions.
\end{exmp}

\subsection{Proof of \Cref{thm: power shadow grading formula}}\label{subsec: first proof}

We now prove the shadowed grading formula for the coefficients of powers of wall-functions, as described in \Cref{subsec: formula wall function}.  This proof relies on the fact that the generalized greedy and generalized theta bases coincide, a result established in \Cref{sec: greedy = theta}.  We use the one-bending property from \Cref{subsec: one-bending} to refine the equality of these bases, allowing us to find the shadowed gradings corresponding to a particularly nice broken line.  This broken line is chosen to have weight equal to a chosen coefficient of the power of a wall-function, yielding the desired formula in terms of shadowed gradings. 

For the entirety of this subsection, fix $a$, $b$, $k$, and $m$ as in \Cref{subsec: formula wall function}, and set $\aone \coloneqq d_{1,k}$ and $\atwo \coloneqq  d_{2,k}$. 

Recall that $\CG(p,q)$ is the set of compatible gradings $\omega$ on $\calP = \calP([d_1]_+,[d_2]_+)$ with $p = \omega(\calP_\north)$ and $q = \omega(\calP_\east)$.  We let $\CG_+^t(p,q) \subseteq \CG(p,q)$ consist of the compatible gradings with $\omega(\calP_\north \cut \sh(\calP_\east)) = t$. Similarly, we let $\CG_-^t(p,q) \subseteq \CG(p,q)$ consist of the gradings with $\omega(\calP_\east \cut \sh(\calP_\north)) = t$.  Let 
$$\CG^t(p,q) = \begin{cases}\CG_+^t(p,q) & \text{ if $b\aone \geq a\atwo$},\\ \CG_-^t(p,q) & \text{ otherwise,}\end{cases}$$
and let $|\CG^t(p,q)| = \sum_{\omega\in \CG^t(ka,kb)} \text{wt}(\omega) \in \mathbb N[\mnd]$.  We show in the following lemma that $\CG^0(p,q)$ is precisely the set of shadowed gradings in $\CG(p,q)$.

\begin{lem}\label{lem: one of cgpm vanishes}
    One of $\CG^0_{\pm}(p, q)$ is empty unless $p = q = 0$. Precisely, when $(p, q)\neq (0, 0)$,
    \begin{enumerate}
        \item[\emph{(1)}] $\CG^0_+(p, q) = \varnothing$ if $q\aone \leq p\atwo$; \\
        \item[\emph{(2)}] $\CG^0_-(p, q) = \varnothing$ if $q\aone \geq p\atwo$.
    \end{enumerate}
\end{lem}
\begin{proof} 

By symmetry, it is enough to prove the first case, where we assume $\frac{q}{p} \leq \frac{d_2}{d_1}$.  Suppose such a compatible grading $\omega$ exists.   By assumption, we have $\frac{\omega(\calP_\east)}{\omega(\sh(\calP_\east))} = \frac{q}{p}$.  Hence there must be some horizontal edge $u$ such that $\frac{\omega(u)}{\omega(\sh(u))} \leq \frac{q}{p}$.  Choose $v$ to be the unique vertical edge such that $|(\overrightarrow{uv})_\north| = \omega(u)$.  

By the compatibility of $\omega$, we must have $|(\overrightarrow{uv})_\east| \geq \omega(\sh(u)) + 1$.   However, by the modular definition of the maximal Dyck path (see, for example, \cite[Lemma 1.3]{BLRS}), we then must have that $\omega(\sh(u)) d_2 < \omega(u) d_1$.  However, this yields a contradiction since we have $\frac{\omega(u)}{\omega(\sh(u))}  \leq \frac{q}{p} \leq \frac{d_2}{d_1}$.
\end{proof}

 Fix an endpoint $Q = (Q_1,Q_2) \in \R^2_{>0}$ such that the quantity $(-m+ka)Q_2 -(-n+kb)Q_1$ is positive if $b\aone \leq \atwo a$, and negative otherwise.  Let $\BL(ka,kb)$ denote the set of broken lines with initial exponent $(-\aone,-\atwo)$ and final exponent $(-\aone+ka, -\atwo+kb)$ that terminate at $Q$. Let $\BL^t(ka,kb)$ denote the subset of broken lines in $\BL(ka,kb)$ that bend at the positive $y$-axis with multiplicity $t$. Given a set $B$ of broken lines, let $|B| \coloneqq \sum_{\beta \in B} c(\beta)$.

\begin{lem}\label{lem: broken line to compatible grading refined equality}
  For all $t \in \{0,1,\dots,kb\}$, we have 
  $$|\BL^t(ka,kb)| = |\CG^t(ka,kb)|\,.$$
\end{lem}
\begin{proof}
We handle the case when $b\aone \geq a\atwo$, with the other case following analogously.  We proceed by induction on the quantity $kb- t$. The base case $kb - t = 0$ is readily checked, as both quantities are simply a product of two binomial coefficients. 

For the inductive step, we first note that the sets $\CG(ka,kb-t)$ and $\BL(ka,kb-t)$ can be refined as $\CG(ka,kb-t) = \bigcup_{s = 0}^{kb-t} \CG^s(ka,kb-t)$ and $\BL(ka,kb-t) = \bigcup_{s = 0}^{kb-t} \BL^s(ka,kb- t)$.  Since these unions are disjoint, we have
\begin{align*}
  |\CG(ka,kb- t)| & = \sum_{s=0}^{kb - t} |\CG^s(ka,kb- t)|, \\
  |\BL(ka,kb- t)| & = \sum_{s=0}^{kb - t} |\BL^s(ka,kb- t)|\,.
\end{align*}

Applying \Cref{thm: gen greedy equals theta} yields
$$\sum_{s=0}^{kb - t} |\CG^s(ka,kb- t)| = \sum_{s=0}^{kb - t} |\BL^s(ka,kb- t)|\,.$$
Using the inductive hypothesis, we have $|\CG^s(ka,kb- t)| = |\BL^s(ka,kb- t)|$ for $s > 0$.  Hence we can also conclude $|\CG^0(ka,kb- t)| = |\BL^0(ka,kb- t)|$.  By separating out the contribution of the edges in $\calP_N \cut \sh(\calP_\east)$ to the weight of each compatible grading, we see
$$|\CG^t(ka,kb)| = \pi_{d_1 - ka, t}(p_{2,\bullet})|\CG^0(ka,kb- t)|\,.$$
We also have 
$$|\BL^t(ka,kb)| = \pi_{d_1 - ka, t}(p_{2,\bullet})|\BL^0(ka,kb- t)|\,,$$
which can be seen examining the preimage of the map taking each broken line in $\BL^t(ka,kb)$ to a broken line in $\BL^0(ka,kb- t)$ by forgetting the last bending. We can therefore conclude $|\BL^t(ka,kb)| = |\CG^t(ka,kb)|$, as desired.
\end{proof}

\begin{prop}\label{prop: broken line weight single bend}
We have $|\BL^0(ka,kb)| = \lambda^{(m)}(ka,kb)$. 
\end{prop}
\begin{proof}
It follows directly from \Cref{cor: vectors one-bending}(2) and \Cref{rem: one-bending nonempty} (using our choice of $Q$) that $\BL^0(ka,kb)$ consists of a single broken line. This broken line, which we denote by $\beta$, bends with multiplicity $k$ at the wall of slope $\frac{b}{a}$ with no other bendings.  By hypothesis, we have $|(\aone ,\atwo ) \cdot (-b,a)| = m$. Thus, the weight associated to the bending of $\beta$ over $\frakd$ with multiplicity $k$ is given by the coefficient of $x^{kb}y^{ka}$ in $(f_{\frakd})^m$, which is precisely $\lambda(ka,kb)$. Since $\beta$ is the unique broken line in $\BL^0(ka,kb)$ and $\beta$ only bends once, we can conclude that $|\BL^0(ka,kb)| = |\beta| = \lambda^{(m)}(ka,kb)$, as desired.
\end{proof}

\begin{rem} In the cluster algebra case, a tight grading has weight $1$ if it takes values in $\{0,
\ell_1\}$ on the vertical edges and $\{0,\ell_2\}$ on the horizontal edges, and weight $0$ otherwise. If we restrict to considering only these gradings, then the total weight of gradings in $\CG^t(ka,kb)$ is the same as its cardinality. Thus, in order for $|\CG^t(ka,kb)|$ to be nonzero, $t$ must be a multiple of $\ell_1$.
\end{rem}

The first main theorem now readily follows.  

\begin{proof}[Proof of \Cref{thm: power shadow grading formula}]
By \Cref{lem: one of cgpm vanishes}, the gradings in $\CG^0(ka,kb)$ are precisely the shadowed gradings contributing to the formula in \Cref{thm: power shadow grading formula}. Hence the main result follows directly by combining the equalities in \Cref{lem: broken line to compatible grading refined equality} and \Cref{prop: broken line weight single bend} when $t = 0$.
\end{proof}

\subsection{Another proof of \Cref{thm: power shadow grading formula}}\label{subsec: second proof}

We provide an alternative proof that does not rely on the results in \Cref{sec: greedy = theta}.

We first prove some simple properties of shadowed gradings, allowing us to relate them to the the polynomials $d_+(p,q)$ and $d_-(p,q)$.

\begin{lem}\label{lem: recursion of cgpm}
    We have
    \begin{align*}
        |\CG(p,q)| &= \sum_{s=0}^p \pi_{\atwo - q, s}(p_{1, \bullet})\left|\CG^0_+(p-s, q)\right|\\
        &= \sum_{s=0}^q \pi_{\aone - p, s}(p_{2, \bullet})\left|\CG^0_-(p, q-s)\right|\,.
    \end{align*}
\end{lem}
\begin{proof} 
The first equality follows by viewing $s$ as total value of $\omega$ on the horizontal edges outside of $\sh(S_2)$, so $|\CG(p,q)| = \sum_{s = 0}^p |\CG^s_+(p-s,q)|$.  Each summand $\pi_{\atwo - q, s}(p_{1, \bullet})\left|\CG^0_+(p-s, q)\right|$ on the right hand side is then equal to $|\CG^s_+(p,q)|$, where $\pi_{\atwo - q, s}(p_{1, \bullet})$ accounts for the weight contributed by the horizontal edges outside $\sh(S_\north)$ and $|\CG^0_+(p-s, q)|$ accounts for the remaining weight of all vertical edges and the horizontal edges in $\sh(S_\north)$.  The second equality follows similarly, reversing the roles of vertical and horizontal edges.
\end{proof}

\begin{prop}\label{prop: cgpm = dpm}
    Fix $(\aone, \atwo)$. For any $(p, q)\in \mathbb N^2$, we have
    \[
        |\CG_+^0(p, q)| = d_+(p, q) \quad \text{and} \quad |\CG_-^0(p, q)| = d_-(p, q).
    \]
\end{prop}
\begin{proof} 
Recall $d_\pm(p, q)$ from \Cref{def: dpq}. We prove by induction on $(p, q)$. The base case is
\[
    |\CG_+^0(0, 0)| = d_+(0, 0) = 1 \quad \text{and} \quad |\CG_-^0(0, 0)| = d_-(0, 0) = 1.
\]
Then the result follows from \Cref{lem: recursion of cgpm}, \Cref{lem: cpq by dpq}, and $|\CG(p, q)| = c(p, q)$ from \Cref{thm: generalized greedy compatible gradings}.
\end{proof}

\begin{cor}\label{cor: recursion for dpm}
    The polynomials $d_\pm(p, q)$ satisfy
    \[
        \sum_{s=0}^p \pi_{\atwo-q, s}(p_{1, \bullet})d_+(p-s, q) = \sum_{s=0}^q \pi_{\aone-p, s}(p_{2, \bullet}) d_-(p, q-s).
    \]
    When $(p, q) \neq (0, 0)$, we have
    \begin{align*}
        d_+(p, q) = 0 \quad \text{when}\quad \aone q \leq \atwo p;\\
        d_-(p, q) = 0 \quad \text{when}\quad \aone q \geq \atwo p.
    \end{align*}
\end{cor}

\begin{proof}
    This is a direct corollary of \Cref{lem: one of cgpm vanishes}, \Cref{lem: recursion of cgpm}, and \Cref{prop: cgpm = dpm}.
\end{proof}

As a consequence of \Cref{prop: cgpm = dpm}, we present the following sharpened recursive formula of $c(p, q)$, which generalizes \cite[Proposition 1.6]{LLZ} and \cite[Proposition 2.13]{Rupgengreed}.

\begin{prop}\label{prop: sharpened recursion cpq}
  Fix $(d_1, d_2)\in \mathbb N^2$. The polynomials $c(p, q)$ satisfy the following recursion
  \[
    c(p, q) = \begin{dcases}
      \sum_{k=1}^p -c(p-k, q)\varsigma_{\atwo -q, k}(p_{1, \bullet}) \quad & \text{if $\aone q \leq \atwo p$;} \\
      \sum_{k=1}^q -c(p, q-k) \varsigma_{\aone -p, k}(p_{2, \bullet}) \quad & \text{if $\aone q \geq \atwo p$.}
    \end{dcases}
  \]
\end{prop}

\begin{proof}
    With the emphasis on $d_\pm(p, q)$, the stated recursion of $c(p, q)$ follows immediately from \Cref{cor: recursion for dpm} and \Cref{def: dpq}. We also see that $c(p, q)$ is attained by one of two polynomials in \Cref{def: cpq}.
\end{proof}

Fix $(d_1, d_2)\in \mathbb N^2$. Let $(p, q)\in \mathbb N^2$ and $t\in \mathbb N$. Let 
\[
    \BL^t_+(p, q) = \BL^t_+(\aone, \atwo, p, q) \quad \text{(resp. $\BL^t_-(p, q) = \BL^t_-(\aone, \atwo, p, q)$)}
\]
denote the set of broken lines with initial exponent $(-\aone, -\atwo)$, final exponent $(-\aone + p, -\atwo + q)$, and endpoint $Q_+\in \mathbb R_{>0}^2$ (resp. $Q_-\in \mathbb R_{>0}^2$) chosen close enough to the positive $x$-axis (resp. positive $y$-axis) such that the bend at the positive $x$-axis (resp. positive $y$-axis) corresponds to the monomial $x^t$ (resp. $y^t$).

\begin{lem}\label{lem: erase last bending}
    We have
    \begin{align*}
        |\BL_+^t(p, q)| & = \pi_{\atwo-q, t}(p_{1, \bullet}) |\BL_+^0(p-t, q)|,\\
        |\BL_-^t(p, q)| & = \pi_{\aone-p, t}(p_{2, \bullet}) |\BL_-^0(p, q-t)|.
    \end{align*}
\end{lem}

\begin{proof}
    When it is non-empty, erasing the last bending of any broken line in $\BL_+^t(p, q)$ (resp. $\BL_-^t(p, q)$) induces a surjection to $\BL_+^0(p-t, q)$ (resp. $\BL_-^0(p, q-t)$). The computation of the weights follows from the initial wall-functions on the axes.

    If $\BL_+^t(p, q)$ is empty, then either $\pi_{d_2-q, t}(p_{1, \bullet}) = 0$ or $\BL_+^0(p-t, q)$ is empty. The same is true for the $-$ case.
\end{proof}

\begin{lem}\label{lem: one of blpm vanishes}
    One of $\BL^0_{\pm}(p, q)$ is empty unless $p = q = 0$. Precisely, when $(p, q)\neq (0, 0)$,
    \begin{enumerate}
        \item[\emph{(1)}] $\BL^0_+(p, q) = \varnothing$ if $q\aone \leq \atwo p$, \\
        \item[\emph{(2)}] $\BL^0_-(p, q) = \varnothing$ if $q\aone \geq \atwo p$.
    \end{enumerate}
\end{lem}

\begin{proof}
    Since the broken lines in consideration do not bend at positive axes, the slopes of their linearity domains can only continue to increase or decrease from the start. Therefore, the order of the terminal slope and the initial slope determines which one of the two sets $\BL_+^0(p, q)$ and $\BL_-^0(p, q)$ is empty, unless the broken lines do not bend at all, which is the case $p = q = 0$.
\end{proof}

\begin{lem}\label{lem: recursion of blpm}
    \[
        \sum_{t=0}^p \pi_{\atwo - q, t}(p_{1, \bullet})|\BL^0_+(p-t, q)| =
        \sum_{t=0}^q \pi_{\aone - p, t}(p_{2, \bullet})|\BL^0_-(p, q-t)|.
    \]
\end{lem}

\begin{proof}
    By Lemma \ref{lem: erase last bending}, it is equivalent to show
    \[
        \sum_{t = 0}^q |\BL_+^t(p, q)| = \sum_{t = 0}^p |\BL_-^t(p, q)|.
    \]
    The left-hand side computes the coefficient of $x^{p - \aone}y^{q - \atwo}$ in the theta function $\vartheta_{Q_+, (-\aone, -\atwo)}$. The right-hand side does the same for $\vartheta_{Q_-, (-\aone, -\atwo)}$. By \Cref{cor: CPS consistency non strict}, they equal.
\end{proof}

\begin{prop}\label{prop: blpm = dpm}
    Fix $(\aone, \atwo)$. For any $(p, q)\in \mathbb N^2$, we have
    \[
        |\BL_+^0(p, q)| = d_+(p, q) \quad \text{and} \quad |\BL_-^0(p, q)| = d_-(p, q).
    \]
\end{prop}

\begin{proof}
    Compare the recursion of $|\BL_\pm^0(p, q)|$ determined by \Cref{lem: one of blpm vanishes} and \Cref{lem: recursion of blpm} with the recursion of $d_\pm(p, q)$ given in \Cref{cor: recursion for dpm}.
\end{proof}

Let $(\aone, \atwo)$ and $(ka, kb)$ satisfy the one-bending property \Cref{def: one-bending} with $\aone \geq ka$ and $\atwo \geq kb$.

\begin{proof}[Proof of \Cref{thm: power shadow grading formula}]
    We first deal with the case where $d_1>ka$ and $d_2>kb$. Suppose that
    \[
        m = \aone b - \atwo a > 0.
    \]
    By the one-bending assumption, every broken line $\beta$ in $\BL_+^0(ka, kb)$ is of the same bend-type, that is, the only bend is at the ray $\mathbb R_{\leq 0}(a, b)$ with multiplicity $m$. Notice that $\beta$ indeed reaches the first quadrant. The sum of weight monomials of every $\beta$ in $\BL_+^0(ka, kb)$ gives the coefficient $\lambda^{(m)}(ka, kb)$ of $x^{ka}y^{kb}$ in $(f_{\mathbb R_{\leq 0}(a, b)})^m$. Hence by \Cref{prop: blpm = dpm}
    \[
        \lambda^{(m)}(ka, kb) = |\BL_+^0(ka, kb)| = d_+(ka, kb).
    \]
    The situation where $m = \atwo a - \aone b > 0$ is similar.

    Suppose that $\aone = ka$ or $\atwo = kb$. A one-bending broken line $\beta$ at $\mathbb R_{\leq 0}(a, b)$ with initial exponent $(-\aone, -\atwo)$ and final exponent $(ka-\aone, kb-\atwo)$ does not reach the first quadrant. So it does not belong to $\BL_\pm^0(ka, kb)$. Suppose $\atwo = kb$. Then again
    \[
        m = \aone b - \atwo a >0.
    \]
    Consider the theta function $\vartheta_{Q, (-\aone, -\atwo)}$ for $Q$ chosen generally in the fourth quadrant. By the one-bending property, the sum of the weight of the aforementioned one-bending $\beta$ with endpoint $Q$ is the coefficient $\lambda^{(m)}(ka, kb)$. Meanwhile, it also computes the coefficient of $x^{ka-\aone}y^{kb-\atwo} = x^{ka-\aone}$ in $\vartheta_{Q, (-\aone, -\atwo)}$. By \Cref{cor: CPS consistency non strict}, $\lambda^{(m)}(ka, kb)$ equals the coefficient of $x^{ka-\aone}$ in $\vartheta_{Q_+, (-\aone, -\atwo)}$. The latter actually equals
    \[
        |\BL_+^0(ka, kb)| = d_+(ka, kb)
    \]
    because any broken line contributes to $x^{ka-\aone}$ must not bend at positive $x$-axis, thus belonging to $\BL_+^0(ka, kb)$. The situation when $\aone = ka$ is similar.

    Summarizing we have shown
    \[
        \lambda^{(m)}(ka, kb) = \begin{cases}
            d_+(ka, kb) \quad &\text{if}\ d_1b - d_2a > 0,\\
            d_-(ka, kb) \quad &\text{if}\ d_1b - d_2a < 0.
        \end{cases}
    \]
    Finally the polynomials $d_\pm(ka, kb)$ are related to weighted counts of shadowed gradings by \Cref{lem: one of cgpm vanishes} and \Cref{prop: cgpm = dpm}.
\end{proof}

\subsection{A universal scattering diagram}\label{subsec: universal sd rk 2}

Recall the monoid $\mndP_{\ell_1, \ell_2}$ in \Cref{subsec: gen cluster sd rk 2}. For $\ell_1'\geq \ell_1$ and $\ell_2'\geq \ell_2$, consider the algebra map
\[
    \varphi \colon \widehat{\Bbbk[\mndP_{\ell_1', \ell_2'}]} \rightarrow \widehat{\Bbbk[\mndP_{\ell_1, \ell_2}]}
\]
that evaluates $p_{1, i} = p_{2, j} = 0$ for $i>\ell_1$ and $j>\ell_2$.

For any $\ell\geq 1$, denote
\[
    P_1^{(\ell)} = 1 + p_{1, 1}x + \dots + p_{1, \ell}x^{\ell} \quad \text{and} \quad
    P_2^{(\ell)} = 1 + p_{2, 1}y + \dots + p_{2, \ell}y^{\ell}.
\]
By the functoriality (\ref{eq: functoriality complete}), we have
\[
    \varphi\left(\Scat\left(P_1^{(\ell_1')}, P_2^{(\ell_2')}\right)\right) \equiv \Scat\left(P_1^{(\ell_1)}, P_2^{(\ell_2)}\right).
\]
The scattering diagrams associated to all $(\ell_1, \ell_2)\in \mathbb N^2$ thus have a (projective) limit, viewed as a universal scattering diagram. By the tight grading formula \Cref{thm: power shadow grading formula}, the limiting wall-function on $\mathbb R_{\leq 0}(a, b)$ is given by
\begin{equation}
    \sum_{k\geq 0}\sum_{w(\mathbf Q_1)=ka}\sum_{w(\mathbf Q_2)=kb}\lambda_{a, b}(\mathbf Q_1, \mathbf Q_2)p_1^{\mathbf Q_1}p_2^{\mathbf Q_2} x^{ka} y^{kb},
\end{equation}
where $\lambda_{a, b}(\mathbf Q_1, \mathbf Q_2)$ is the number of tight gradings of weight polynomial $p_1^{\mathbf Q_1}p_2^{\mathbf Q_2}$. To recover the wall-function in $\Scat(P_1^{(\ell_1)}, P_2^{(\ell_2)})$, simply let the variables $p_{1, i}$, $i> \ell_1$ and $p_{2, j}$, $j>\ell_2$ be zero. This amounts to considering the subset of tight gradings $\omega$ whose value on horizontal edges is bounded above by $\ell_2$ and on vertical edges is bounded above by $\ell_1$.

We note that the limiting structure is more general than the scattering diagrams in \Cref{def: sd rk 2} as the monoid in the limit is not finitely generated. However, one can nonetheless assign degree $k$ to the variable $p_{i, k}$ and have a filtration of ideals $\frakm^{(k)}\subseteq \Bbbk[\mndP]$ generated by elements of degree at least $k$. Then we can still define scattering diagrams with finitely many walls over $\Bbbk[\mndP]/\frakm^{(k)}$ and a limiting object of those is regarded as a scattering diagram over the completion $\widehat{\Bbbk[\mndP]} = \lim\limits_{\longleftarrow}\Bbbk[\mndP]/\frakm^{(k)}$ just as in \Cref{def: sd rk 2}. Wall-crossing automorphisms and consistency can be discussed analogously.

\section{Applications of the tight and shadowed grading formulas}\label{sec: application}

Since rank-2 scattering diagrams are known to encode interesting invariants \cite{GPS, GP, Rei08, RW}, we explore these links in terms of tight and shadowed gradings in this section. We give formulas for computing Euler characteristics of quiver moduli spaces in \Cref{subsec: euler char quiver} and relative Gromov--Witten invariants on toric surfaces in \Cref{subsec: GW invariants}.

\subsection{Euler characteristics of quiver moduli}\label{subsec: euler char quiver}

We review a result of Reineke and Weist \cite{RW} expressing wall-function coefficients in terms of Euler characteristics of certain moduli spaces of quiver representations. Then we establish the link to tight gradings in \Cref{cor: euler char by shadowed grading}.

In \cite{RW} (as well as \cite{GPS, GP}), the coefficients in the initial wall-functions differ with those in \Cref{subsec: gen cluster sd rk 2}. Precisely, we consider polynomials
\begin{equation}\label{eq: product of binomials}
    P_1 = \prod_{i=1}^{\ell_1}(1+s_ix) \quad \text{and} \quad P_2 = \prod_{j=1}^{\ell_2}(1 + t_jy)
\end{equation}
and the scattering diagram $\Scat(P_1, P_2)$.
This setup can be accommodated in the framework of \Cref{subsec: sd in rk 2} and \Cref{subsec: gen cluster sd rk 2} by choosing the monoid 
$\mndP = M \oplus \mnd_1 \oplus \mnd_2$ where $\mnd_1\cong \mathbb N^{\ell_1}$ is the (multiplicative) monoid of monomials in $\{s_1, \dots, s_{\ell_1}\}$ and $\mnd_2\cong \mathbb N^{\ell_2}$ is the monoid of monomials in $\{t_1, \dots, t_{\ell_2}\}$. The coefficients $p_{1,k}$ and $p_{2,k}$ used in \Cref{subsec: gen cluster sd rk 2} are elementary symmetric polynomials respectively in $s_i$ and in $t_j$ by comparing (\ref{eq: product of binomials}) with (\ref{eq: initial function rk 2}).

Let $K = K(\ell_1, \ell_2)$ be the complete bipartite quiver with $\ell_1$ sinks $\{s_1, \dots, s_{\ell_1}\}$ and $\ell_2$ sources $\{t_1, \dots, t_{\ell_2}\}$. For example, $K(3, 2)$ is depicted in \Cref{fig: K32}.
\begin{figure}[ht]
    \begin{tikzcd}[row sep = tiny]
    	{s_1} \\
    	&& {t_1} \\
    	{s_2} \\
    	&& {t_2} \\
    	{s_3}
    	\arrow[from=2-3, to=1-1]
    	\arrow[from=2-3, to=3-1]
    	\arrow[from=2-3, to=5-1]
    	\arrow[from=4-3, to=1-1]
    	\arrow[from=4-3, to=3-1]
    	\arrow[from=4-3, to=5-1]
    \end{tikzcd}
    \caption{The quiver $K(3,2)$}
    \label{fig: K32}
\end{figure}
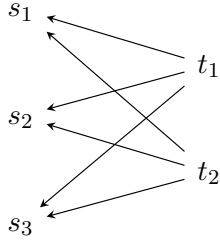

Let $\mathbf P_1 = (d_{1, i})_{1\leq i\leq \ell_1}$ and $\mathbf P_2 = (d_{2, j})_{1\leq j \leq \ell_2}$ be tuples of nonnegative integers. They will be referred to as \emph{ordered partitions} once we fix the sums $|\mathbf P_1| = \sum d_{1, i}$ and $|\mathbf P_2| = \sum d_{2, j}$. As the notation suggests, $\mathbf P_1$ and $\mathbf P_2$ represent two monomials
\[
    s^{\mathbf P_1}\coloneqq \prod_{i=1}^{\ell_1} s_i^{d_{1,i}}\in \mnd_1 \quad \text{and} \quad t^{\mathbf P_2} \coloneqq \prod_{j=1}^{\ell_2} t_j^{d_{2,j}}\in \mnd_2,
\]
Meanwhile they also represent a dimension vector $(\mathbf P_1, \mathbf P_2)$ on the quiver $K$ such that the dimension on the sink $s_i$ is $d_{1, i}$ and on the source $t_j$ is $d_{2, j}$.

Let $V_i = \mathbb C^{d_{1, i}}$ for $i = 1, \dots, \ell_1$ and $W_j = \mathbb C^{d_{2, j}}$ for $j = 1, \dots \ell_2$. Consider the following affine space of representations of $K$ (in $\mathbb C$-vector spaces)
\[
  \operatorname{rep}(\mathbf P_1, \mathbf P_2) \coloneqq \{(f_{i,j})_{1\leq i \leq \ell_1, 1\leq j\leq \ell_2}\mid f_{i, j}\in \operatorname{Hom}_{\mathbb C}(W_j, V_i)\}.
\]

Let $G = \prod_{i=1}^{\ell_1}\operatorname{GL}(V_i) \times \prod_{j=1}^{\ell_2}\operatorname{GL}(W_j)$. Then $G$ acts on the two spaces 
\[
  \operatorname{rep}^\mathrm{f}(\mathbf P_1, \mathbf P_2) \coloneqq \operatorname{rep}(\mathbf P_1, \mathbf P_2) \times \prod_{i=1}^{\ell_1} V_i
  \quad \text{and} \quad
  \operatorname{rep}^\mathrm{b}(\mathbf P_1, \mathbf P_2) \coloneqq \operatorname{rep}(\mathbf P_1, \mathbf P_2) \times \prod_{j=1}^{\ell_2} W_j
\]
respectively by
\[
  (g_i, h_j) \cdot (f_{i,j}, v_i) = (g_if_{i,j}h_j^{-1}, g_iv_i) 
  \quad \text{and} \quad 
  (g_i, h_j) \cdot (f_{i,j}, w_j) = (g_if_{i,j}h_j^{-1}, h_jw_j).
\]
The superscripts $^\mathrm{f}$ and $^\mathrm{b}$ stand for respectively \emph{front} and \emph{back framing}, where the vectors $v_i$ and $w_j$ are called \emph{framing vectors}. Elements in $\operatorname{rep}^\bullet(\mathbf P_1, \mathbf P_2)$ are called \emph{framed representations}.

We further consider back or front framed representations that are \emph{stable} in the following sense. We call $(f_{i,j}, w_j)\in \operatorname{rep}^\mathrm{b}(\mathbf P_1, \mathbf P_2)$ (resp. $(f_{i,j}, v_i)\in \operatorname{rep}^\mathrm{f}(\mathbf P_1, \mathbf P_2)$) \emph{stable} if for every tuple $U_j\subseteq W_j$ of subspaces, we have
\[
  \sum_{i=1}^{\ell_1} \sum_{j=1}^{\ell_2} \dim f_{i, j}(U_j) \geq \frac{|\mathbf P_1|}{|\mathbf P_2|} \sum_{j=1}^{\ell_2} U_j,
\]
with strictly inequality if $w_j\in U_j$ (resp. $v_i\in \sum_{j=1}^{\ell_2} f_{i,j}(U_j)$) for every $j$ (resp. $i$). It is clear that framed stable representations are closed under the action of $G$.

By \cite[Corollary 5.4]{RW} (based on the construction of \cite{ER}), there are smooth irreducible varieties $\mathcal M^\mathrm{b}(\mathbf P_1, \mathbf P_2)$ and $\mathcal M^\mathrm{f}(\mathbf P_1, \mathbf P_2)$ as moduli spaces parametrizing respectively back and front framed stable representations up to the action of $G$. We use $\chi(\cdot)$ to denote the Euler characteristic of a complex variety in analytic topology.

\begin{thm}[{\cite[Theorem 6.1]{RW}}]\label{thm: RW}
  For any coprime $(a, b)\in \mathbb N^2$, we have
  \begin{align*}
    \left(f_{\mathbb R_{\leq 0}(a, b)}\right)^b & = \sum_{k\geq 0} \sum_{|\mathbf P_1| = ka}\sum_{|\mathbf P_2| = kb} \chi(\mathcal M^\mathrm{b}(\mathbf P_1, \mathbf P_2))s^{\mathbf P_1}t^{\mathbf P_2} x^{ka}y^{kb}, \\
    \left(f_{\mathbb R_{\leq 0}(a, b)}\right)^a & = \sum_{k\geq 0} \sum_{|\mathbf P_1| = ka}\sum_{|\mathbf P_2| = kb} \chi(\mathcal M^\mathrm{f}(\mathbf P_1, \mathbf P_2))s^{\mathbf P_1}t^{\mathbf P_2} x^{ka}y^{kb} .
  \end{align*}
\end{thm}

For coprime $(a, b)\in \mathbb N^2$ and $k\geq 1$, consider $(\aone, \atwo) = (kab+1, kb^2)$. By \Cref{cor: vectors one-bending}(2), the vectors $(\aone, \atwo)$ and $(ka, kb)$ satisfy the one-bending property (\Cref{def: one-bending}). Notice that $a\cdot kb^2 - b\cdot (kab + 1) = -b$.
Recall from \Cref{subsec: formula wall function} that by \Cref{thm: power shadow grading formula}, $\lambda^{(b)}_{a, b}(\mathbf Q_1, \mathbf Q_2)$ is the number of shadowed gradings on $\mathcal P(kab + 1, kb^2)$ of weight monomial $p_1^{\mathbf Q_1}p_2^{\mathbf Q_2}$ where $w(\mathbf Q_1) = ka$ and $w(\mathbf Q_2) = kb$. Similarly, $\lambda^{(a)}_{a, b}(\mathbf Q_1, \mathbf Q_2)$ is the number of shadowed gradings on $\mathcal P(ka^2, kab+1)$ of weight monomial $p_1^{\mathbf Q_1}p_2^{\mathbf Q_2}$.

We now write down explicit formulas of Euler characteristics of framed stable quiver moduli in terms of counts of shadowed gradings.

\begin{cor}\label{cor: euler char by shadowed grading}
    For any coprime $(a, b)\in \mathbb N^2$, ordered partitions $|\mathbf P_1| = ka$, $|\mathbf P_2| = kb$, we have
      \begin{align*}
        \chi(\mathcal M^\mathrm{b}(\mathbf P_1, \mathbf P_2)) & = \sum_{w(\mathbf Q_1)=ka}\sum_{w(\mathbf Q_2)=kb} \lambda^{(b)}_{a, b}(\mathbf Q_1, \mathbf Q_2) \mu^{\mathbf Q_1}_{\mathbf P_1} \mu^{\mathbf Q_2}_{\mathbf P_2},\\
        \chi(\mathcal M^\mathrm{f}(\mathbf P_1, \mathbf P_2)) & = \sum_{w(\mathbf Q_1)=ka}\sum_{w(\mathbf Q_2)=kb} \lambda^{(a)}_{a, b}(\mathbf Q_1, \mathbf Q_2) \mu^{\mathbf Q_1}_{\mathbf P_1} \mu^{\mathbf Q_2}_{\mathbf P_2},
      \end{align*}
      where $\mu_{\mathbf P_i}^{\mathbf Q_i}$ for $i=1,2$ are nonnegative integers defined by the expansions
    \begin{equation}\label{eq: expand p with s t}
      p_1^{\mathbf Q_1} = \sum_{|\mathbf P_1| = w(\mathbf Q_1)} \mu^{\mathbf Q_1}_{\mathbf P_1} s^{\mathbf P_1} \quad \text{and} \quad
      p_2^{\mathbf Q_2} = \sum_{|\mathbf P_2| = w(\mathbf Q_2)} \mu^{\mathbf Q_2}_{\mathbf P_2} t^{\mathbf P_2}.
    \end{equation}
\end{cor}

\begin{proof}
  By \Cref{thm: power shadow grading formula}, we can express $\left(f_{\mathbb R_{\leq 0}(a, b)}\right)^b$ as
  \begin{align*}
      & \sum_{k\geq 0}\sum_{w(\mathbf Q_1)=ka}\sum_{w(\mathbf Q_2)=kb}\lambda^{(b)}_{a, b}(\mathbf Q_1, \mathbf Q_2)p_1^{\mathbf Q_1}p_2^{\mathbf Q_2} x^{ka} y^{kb}.  \\
      = & \sum_{k\geq 0}\sum_{\substack{|\mathbf P_1|=ka\\ |\mathbf P_2|=kb}}\left(\sum_{\substack{w(\mathbf Q_1)=ka\\ w(\mathbf Q_2)=kb}}\lambda^{(b)}_{a, b}(\mathbf Q_1, \mathbf Q_2) \mu^{\mathbf Q_1}_{\mathbf P_1} \mu^{\mathbf Q_2}_{\mathbf P_2} \right) s^{\mathbf P_1}t^{\mathbf P_2} x^{ka} y^{kb}.
  \end{align*}
  Comparing the above coefficients with the ones in \Cref{thm: RW} we have the desired equation. The other case is similar.
\end{proof}

\subsection{Formula for relative Gromov--Witten invariants}\label{subsec: GW invariants}

In \cite{GPS}, Gross, Pandharipande, and Siebert expressed the wall-function $f_{\mathbb R_{\leq 0}(a, b)}$ in terms of relative Gromov--Witten invariants on toric surfaces. We briefly recall their results and give a combinatorial formula for these invariants in \Cref{thm: gw by shadowed grading}.

For $(a, b)$ positive and coprime, let $X = X_{a, b}$ denote the toric surface over $\mathbb C$ associated to the complete fan generated by rays $\mathbb R_{\geq 0}(0, 1)$, $\mathbb R_{\geq 0}(1, 0)$, and $\mathbb R_{\leq 0}(a, b)$. The three rays correspond to three toric divisors in $X$ which we denote as $D_1$, $D_2$, and $D_\mathrm{out}$.

Fix two positive integers $\ell_1$ and $\ell_2$. As in \Cref{subsec: euler char quiver}, let $\mathbf P_1 = (d_{1, i})_{1\leq i\leq \ell_1}$ and $\mathbf P_2 = (d_{2, j})_{1\leq j\leq \ell_2}$ be ordered partitions such that
\[
  |\mathbf P_1| = \sum_{i=1}^{\ell_1} d_{1, i} = ka \quad \text{and} \quad |\mathbf P_2| = \sum_{j=1}^{\ell_2} d_{2, j} = kb
\]
for some $k\geq 1$.
Consider $\nu \colon X[(\mathbf P_1, \mathbf P_2)] \rightarrow X$ the blow-up of $X$ along $\ell_1$ and $\ell_2$ general points respectively of $D_1$ and $D_2$ away from the $0$-dimensional toric stratum $S$ (i.e., three toric fixed points). Denote
\[
  X^\circ[(\mathbf P_1, \mathbf P_2)] = X[(\mathbf P_1, \mathbf P_2)] \cut \nu^{-1}(S)
  \quad \text{and} \quad 
  D_\mathrm{out}^\circ = D_\mathrm{out} \cut \nu^{-1}(S)\,.
\]

Let $\beta\in H_2(X, \mathbb Z)$ be the homology class determined by the intersection numbers
\[
  \beta \cdot D_1 = ak, \quad \beta \cdot D_2 = bk, \quad \beta \cdot D_\mathrm{out} = k.
\]
Define the homology class
\[
  \beta[(\mathbf P_1, \mathbf P_2)] = \nu^*(\beta) - \sum_{i=1}^{\ell_1} d_{1,i}[E_{1, i}] - \sum_{j=1}^{\ell_2} d_{2, j}[E_{2, j}] \in H_2(X[(\mathbf P_1, \mathbf P_2)], \mathbb Z).
\]
where $E_{1, i}$ and $E_{2, j}$ are the exceptional divisors over the blown-up points on $D_1$ and $D_2$.

There is a moduli space $\mathfrak M$ of relative stable maps of genus 0 curves into $X^\circ[(\mathbf P_1, \mathbf P_2)]$ representing the homology class $\beta[(\mathbf P_1, \mathbf P_2)]$ with full tangency order $k$ at an unspecified point in $D_\mathrm{out}^\circ$. The moduli space $\mathfrak M$ is proper of virtual dimension zero and thus the integration over the virtual fundamental class defines a Gromov--Witten type invariant
\[
  N_{a, b}[(\mathbf P_1, \mathbf P_2)] \coloneqq \int_{\mathfrak M^\mathrm{vir}} 1 \in \mathbb Q.
\]
We refer the reader to \cite[Section 4]{GPS} for more details on the construction of the moduli space and the definition of $N_{a, b}[(\mathbf P_1, \mathbf P_2)]$.

\begin{thm}[{\cite[Theorem 5.4]{GPS}}]\label{thm: GPS}
  For any positive coprime $(a, b)$, we have
  \begin{equation}\label{eq: GPS}
    \log f_{\mathbb R_{\leq 0}(a, b)} = \sum_{k\geq 1}\sum_{\substack{|\mathbf P_1| = ka\\ |\mathbf P_2| = kb}} k N_{a, b}[(\mathbf P_1, \mathbf P_2)] s^{\mathbf P_1}t^{\mathbf P_2} x^{ka} y^{kb}.
  \end{equation}
\end{thm}

Combined the above theorem of Gross--Pandharipande--Siebert with the tight grading formula \Cref{thm: power shadow grading formula}, we have the following direct corollary.

\begin{cor}\label{cor: log tight grading GW}
  We have for any $m\geq 1$,
  \begin{align*}
      &\sum_{k\geq 1}\sum_{\substack{|\mathbf P_1| = ka\\ |\mathbf P_2| = kb}} k N_{a, b}[(\mathbf P_1, \mathbf P_2)] s^{\mathbf P_1}t^{\mathbf P_2} x^{ka} y^{kb} \\
       &\quad=  \frac{1}{m} \log \left(1 + \sum_{k\geq 1}\sum_{\substack{w(\mathbf Q_1)=ka\\ w(\mathbf Q_2) = kb}} \lambda^{(m)}_{a,b}(\mathbf Q_1, \mathbf Q_2) p_1^{\mathbf Q_1} p_2^{\mathbf Q_2} x^{ka}y^{kb} \right)\,.
  \end{align*}
\end{cor}

Next we present a formula for each single $N_{a, b}[(\mathbf P_1, \mathbf P_2)]$ that is linear on counts of compatible gradings without involving taking logarithms.

\begin{thm}\label{thm: gw by shadowed grading}
  For any $(a, b)$ coprime, $k\geq 1$ and ordered partitions $|\mathbf P_1| = ka$ and $|\mathbf P_2| = kb$, we have
  \[
    N_{a, b}[(\mathbf P_1, \mathbf P_2)] = \frac{1}{k}\sum_{i=1}^{k} \frac{(-1)^{i-1}}{i}\binom{k}{i} \sum_{\substack{w(\mathbf Q_1) = ka \\ w(\mathbf Q_2) = kb}} \lambda^{(i)}_{a, b}(\mathbf Q_1, \mathbf Q_2)\mu_{\mathbf P_1}^{\mathbf Q_1}\mu_{\mathbf P_2}^{\mathbf Q_2}.
  \]
\end{thm}

\begin{proof}
    Denote $f_{\mathbb R_{\leq 0}(a, b)} = 1 + Y$. In the infinite series
    \[
        \log f_{\mathbb R_{\leq 0}(a, b)} = Y - \frac{1}{2}Y^2 + \dots + \frac{(-1)^{k-1}}{k} Y^{k} + \cdots
    \]
    every term with $x^{ka}y^{kb}$ is contained in the part
    \begin{equation}\label{eq: generate cik}
        Y - \frac{1}{2}Y^2 + \dots + \frac{(-1)^{k-1}}{k} Y^{k} = c_{0, k} + c_{1, k}(1+Y) + c_{2, k} (1+Y)^2 + \dots + c_{k, k} (1+Y)^{k}.
    \end{equation}
  By \Cref{thm: power shadow grading formula}, combining the terms with $x^{ka}y^{kb}$ we obtain
  \[
    \sum_{i=1}^k c_{i, k}\sum_{\substack{w(\mathbf Q_1) = ka \\ w(\mathbf Q_2) = kb}} \lambda^{(i)}_{a, b}(\mathbf Q_1, \mathbf Q_2) p_1^{\mathbf Q_1} p_2^{\mathbf Q_2} x^{ka} y^{kb}
  \]
    where each $\lambda_{a, b}^{(i)}(\mathbf Q_1, \mathbf Q_2)$ counts shadowed gradings. Expand this sum using the expression (\ref{eq: expand p with s t}) of $p_i^{\mathbf Q_i}$ in terms of $s^{\mathbf P_1}$ and $t^{\mathbf P_2}$ and then compare with the right-hand side of (\ref{eq: GPS}). We see that $N_{(a, b)}[(\mathbf P_1, \mathbf P_2)]$ is expressed as desired modulo that $c_{i, k} = \frac{(-1)^{i-1}}{i}\binom{k}{i}$. To compute these numbers, we write each term on the left-hand side of (\ref{eq: generate cik}) as
    \[
        \frac{(-1)^{j-1}}{j} Y^{j} = \frac{(-1)^{j-1}}{j} (1+Y - 1)^{j} = \sum_{i=0}^{j} \frac{(-1)^{i-1}}{j} \binom{j}{i} (1 + Y)^{i}.
    \]
    Then we have
    \[
        c_{i, k} = \sum_{j=i}^k \frac{(-1)^{i-1}}{j}\binom{j}{i} = \sum_{j=i}^{k} \frac{(-1)^{i-1}}{i}\binom{j-1}{i-1} = \frac{(-1)^{i-1}}{i} \binom{k}{i}.
    \]
\end{proof}

In particular, when $k=1$, we have
\[
  N_{a, b}[(\mathbf P_1, \mathbf P_2)] = \sum_{\substack{w(\mathbf Q_1)=a\\ w(\mathbf Q_2) = b}} \lambda_{a, b}^{(1)}(\mathbf Q_1, \mathbf Q_2) \mu^{\mathbf Q_1}_{\mathbf P_1} \mu^{\mathbf Q_2}_{\mathbf P_2},
\]
where $\lambda_{a, b}^{(1)}$ are honest coefficients of $f_{\mathbb R_{\leq 0}(a, b)}$.

\begin{exmp}
  Let $\ell_1 = 3$ and $\ell_2 = 2$. Consider the ordered partitions
  \[
    \mathbf P_1 = 3 + 0 + 0 \quad \text{and} \quad \mathbf P_2 = 3 + 0.
  \]
  They correspond respectively to the monomials $s_1^3$ and $t_1^3$. Only when $\mathbf Q_1 = p_{1,1}^3$ and $\mathbf Q_2 = p_{2,1}^3$, we have non-zero coefficients
  \[
    \mu_{\mathbf P_1}^{\mathbf Q_1} = 1 \quad \text{and} \quad \mu_{\mathbf P_2}^{\mathbf Q_2} = 1.
  \]
  Then the formula in \Cref{thm: gw by shadowed grading} gives
  \[
    N_{1,1}[(3+0, 3+0+0)] = \frac{1}{3}\left(3\lambda^{(1)}_{1,1}(p_{1,1}^3, p_{2,1}^3) - \frac{3}{2}\lambda_{1,1}^{(2)}(p_{1,1}^3, p_{2,1}^3) + \frac{1}{3}\lambda^{(3)}_{1,1}(p_{1,1}^3, p_{2,1}^3)\right).
  \]
    Using compatible gradings, we can see that only the third term is non-trivial
  \[
    \lambda^{(3)}_{1,1}(p_{1,1}^3, p_{2,1}^3) = 1. 
  \]
  For this particular grading (depicted in \Cref{fig: tight grading multiple cover}), one can choose $(\aone ,\atwo ) = (12, 9)$, which with $(ka, kb) = (3, 3)$ satisfy the one-bending property by \Cref{cor: vectors one-bending}. Notice $\aone b-\atwo a = 3$.
  Therefore, we have computed $N_{1,1}[(3+0, 3+0+0)] = 1/9$, which coincides with the multiple cover formula \cite{BP05} and \cite[Proposition 5.2]{GPS}.
  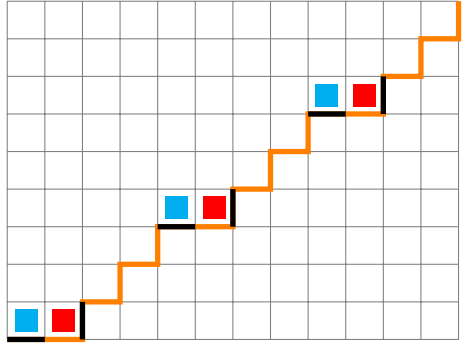
\begin{figure}[ht]
    \centering
    \begin{tikzpicture}[scale=.5]
    \draw[step=1,color=gray] (0,0) grid (12,9);
    \vrec{0}{0}{1}{1}
    \vrec{4}{3}{1}{1}
    \vrec{8}{6}{1}{1}
    \hrec{2}{0}{1}{1}
    \hrec{6}{3}{1}{1}
    \hrec{10}{6}{1}{1}
    \draw[line width=2,color=orange] 
    (0,0)--(2,0)--(2,1)--(3,1)--(3,2)--(4,2)--(4,3)--(6,3)--(6,4)--(7,4)--(7,5)--(8,5)--(8,6)--(10,6)--(10,7)--(11,7)--(11,8)--(12,8)--(12,9);
    \draw[line width=2pt] (0,0)--(1,0);
    \draw[line width=2pt] (4,3)--(5,3);
    \draw[line width=2pt] (8,6)--(9,6);
    \draw[line width=2pt] (2,0)--(2,1);
    \draw[line width=2pt] (6,3)--(6,4);
    \draw[line width=2pt] (10,6)--(10,7);
    \end{tikzpicture}
    \caption{A shadowed grading of weight monomial $p_{1,1}^3p_{2,1}^3$}
    \label{fig: tight grading multiple cover}
    \end{figure}
\end{exmp}

\section{Generalized cluster structures}\label{sec: gen cluster algebra}

\subsection{Generalized cluster structures}\label{subsec: gen cluster struct}

Fomin and Zelevinsky's cluster algebra \cite{FZ} has been generalized by Chekhov and Shapiro \cite{CS} to incorporate polynomial exchange relations (as opposed to binomials). We review the basic notions here.

Let $\mathbb P$ be a torsion-free (multiplicative) abelian group, known as the \emph{coefficient group} \cite{FZ,BFZ}. Let $n$ be a positive integer and $I$ denote the index set $\{1, \dots, n\}$. We take a field $\mathcal F$ isomorphic to the field of rational functions in $n$ independent variables over $\operatorname{Frac}(\mathbb {ZP})$, the fraction field of $\mathbb {ZP}$ the group ring of $\mathbb P$.

\begin{defn}\label{def: gen seed}
    A \emph{seed} in $\mathcal F$ is a triple $\Sigma = (\mathbf x, \mathbf p, B)$, where
    \begin{itemize}
        \item the \emph{cluster} $\mathbf x = \{x_1, \dots, x_n\}$ is an algebraically independent subset that generates $\mathcal F$ over $\operatorname{Frac}(\mathbb {ZP})$;
        \item the \emph{exchange matrix} $B = (b_{i,j})$ is a skew-symmetrizable $n\times n$ integer matrix;
        \item the \emph{coefficient tuples} $\mathbf p = (\mathbf p_1, \dots, \mathbf p_n)$ consist of, for each $i\in I$, an $(\ell_i+1)$-tuple $\mathbf p_i = (p_{i,0}, \dots, p_{i, \ell_i})$ of elements in $\mathbb P$ for $\ell_i = \gcd(b_{j,i}\mid j\in I)$.
    \end{itemize}
\end{defn}

For the above data, we refer to any $x_i\in \mathbf x$ a \emph{cluster variable} and a monomial in $\mathbf x$ a \emph{cluster monomial}. For $k\in I$, we define $x_k'\in \mathcal F$ by the \emph{exchange relation}
\begin{equation}\label{eq: gen exchange reln}
    x_kx_k' = P_k(\mathbf x) = \sum_{s=0}^{\ell_k}p_{k,s} \prod_{b_{i,k}>0}x_i^{sb_{i,k}/\ell_k}\prod_{b_{j,k}<0}x_j^{(s-\ell_k)b_{j,k}/\ell_k}.
\end{equation}
The set $\mathbf x - \{x_k\} \cup \{x_k'\}$ is clearly again an algebraically independent generating set for $\mathcal F$.

\begin{defn}\label{def: gen seed mutation}
    We say that a seed $\Sigma' = (\mathbf x', \mathbf p', B')$ in $\mathcal F$ is obtained from $\Sigma$ by a \emph{seed mutation} in direction $k\in I$ if
    \begin{itemize}
        \item $\mathbf x' = \mathbf x - \{x_k\} \cup \{x_k'\}$ where $x_k'$ is defined as in (\ref{eq: gen exchange reln});
        \item the tuples $\mathbf p'_i = (p'_{i,0}, \dots, p'_{i, \ell_i})$ satisfy when $i = k$
        \[
            p'_{k, s} = p_{k, \ell_k-s}, \quad \text{for $s = 0, \dots, \ell_k$}
        \]
        and when $i\neq k$, for $s = 0, \dots, \ell_i$
        \begin{equation}\label{eq: coefficient mutation not k}
            \frac{p'_{i,s}}{p'_{i, 0}} = \begin{dcases}
                p_{k, \ell_k}^{sb_{k,i}/\ell_i} \frac{p_{i,s}}{p_{i,0}} \quad & \text{if $b_{k,i}\geq 0$}, \\
                p_{k, 0}^{sb_{k,i}/\ell_i} \frac{p_{i,s}}{p_{i,0}} \quad & \text{if $b_{k,i}\leq  0$};
            \end{dcases}
        \end{equation}
        \item $B' = (b_{i,j}')$ is the \emph{Fomin--Zelevinsky matrix mutation} of $B$ in direction $k$ (denoted as $B' = \mu_k(B)$) given by
        \begin{equation}\label{eq: matrix mutation}
            b_{i,j}' = \begin{cases}
                -b_{i,j} \quad & \text{if $i = k$ or $j = k$}, \\
                b_{i,j} + [-b_{i,k}]_+b_{k,j} + b_{i,k}[b_{k,j}]_+ \quad & \text{otherwise}.
            \end{cases}
        \end{equation}
    \end{itemize}
\end{defn}

It is straightforward to check that a seed mutation is involutive, that is, in the situation of \Cref{def: gen seed mutation}, the seed $\Sigma$ is also obtained from $\Sigma'$ by a seed mutation in direction $k$. It is worth mentioning that a seed mutation is non-deterministic. For example one sees in (\ref{eq: coefficient mutation not k}) that uniformly multiplying an element in $\mathbb P$ to every $p'_{i; s}$, $s=0, \dots, \ell_i$ does not affect the equality.

Let $\mathbb T_n$ be the $n$-regular (infinite) tree such that the $n$ edges incident to any vertex are distinctively labeled by $\{1, \dots, n\}$.

\begin{defn}\label{def: cluster structure}
    A \emph{generalized cluster structure} in $\mathcal F$ is an association of seeds
    \[
        v \mapsto \Sigma_v,\quad v\in \mathbb T_n,
    \]
    such that for any $k$-labeled edge $v\frac{k}{\quad\quad}v'$ in $\mathbb T_n$, the seeds $\Sigma_{v}$ and $\Sigma_{v'}$ are obtained from each other by a seed mutation in direction $k$. When $\ell_i = 1$ for every $i\in I$, this is what we refer to as Fomin--Zelevinsky's \emph{ordinary cluster structure} \cite{FZ,BFZ}.
\end{defn}

For a seed $\Sigma_v = (\mathbf x_v, \mathbf p_v, B_v)$ associated with $v\in \mathbb T_v$, we often use the following notations to denote the gadgets in \Cref{def: gen seed}:
\begin{equation}\label{eq: notation seed pattern}
    \mathbf x_v = \{x_{1;v}, \dots, x_{n;v}\},\quad \mathbf p_v = (\mathbf p_{i;v})_{i=1}^{n},\quad \mathbf p_{i;v} = (p_{i,s}^v)_{s=0}^{\ell_i},\quad B_v = (b^{v}_{i,j}).
\end{equation}

The first property one would hope a generalized cluster structure to inherit from that of ordinary cluster structures is the Laurent phenomenon \cite[Theorem 3.1]{FZ}. This is indeed the case.

\begin{thm}[{\cite[Theorem 2.5]{CS}}]\label{thm: gen laurent CS}
    In a generalized cluster structure, any cluster variable is expressed in terms of any given cluster as a Laurent polynomial with coefficients in $\mathbb{ZP}$, that is, for any $v$ and $v'$ in $\mathbb T_n$ and any $i\in I$ we have
    \[
        x_{i;v'}\in \mathbb {ZP}[x_{1;v}^{\pm1}, \dots, x_{n;v}^{\pm 1}].
    \]
\end{thm}

The following is a main result of this paper, the positivity of Laurent phenomenon. Its proof will be given in \Cref{subsec: theta and positive}.

\begin{thm}\label{thm: positive Laurent}
    Any cluster variable, when expressed as a Laurent polynomial in any given cluster, has coefficients in $\mathbb {NP}$, the subset in $\mathbb {ZP}$ of linear combinations in $\mathbb P$ with nonnegative integer coefficients.
\end{thm}

\begin{rem}
    In the case of ordinary cluster structures, the positivity was conjectured originally by Fomin and Zelevinsky in \cite{FZ}, which was first proven by Lee and Schiffler \cite{LSpositive} when $B$ is skew-symmetric and then by Gross, Hacking, Keel, and Kontsevich \cite{GHKK} when $B$ is in general skew-symmetrizable. A less general version of \Cref{thm: positive Laurent} was conjectured by Chekhov and Shapiro \cite[Section 5]{CS} where the coefficient tuples are required to be reciprocal.
\end{rem}

\subsection{Frozen directions and coefficients}

In this subsection, every cluster structure is by default generalized. The cluster structure introduced so far can be slightly upgraded by introducing \emph{frozen directions} of mutations. This amounts to choosing a subset $I_\mathrm{uf}\subseteq I$ for \emph{unfrozen directions} and forbidding mutations in \emph{frozen directions} $I_\mathrm{fr} = I\setminus I_\mathrm{uf}$. By permuting the indices, we can always assume
\[
    I_\mathrm{uf} = \{1, \dots, |I_\mathrm{uf}|\} \subseteq I.
\]
So a cluster structure with frozen directions will be defined on the tree $\mathbb T_{|I_\mathrm{uf}|}$. Only the first $|I_\mathrm{uf}|$ columns of the exchange matrices will be relevant in exchange relations (\ref{eq: gen exchange reln}) and matrix mutations (\ref{eq: matrix mutation}).

Since no mutation is performed in frozen directions, all clusters share the same set of \emph{frozen cluster variables}
\[
    \{x_i\mid i\in I_\mathrm{fr}\}.
\]
Notice that in an exchange polynomial $P_k(\mathbf x)$ we can single out a monomial in frozen variables, combine it with $p_{k,s}$, and treat them together as a coefficient. The new coefficients
\[
    p_{k,s}\prod_{\{i\in I_\mathrm{fr}\mid b_{i,k}>0\}} x_i^{sb_{i,k}/\ell_k}
    \prod_{\{j\in I_\mathrm{fr}\mid b_{j,k}<0\}} x_i^{(s-\ell_k)b_{j,k}/\ell_k}
\]
follow the same mutation rule in \Cref{def: gen seed mutation}. Hence a cluster structure with frozen directions can be regarded as a special case of \Cref{def: cluster structure} with the square exchange matrices $(b_{i,j})_{i,j\in I_\mathrm{uf}}$. One nuance is to invert frozen variables so that the new coefficients naturally live in the product of $\mathbb P$ and the (multiplicative) group of all Laurent monomials in frozen variables.

\begin{rem}
    Starting with a cluster structure without frozen directions, we can add frozen directions by extending a square exchange matrix $B$ by adding rows below it, obtaining a rectangular exchange matrix $\widetilde B$. This construction will be useful as later the scattering diagram technique behaves particularly well when $\widetilde B$ is of full rank, which can always be achieved by extension.
\end{rem}

We now discuss how much that a seed mutation is non-deterministic on coefficients influences cluster variables following Fraser \cite[Section 2.4]{Fr20}. For a seed $\Sigma = (\mathbf x, \mathbf p, B)$ (with possible frozen directions), define for $i\in I_\mathrm{uf}$, $s = 1, \dots, \ell_i$, monomials
\begin{equation}\label{eq: y hat variable}
    \hat y_{i,s}\coloneqq \frac{p_{i,s}}{p_{i,0}} \prod_{j\in I} x_j^{sb_{j,i}/\ell_i}.
\end{equation}

\begin{prop}[{\cite[Lemma 2.10]{Fr20}}]\label{prop: Fraser lemma}
    If for two cluster structures
    \[
        v\mapsto \Sigma_v \quad \text{and} \quad v\mapsto \Sigma_v', \quad v\in \mathbb T_{|I_\mathrm{uf}|},
    \]
    there is some vertex $v_0$ such that $\Sigma_{v_0} \sim \Sigma'_{v_0}$ meaning
    \[
        \text{$\frac{x_{i;v_0}}{x_{i;v_0}'}\in \mathbb P$ for any $i\in I$ and $\hat y_{i,s;v_0} = \hat y_{i,s;v_0}'$ for any $i\in I_\mathrm{uf}$ and $1\leq s\leq \ell_i$,}
    \]
    then we have $\Sigma_v\sim \Sigma_v'$ for any $v\in \mathbb T_{|I_\mathrm{uf}|}$.
\end{prop}

\begin{defn}\label{def: cluster algebra higher rk}
    Given a (generalized) cluster structure, the associated \emph{(generalized) cluster algebra} $\mathcal A$ is the $\mathbb {ZP}$-subalgebra of $\mathcal F$ generated by all cluster variables
    \[
        \bigcup_{v\in \mathbb T_n} \mathbf x_v = \{x_{i;v}\mid i\in I, v\in \mathbb T_n\}.
    \]
    The \emph{(generalized) upper cluster algebra} $\mathcal U$ is defined to be the intersection
    \[
        \bigcap_{v\in \mathbb T_n} \mathbb{ZP}[x_{i;v}^{\pm 1}\mid i\in I].
    \]
\end{defn}

It follows from \Cref{prop: Fraser lemma} that in any cluster structure where $\Sigma$ is associated to a given vertex $v_0$, any cluster variable in any seed is well-defined up to a factor in $\mathbb P$. Therefore we shall use the notations $\mathcal A = \mathcal A(\Sigma)$ and $\mathcal U = \mathcal U(\Sigma)$ for any seed $\Sigma$ in the cluster structure since $\Sigma$ is sufficient to determine both $\mathcal A$ and $\mathcal U$. By \Cref{thm: gen laurent CS}, $\mathcal A$ is contained in $\mathcal U$.

\begin{rem}\label{rem: pi0 is trivial}
    Another manipulation on a seed $\Sigma$ that does not change $\mathcal A(\Sigma)$ is to scale the coefficient tuple $\mathbf p_i = (p_{i, 0}, \dots, p_{i, \ell_i})$ by some uniform factor $p\in \mathbb P$. It follows from \Cref{prop: Fraser lemma} that any cluster variable is again well-defined up to a factor in $\mathbb P$. Thus to address the positivity in \Cref{thm: positive Laurent}, we can always assume that there is some seed $\Sigma$ where every $p_{i,0} = 1$ for $i\in I$. Moreover, we can take the free abelian group $\mathbb P'$ generated by symbols
    \[
        \{p_{i,s}\mid i\in I_\mathrm{uf}, 1\leq s\leq \ell_i\},
    \]
    with the natural group homomorphism $\psi$ to $\mathbb P$ suggested by the notation. Cluster variables generated in the new cluster structure are subtraction-free rational expressions in $|I|$ variables with coefficients in $\mathbb P'$. Then $\psi$ induces a map that sends a cluster variable in the new cluster structure to the corresponding cluster variable in the old cluster structure. Therefore once the positivity of Laurent phenomenon is established for $\mathbb P'$, it also holds for $\mathbb P$.
\end{rem}

\section{Generalized cluster scattering diagrams}\label{sec: gen cluster sd higher rk}

\subsection{Definitions}\label{subsec: gen cluster sd higher rk}

We now introduce a construction generalizing Gross--Hacking--Keel--Kontsevich's cluster scattering diagram \cite{GHKK} mostly carried out by Mou \cite{Mou} and Cheung, Kelley, and Musiker \cite{CKM}. Here we take a perspective of \cite{BFFMNC} that both coefficients and frozen directions are present. We impose no reciprocal conditions on coefficients to match the generality in \Cref{subsec: gen cluster struct}.

The \emph{generalized cluster scattering diagram} $\frakD_\mathbf s$ we are about to define is essentially attached to a seed $\Sigma = (\mathbf x, \mathbf p, \widetilde B)$ where $\widetilde B$ is a $|I|\times |I_\mathrm{uf}|$ rectangular exchange matrix and $p_{i, 0} = 1$ for any $i\in I_\mathrm{uf}$. The coefficient group $\mathbb P$ is taken to be the (multiplicative) abelian group freely generated by
\[
    \{p_{i, j}\mid i\in I_\mathrm{uf}, 1\leq j\leq \ell_i\}.
\]
For easier comparison with \cite{GHKK}, we closely follow the language therein (see also \cite{GHK}) where the information in $\Sigma$ can be easily extracted from the so-called \emph{fixed data} and \emph{seed data}.

\begin{defn}\label{def: fixed data}
    The fixed data $\Gamma$ consists of 
	\begin{itemize}
		\item a lattice $N$ of finite rank with a skew-symmetric bilinear form $\omega \colon N\times N \rightarrow \mathbb Q$;
		\item an \emph{unfrozen sublattice} $N_\mathrm{uf}\subset N$, a saturated sublattice of $N$;
		\item an index set $I = \{1, \dots, \operatorname{rank} N\}$ and a subset $I_\mathrm{uf} = \{1, \dots, \operatorname{rank} N_\mathrm{uf}\}$;
		\item positive integers $d_i$ for $i\in I$ with greatest common divisor $1$;
		\item a sublattice $N^\circ \subset N$ of finite index such that $\omega( N_\mathrm{uf}, N^\circ )\subseteq \mathbb Z$ and $\omega( N, N_\mathrm{uf}\cap N^\circ )\subseteq \mathbb Z$;
		\item $M = \operatorname{Hom}(N, \mathbb Z)$; 
        \item a free (multiplicative) abelian group $\mathbb P$ of finite rank.
	\end{itemize}
\end{defn}

\begin{defn}\label{def: seed data}
	Given the fixed data $\Gamma$, a seed data $\mathbf s = ( \mathbf e, \mathbf p)$ consists of $\mathbf e = (e_i)_{i\in I}\in N^I$ and $\mathbf p = (\mathbf p_i)_{i\in I_\mathrm{uf}}$ where $\mathbf p_i = (p_{i,j})_{1\leq j \leq \ell_i}$ and $p_{i,j}\in \mathbb P$ such that
	\begin{enumerate}
		\item $\{ e_i \mid i\in I\}$ is a basis for $N$;
		\item $\{ e_i \mid i\in I_\mathrm{uf}\}$ is a basis for $N_\mathrm{uf}$;
		\item $\{d_ie_i\mid i\in I\}$ is a basis for $N^\circ$;
		\item for $i\in I_\mathrm{uf}$, the element $w_i \coloneqq \omega(-,d_i e_i)/\ell_i$ is primitive in $M$;
        \item $\{p_{i,j}\mid i\in I_\mathrm{uf}, 1\leq j\leq \ell_i\}$ is a free generating set for $\mathbb P$.
	\end{enumerate}
\end{defn}

We impose the following \emph{injectivity assumption}.

\noindent\textbf{Injectivity assumption:} The map
\begin{equation}\label{eq: inj assumption}
    N_\mathrm{uf}\cap N^\circ \rightarrow M, \quad v\mapsto \omega(-, v)
\end{equation}
is injective.

Before defining scattering diagrams, we briefly explain how to obtain a seed $\Sigma = (\mathbf x, \mathbf p, \widetilde B)$ from $\Gamma$ and $\mathbf s = (\mathbf e, \mathbf p)$. The ambient field $\mathcal F$ is the fraction field of $\operatorname{Frac}(\mathbb {ZP})[M]$. Let $\{e_i^*\in M\mid i\in I\}$ denote the dual basis of $\mathbf e$. The cluster $\mathbf x$ in $\Sigma$ is
\[
    \{x_i\mid i\in I\} = \{x^{e_i^*}\mid i\in I\}.
\]
The coefficient tuples $\mathbf p = (\mathbf p_i)_{i\in I_\mathrm{uf}}$ are copied from $\mathbf s$. For $i\in I$ and $j\in I_\mathrm{uf}$, let
\[
    b_{i, j} = \omega(e_i, d_je_j)\in\mathbb Z.
\]
Then we put $\widetilde B = (b_{i, j})\in \mathbb Z^{|I|\times |I_\mathrm{uf}|}$. The condition that each $\ell_i$ divides the $i$-th column of $\widetilde B$ is ensured by (4) of \Cref{def: seed data}. The injectivity assumption is then equivalent to that $\widetilde B$ is of full rank.

\begin{rem}\label{rem: recover seed data}
    It is straightforward to construct fixed and seed data from a seed $\Sigma$. Let $N = \mathbb Z^{|I|}$ and $\{e_i\mid i\in I\}$ be the standard basis. For $i\in I_\mathrm{uf}$, let $d_i$ be integers with greatest common divisor $1$ such that the diagonal matrix $(d_i\delta_{i,j})_{i,j\in I_\mathrm{uf}}$ is a left skew-symmetrizer for the first $|I_\mathrm{uf}|$ rows of $\widetilde B$. We specify the form $\omega$, which is determined by $\omega(e_i, e_j)$ for $i, j\in I$. Let
    \[
        \omega(e_i, e_j) = \begin{dcases}
            b_{i,j}/d_j \quad & \text{$i\in I$ and $j\in I_\mathrm{uf}$}, \\
            -b_{j,i}/d_i \quad & \text{$i\in I_\mathrm{uf}$ and $j\in I_\mathrm{fr}$}, \\
            0 \quad & \text{$i\in I_\mathrm{fr}$ and $j\in I_\mathrm{fr}$}.
        \end{dcases}
    \]
    For any $j\in I_\mathrm{fr} = I\setminus I_\mathrm{uf}$, set $d_j = \operatorname{lcm}(d_i\mid i\in I_\mathrm{uf})$. With these choices, the conditions in \Cref{def: fixed data} and \Cref{def: seed data} are all satisfied.
\end{rem}

The scattering diagram $\frakD_\mathbf s$ will live in the vector space $M_\mathbb R = M\otimes \mathbb R$. As in the rank-$2$ case, we need a monoid $\mndP$ with a monoid map $r\colon \mndP \rightarrow M$. For our purposes in higher ranks, the choice will soon be specified.

Denote by $\mathbb P^\oplus = \mathbb P^\oplus_\mathbf s$ the sub-monoid of $\mathbb P$ generated by all $p_{i,s}$, which is isomorphic to $\bigoplus_{i\in I_\mathrm{uf}}\mathbb N^{\ell_i}$ (written additively). Let
\[
    \mndP = \mndP_\mathbf s \coloneqq M \oplus \mathbb P^\oplus
\]
equipped with the monoid map
\[
    r\colon \mndP \rightarrow M, \quad (m, p) \mapsto m.
\]
Recall that $\Bbbk$ denotes a ground field of characteristic zero. Analogous to the rank-2 situation, define ideals (the same notation for both)
\[
    \frakm \coloneqq \mndP \setminus \mndP^\times \quad \text{and} \quad \frakm \coloneqq \Bbbk[\mndP\setminus \mndP^\times]\subseteq \Bbbk[\mndP],
\]
where $\mndP^\times$ is all units in $\mndP$, which is $\{(m, 1)\mid m\in M\}$ in the current case. Monomials in $\Bbbk[\mndP]$ are associated with elements in $\mndP$, thus simply denoted as $px^m$ for $(m, p)\in \mndP$.

The following definitions of a wall and a scattering diagram are with respect to the fixed data $\Gamma$ and a seed data $\mathbf s$.

\begin{defn}
    A \emph{wall} is the data $\frakd$ of
    \begin{itemize}
        \item the \emph{support} $S_\frakd$, a convex rational polyhedral cone of codimension one in $M_{\mathbb R}$
        with primitive normal vector 
        \[
            n_\frakd \in N^+ = N^+_\mathbf s \coloneqq \left\{\sum_{i\in I_\mathrm{uf}}a_ie_i \,\middle\vert\, a_i\in \mathbb N \right\}\setminus\{0\};
        \]
        \item the \emph{wall-function} $f_\frakd$, an element in $\Bbbk[\mndP]/I$ of an $\frakm$-primary ideal $I$ or in $\widehat{\Bbbk[\mndP]}\coloneqq \lim\limits_{\longleftarrow} \Bbbk[\mndP]/{\mathfrak m^k}$ satisfying
        \[
            f_{\frakd} = 1 + \sum_{(m, p)\in \mndP} c_{(m, p)} px^m,
        \]
        where $r(m, p) = m \in \mathbb R_{>0}\omega(-, n_\frakd)\cap M$.
    \end{itemize}
\end{defn}

For a wall $\frakd$, the primitive generator $m_\frakd$ of $\mathbb R_{\geq 0}\omega(-, n_\frakd)\cap M$ is called its \emph{direction}. Any monomial $px^m$ that appears in $f_\frakd$ then has $m = km_\frakd$ for some integer $k\geq 1$.

\begin{defn}
    A \emph{scattering diagram} over $\Bbbk[\mndP]/I$ for an $\frakm$-primary ideal $I$ is a finite collection of walls whose wall-functions are in $\Bbbk[\mndP]/I$. A \emph{scattering diagram} $\frakD$ over $\widehat{\Bbbk[\mndP]}$ is a countable collection of walls whose wall-functions are in $\widehat{\Bbbk[\mndP]}$ such that for each $k\geq 0$, the set
    \[
        \frakD_k \coloneqq \{\frakd\in \frakD\mid f_\frakd \not\equiv 1 \mod \frakm^{k+1}\}
    \]
    is a scattering diagram over $\Bbbk[\mndP]/\frakm^{k+1}$.
\end{defn}

As in the rank-2 situation, we write
\begin{equation}\label{eq: support and sing locus}
    \operatorname{Supp}(\frakD) \coloneqq \bigcup_{\frakd\in \frakD} S_\frakd \quad \text{and}\quad
    \operatorname{Sing}(\frakD) \coloneqq \bigcup_{\frakd\in \frakD} \partial S_\frakd \cup \bigcup_{\frakd, \frakd'}(S_\frakd \cap S_{\frakd'})
\end{equation}
for the \emph{support} and \emph{singular locus} of $\frakD$ where the last union is taken over all pairs of walls such that the intersection of their supports is of dimension $\dim M_\mathbb R - 2$.

Wall-crossing automorphisms are defined by the higher-rank analogue of \eqref{eq: wcs automorphism}, which act on the complete algebra $\widehat{\Bbbk[\mndP]}$. The notion of a path-ordered product $\mathfrak p_{\gamma} = \mathfrak p_{\gamma, \frakD}$ is also extended verbatim to the higher-rank situation for a \emph{regular path} $\gamma$ that avoids $\operatorname{Sing}(\frakD)$ and whose endpoints avoid $\operatorname{Supp}(\frakD)$. The equivalence relation between scattering diagrams and the consistency of a scattering diagram are defined in the same way as respectively in \Cref{def: equiv sd rk 2} and \Cref{def: consistent sd rk 2}.

\begin{defn}
    A wall $\frakd$ is said to be \emph{incoming} if its direction $m_\frakd$ (or equivalently $\omega(-, n_\frakd)$) is contained in the support $S_\frakd$. Otherwise it is called \emph{non-incoming} or \emph{outgoing}.
\end{defn}

Consider the initial scattering diagram
\[
    \frakD_{\mathrm{in}, \mathbf s} \coloneqq \{\frakd_i\mid i\in I_\mathrm{uf}\} \quad \text{where} \quad S_{\frakd_i} \coloneqq e_i^\perp \quad \text{and} \quad f_{\frakd_i} \coloneqq 1 + \sum_{j=1}^{\ell_i} p_{i, 1}x^{jw_i}.
\]
Clearly every wall in $\frakD_{\mathrm{in}, \mathbf s}$ is incoming.

The following result slightly generalizes \cite[Theorem 1.12]{GHKK} and is again a consequence of \cite[2.1.6]{KSwcs}; see \cite[Section 6.2]{Mou} and \cite[Section 3.2]{CKM}.

\begin{thm}\label{thm: gen cluster sd}
    There is a scattering diagram $\frakD_{\mathbf s}$ satisfying:
    \begin{enumerate}
        \item $\frakD_{\mathbf s}$ is consistent,
        \item $\frakD_{\mathbf s} \subseteq \frakD_{\mathrm{in}, \mathbf s}$,
        \item $\frakD_{\mathbf s}\setminus \frakD_{\mathrm{in}, \mathbf s}$ consists only of outgoing walls.
    \end{enumerate}
    Moreover, $\frakD_{\mathbf s}$ is unique up to equivalence.
\end{thm}

When $p_{i,0}$ and $p_{i, \ell_i}$ specialize to $1$ and any $p_{i,j}$, $j\notin \{0,\ell_i\}$ specializes to $0$ (then $f_{\frakd_i}$ becoming a binomial) for any $i\in I_\mathrm{uf}$, then $\frakD_{\mathbf s}$ is what now known as the \emph{cluster scattering diagram} of Gross--Hacking--Keel--Kontsevich \cite[Section 1.1]{GHKK}. We call $\frakD_{\mathbf s}$ defined by \Cref{thm: gen cluster sd} a \emph{generalized cluster scattering diagram}.

\begin{rem}
    When $I_\mathrm{uf} = I$ and $|I| = 2$, then $\frakD_\mathbf s$ can be easily transformed to the generalized cluster scattering diagram $\frakD$ considered in \Cref{subsec: gen cluster sd rk 2} by a piecewise linear transformation as in \Cref{thm: mutation of sd rk 2}.
\end{rem}

Our main novel result concerning $\frakD_{\mathbf s}$ is the positivity of its wall-functions, \Cref{thm: positive wall-function}, from which all remaining positivity results follow. The proof will be given in \Cref{sec: proof of positivity}.

\begin{thm}\label{thm: positive wall-function}
    There is a representative (in the equivalence class) of $\frakD_{\mathbf s}$ such that for any wall $\frakd\in \frakD_{\mathbf s}$, the wall-function $f_\frakd$ belongs to $\widehat{\mathbb N[\mndP]}$, that is,
    \[
        f_\frakd = 1 + \sum_{(m, p)\in \mndP} c_{(m,p)}px^m
    \]
    has nonnegative integer coefficients $c_{(m,p)}$.
\end{thm}

\subsection{The cluster complex structure}\label{subsec: cluster complex high rk}

The scattering diagram $\frakD_\mathbf s$ possesses a \emph{cluster complex structure} as in the rank-2 situation. This structure is important for uncovering the connection with the corresponding generalized cluster algebra and is thus essential to the proof of \Cref{thm: positive Laurent}. In the ordinary case, the cluster complex structure is described in \cite{GHKK}. In the generalized case, the underlying cone complex is exactly the same \cite{Mou, CKM}, which only depends on the exchange matrix $\widetilde B$. The more involved part in this generalization is the description of wall-functions (\Cref{thm: cluster complex struct any rk}) due to the complication of coefficients.

We need some preparations. Let $B$ be a skew-symmetrizable $n\times n$ matrix associated to a chosen root $v_0\in \mathbb T_n$. Then for each $v\in \mathbb T_n$, we define inductively $\mathbf c$-vectors
\[
    \mathbf c_{v}^{(i)} = \left(c_{j;v}^{(i)} \,\middle\vert\, j=1, \dots, n \right) \in \N^n \ \text{or}\ \mathbb Z_{\leq 0}^n,\quad \text{for} \quad i=1, \dots, n
\]
by the following rule. 
Set initially
\[
    \mathbf c_{v_0}^{(i)} \coloneqq (\delta_{i,j}\mid j=1,\dots, n), \quad i = 1, \dots, n,
\]
the standard unit vectors in $\mathbb Z^n$.
For an edge $v\frac{k}{\quad\quad}v'$ in $\mathbb T_n$, their associated $\mathbf c$-vectors satisfy
\begin{equation}\label{eq: def c vector}
    \mathbf c_{v'}^{(i)} = \begin{dcases}
         - \mathbf c_{v}^{(k)} \quad & \text{if $i = k$;} \\
         \mathbf c_{v}^{(i)} + b_{k,i}^v [\mathbf c_{v}^{(k)}]_+ \quad & \text{if $i\neq k$ and $b_{i, k}^v \leq 0$;} \\
         \mathbf c_{v}^{(i)} + b_{k,i}^v [-\mathbf c_{v}^{(k)}]_+ \quad & \text{if $i\neq k$ and $b_{i, k}^v > 0$}.
    \end{dcases}
\end{equation}
Notice that the relation (\ref{eq: def c vector}) is reflexive for $v$ and $v'$. Therefore, we have well-defined $\mathbf c_v^{(i)}$ for any $v$ by propagating (\ref{eq: def c vector}) along $\mathbb T_n$ starting from $v_0$. The $\mathbf c$-vectors were developed by Fomin and Zelevinsky in \cite{FZco} as a useful tool to study the dynamics of coefficients.

\begin{thm}[{\cite[Corollary 5.5]{GHKK}}]\label{thm: sign coherence c}
    The $\mathbf c$-vectors are always sign coherent, that is, each $\mathbf c$-vector $\mathbf c_v^{(i)}$ for $i = 1, \dots, n$ and $v\in \mathbb T_n$ is either in $\N^n$ or $\mathbb Z_{\leq 0}^n$.
\end{thm}

\begin{rem}
    The sign coherence of $\mathbf c$-vectors was conjectured in \cite[Conjecture 5.4 and Proposition 5.6(iii)]{FZco}. In the skew-symmetric case, it was proven by Derksen, Weyman and Zelevinsky \cite{DWZ10} using quiver representations. More generally, for any skew-symmetrizable exchange matrix, the sign coherence was proven in \cite{GHKK} using scattering diagram techniques. It also follows that the vectors $\mathbf c_v^{(1)}, \dots, \mathbf c_v^{(n)}$ form a basis in $\mathbb Z^n$; see the proof of \cite[Proposition 5.6]{FZco}.
\end{rem}

It is straightforward from the recurrence (\ref{eq: def c vector}) that 
\[
    c_{i; v}^{(i)} + 1 \quad \text{or} \quad c_{i;v}^{(i)} - 1
\]
is divisible by $\ell_i$, and $c_{j;v}^{(i)}$ is divisible by $\ell_i$ when $j\neq i$.

\begin{defn}
    Define non-negative integers
    \begin{equation}
        h_{j;v}^{(i)} \coloneqq \begin{dcases}
            \frac{|c_{i;v}^{(i)}\pm 1|}{\ell_i}\in \mathbb N\quad & \text{if $i = j$};\\
            \frac{|c_{j;v}^{(i)}|}{\ell_i}\in \mathbb N\quad & \text{if $i\neq j$}.
        \end{dcases}
    \end{equation}
\end{defn}

\begin{rem}
    When $\ell_i = 1$ or $2$, there is ambiguity on $h_{i;v}^{(i)}$. However, it does not affect the description in \Cref{thm: cluster complex struct any rk}.
\end{rem}

For given seed data $\mathbf s$ with $\mathbf e = \{e_1, \dots, e_n\}$ and $B = (b_{i,j}) = \omega(e_i, d_je_j)$ for $i,j\in I_\mathrm{uf}$. we define another set of vectors $g^*_{i;v}\in N_\mathrm{uf} \cong \mathbb Z^n$ again indexed by $i\in I_\mathrm{uf} = \{1, \dots, n\}$ and $v\in \mathbb T_n$. These are called \emph{dual $g$-vectors}. The relation they need to satisfy is analogous to (\ref{eq: def c vector}) but with respect to $-B^\intercal$. Explicitly we set initially
\[
    g^*_{i;v_0} \coloneqq e_i, \quad i=1, \dots, n,
\]
and for $v\frac{k}{\quad\quad}v'$ in $\mathbb T_n$, put the relation

\begin{equation}\label{eq: def dual g vector}
    g_{i; v'}^* = \begin{dcases}
         - g_{k; v}^* \quad & \text{if $i = k$;} \\
         g_{i; v}^* - b_{i, k}^v [g_{k; v}^*]_+ \quad & \text{if $i\neq k$ and $b_{i, k}^v \leq 0$;} \\
         g_{i; v}^* - b_{i, k}^v [-g_{k; v}^*]_+ \quad & \text{if $i\neq k$ and $b_{i, k}^v > 0$}.
    \end{dcases}
\end{equation}
For the same reason as $\mathbf c$-vectors, they are sign-coherent by \Cref{thm: sign coherence c} and $(g_{1;v}^*, \dots, g_{n;v}^*)$ form a basis of $N_\mathrm{uf}$. Moreover, $g_{i;v}^*$ (in the basis $\mathbf e$) and $\mathbf c^{(i)}_{v}$ always have the same sign (see \cite[(2,7)]{NZ}), which we denote by
\[
    \varepsilon_{i;v} \coloneqq \begin{cases}
            1 \quad & \text{if $\mathbf c^{(i)}_v\in \N^n$}, \\
            -1 \quad & \text{if $\mathbf c^{(i)}_v\in \mathbb Z_{\leq 0}^n$}.
        \end{cases}
\]

Define for each $v\in \mathbb T_n$, the convex cone 
\[
    \mathcal C_v^+ \coloneqq \{m\in M_\mathbb R\mid \langle m, g_{i;v}^* \rangle \geq 0,\ i\in I_\mathrm{uf}\} \subseteq M_\mathbb R.
\]
It has $n$ facets 
\[
    S_{i;v}\coloneqq (g_{i;v}^*)^\perp \cap \mathcal C_v^+, \quad i\in I_\mathrm{uf}.
\]
One can verify inductively that
\begin{equation}\label{eq: w vector no sign}
    w_{i;v} \coloneqq \omega(-, \frac{d_i}{\ell_i} \varepsilon_{i;v}g_{i;v}^*)\in M.
\end{equation}

Now we are ready to state the \emph{cluster complex structure} of $\frakD_\mathbf s$.

\begin{thm}\label{thm: cluster complex struct any rk}
    There is a pure $|I|$-dimensional cone complex $\Delta^+_\mathbf s$ in $M_\mathbb R$ whose top-dimensional cones are
    \[
        \{\mathcal C_v^+\mid v\in \mathbb T_n\}.
    \]
    In its equivalent class, $\frakD_\mathbf s$ has a representative containing the collection of walls 
    \[
        \frakD(\Delta_\mathbf s^+) \coloneqq \{(S_{i;v}, f_{i;v})\mid i\in I_\mathrm{uf}, v\in \mathbb T_n\}
    \]
    where
    \begin{equation}\label{eq: cluster wall-function}
        f_{i;v} = \begin{dcases}
            1 + \sum_{s=1}^{\ell_i} p_{i,s} \prod_{j=1}^{n} p_{j,\ell_j}^{sh_{j;v}^{(i)}}x^{sw_{i;v}} \quad & \text{if $\ell_ih_{i;v}^{(i)} = |c_{i;v}^{(i)}|-1$} \\
            1 + \sum_{s=1}^{\ell_i} \frac{p_{i,\ell_i-s}}{p_{i,\ell_i}} \prod_{j=1}^n p_{j,\ell_j}^{sh_{j;v}^{(i)}} x^{sw_{i;v}} \quad & \text{if $\ell_ih_{i;v}^{(i)} = |c_{i;v}^{(i)}|+1$} 
        \end{dcases}
    \end{equation}
    such that $\frakD_\mathbf s\setminus \frakD(\Delta_\mathbf s^+)$ has support disjoint from (the codimension $\leq 1$ strata of) $\Delta_\mathbf s^+$.
\end{thm}

\begin{proof}
    The cone complex $\Delta_\mathbf s^+$ is understood and has equivalent description in \cite[Section 4 and 5]{GHKK}. We thus focus on deriving the claimed expression of wall-functions from \cite[Theorem 7.10]{Mou}. There $\frakD_\mathbf s$ has initial wall-functions for $i = 1, \dots, n$
    \[
        \prod_{j=1}^{\ell_i} (1 + t_{i,j}x^{w_i})
    \]
    where $t_{i,j}$ are formal variables. While these were denoted as $p_{i,j}$ in \cite{Mou}, we change everything there from the letter $p$ to $t$ in order to avoid conflict with our coefficients $p_{i,s}$. In fact, we can regard $p_{i, s}$ as symmetric polynomials of $t_{i,j}$ by equating
    \begin{equation}\label{eq: tij and pis}
        1 + \sum_{s=1}^{\ell_i} p_{i, s} x^{sw_i} = \prod_{j=1}^{\ell_i} (1 + t_{i,j}x^{w_i}).
    \end{equation}
    In particular we have $p_{i, \ell_i} = \prod_{j=1}^{\ell_i} t_{i, j}$.
    
    The wall-function $f_{i;v}$ is expressed in \cite[Theorem 7.10]{Mou} as
    \[
        f_{i;v} = \prod_{j=1}^{\ell_i}\left(1 + t_{i,j;v}^{\varepsilon_{i;v}} x^{w_{i;v}} \right) \quad \text{where} \quad \varepsilon_{i;v} = \begin{cases}
            1 \quad & \text{if $\mathbf c^{(i)}_v\in \N^n$}, \\
            -1 \quad & \text{if $\mathbf c^{(i)}_v\in \mathbb Z_{\leq 0}^n$},
        \end{cases}
    \]
    and each $t_{i,j;v}$ is a monomial in $\{t_{i, j}\mid i\in I_\mathrm{uf}, j=1, \dots, \ell_i\}$. The exponent vector of $t_{i,j;v}$ is termed as a \emph{generalized $c$-vector} and has the same sign $\varepsilon_{i;v}$ as $\mathbf c_{v}^{(i)}$ (\cite[Proposition 3.17]{Mou}). By \cite[Lemma 3.15 and Lemma 3.16]{Mou}, the monomial $t_{i,j;v}$ (simultaneously for all $j=1, \dots, \ell_i$) is of the form
    \[
        t_{i,j} \cdot p_{i,\ell_i}^{\frac{c_{i;v}^{(i)}- 1}{\ell_i}}\cdot \prod_{j\neq i}p_{j,\ell_j}^{\frac{c_{j;v}^{(i)}}{\ell_j}} \quad \text{or} \quad
        t_{i,j}^{-1} \cdot p_{i,\ell_i}^{\frac{c_{i;v}^{(i)}+ 1}{\ell_i}}\cdot \prod_{j\neq i}p_{j,\ell_j}^{\frac{c_{j;v}^{(i)}}{\ell_j}}\,.
    \]
    We assume the former case and the latter one can be similarly treated. Now there are two further cases.

    \emph{Case 1:} $\mathbf c_{v}^{(i)}\in \mathbb Z_{\geq 0}^n$. This means
    \[
        \ell_ih_{i;v}^{(i)} = c_{i;v}^{(i)} - 1 \quad \text{and} \quad \ell_jh_{j;v}^{(i)} = c_{i;v}^{(i)}.
    \]
    Then we can express $f_{i;v}$ as
    \[
        \prod_{k=1}^{\ell_i}\left(1 + t_{i,k} \prod_{j=1}^n p_{j, \ell_j}^{h_{j;v}^{(i)}} x^{w_{i;v}} \right) = 1 + \sum_{s=1}^{\ell_i} p_{i, s}\prod_{j=1}^n p_{j, \ell_j}^{sh_{j;v}^{(i)}} x^{sw_{i;v}},
    \]
    proving the first scenario of (\ref{eq: cluster wall-function}).

    \emph{Case 2:} $\mathbf c_{v}^{(i)}\in \mathbb Z_{\leq 0}^n$. In this case, we have
    \[
        \ell_ih_{i;v}^{(i)} = -c_{i;v}^{(i)} + 1 \quad \text{and} \quad \ell_jh_{j;v}^{(i)} = -c_{i;v}^{(i)}.
    \]
    We can then express $f_{i;v}$ as
    \[
        \prod_{k=1}^{\ell_i} \left(1 + t_{i,k}^{-1} \prod_{j=1}^n p_{j,\ell_j}^{h_{j;v}^{(i)}} x^{w_{i;v}} \right)
        = 1 + \sum_{s=1}^{\ell_i} \frac{p_{i,\ell_i-s}}{p_{i,\ell_i}} \prod_{j=1}^n p_{j, \ell_j}^{sh_{j;v}^{(i)}} x^{sw_{i;v}}.
    \]
    This verifies the second scenario of (\ref{eq: cluster wall-function}).
\end{proof}

\subsection{Theta functions and positivity of the Laurent phenomenon}\label{subsec: theta and positive}

Broken lines can also be defined in higher ranks. It suffices for our purposes to consider the initial exponent in $M$. We treat $\Bbbk[\mathbb P^\oplus]$ as the coefficient ring and thus express a monomial (with coefficient) in $\Bbbk[\mndP]$ as $cx^m$ where $c\in \Bbbk[\mathbb P^\oplus]$ and $m\in M$.

\begin{defn}
   Let $m_0 \in M\setminus \{0\}$ and $Q \in M_\R \cut \Supp(\frakD)$. A \emph{broken line} $\bline$ for $m_0$ with endpoint $Q$ is a piecewise linear continuous proper path 
   \[
    \bline\colon (-\infty,0] \to M_\R \cut \operatorname{Sing}(\mathfrak D)
   \]
   with values $-\infty = \tau_0 < \tau_1 < \cdots < \tau_{\ell+1} = 0$ and an associated monomial $c_i x^{m_i}\in \widehat{\Bbbk[\mndP]}$ for each interval $(\tau_{i}, \tau_{i+1})$, $i = 0, \dots, \ell$ such that 
  \begin{enumerate}
    \item $\bline(0) = Q$ and $c_0 = 1$;
    \item $\dot{\bline}(\tau) = -m_i$ for any $\tau \in (\tau_{i},\tau_{i+1})$ for each $i = 0, \dots, \ell$;
    \item the subpath $\bline\vert_{(\tau_{i-1}, \tau_{i+1})}$ transversally crosses the support of some wall at $\tau_i$ for $i = 1, \dots, \ell-1$;
    \item for $i = 0, \dots, \ell -1$, $c_{i+1}x^{m_{i+1}}$ is a monomial term (other than $c_ix^{m_i}$) of the power series
    \[
      c_ix^{m_i} \prod_{\{\mathfrak d\in \mathfrak D\mid \bline(\tau_i)\in S_\mathfrak d\}} f_\mathfrak d^{n_i\cdot m_i}
    \]
    where $n_i$ is primitive in $N$ and orthogonal to any $\frakd$ with $\bline(\tau_i)\in S_\mathfrak d$ such that $n_i \cdot r(m_i)>0$.
  \end{enumerate}
\end{defn}

As in rank $2$, a theta function is defined to be the formal sum of the final monomial terms of broken lines. Precisely, the theta function associated to $m_0$ and $Q$ is
\begin{equation}\label{eq: theta higher rk}
    \vartheta_{Q, m_0} \coloneqq \sum_\bline c(\bline)x^{m(\bline)},
\end{equation}
the sum of all final monomials over all broken line for $m_0$ with endpoint $Q$. 

We summarize in below the main properties of theta functions that we will need.

\begin{prop}\label{prop: theta function properties}
    Let $\frakD$ be a scattering diagram. Let $m_0\in M\setminus \{0\}$ and $Q\in M_\mathbb R\setminus \operatorname{Supp}(\frakD)$.
    \begin{enumerate}
        \item The theta function $\vartheta_{Q, m_0}$ lives in $x^{m_0}\widehat{\Bbbk[\mndP]}$.
        \item Let $\frakD = \frakD_\mathbf s$, and let $Q \in \operatorname{Int}(\mathcal C_v^+)$ and $m\in \mathcal C_v^+\cap M$. Then $\vartheta_{Q, m} = x^m$.
    \end{enumerate}
\end{prop}

\begin{proof}
    (1) generalizes \cite[Proposition 3.4]{GHKK}. The settings slightly differ as our monoid $\mndP$ contains the coefficient part $\mathbb P^\oplus$. The proof nonetheless does not change.

    (2) generalizes \cite[Corollary 3.9]{GHKK}. It follows from generalizing \cite[Proposition 3.8]{GHKK} and \cite[Proposition 3.6]{GHKK} to our setting. The former is the special case when $v = v_0$ and carries over by the same argument. The latter is a bijection between certain broken lines in $\frakD_\mathbf s$ and $\frakD_{\mu_k(\mathbf s)}$. We have not defined the mutation $\mu_k(\mathbf s)$ but it can be understood as in \cite[Definition 6.3]{Mou} by unveiling the relation between the variables $t_{i,j}$ and $p_{i,s}$ (\ref{eq: tij and pis}). Then we can refer to \cite[Proposition 8.9]{Mou}.
\end{proof}

Let $\Sigma$ be a seed arising from $\Gamma$ and $\mathbf s = (\mathbf e, \mathbf p)$ as in \Cref{subsec: gen cluster sd higher rk}. Consider the generalized cluster algebra $\mathcal A = \mathcal A(\Sigma)$.

\begin{thm}\label{thm: cluster monomial theta}
    There is a generalized cluster structure such that $\Sigma_{v_0} = \Sigma$ and the cluster monomials in $\mathcal A$ (with frozen variables inverted)
    \[
        \left\{\prod_{i\in I_\mathrm{uf}}x_{i;v}^{a_i}\prod_{j\in I_\mathrm{fr}}x_{j;v}^{b_j}\,\middle\vert\, a_i\in \mathbb N, b_j\in \mathbb Z, v\in \mathbb T_{|I_\mathrm{uf}|}\right\}    
    \]
    are precisely the theta functions 
    \[
        \{\vartheta_{Q, m}\mid m\in \Delta_\mathbf s^+ \cap M\}
    \]
    for a general $Q\in \mathcal C_{v_0}^+$.
\end{thm}

\begin{proof}
    In the ordinary case, this is \cite[Theorem 4.9]{GHKK}. The generalized case is discussed in \cite[Section 8.4]{Mou} (and also in \cite{CKM} with reciprocal coefficients) and essentially relies on \Cref{thm: cluster complex struct any rk} the cluster complex structure of $\frakD_\mathbf s$. We elaborate more as the setup here differs slightly.

    In \cite{Mou}, we assume $I_\mathrm{uf} = I$ namely without any frozen directions. Hence in that case the maximal cones in $\Delta_\mathbf s^+$ are all simplicial whose (primitive) generators are called $g$-vectors. With frozen directions, we can still define $g$-vectors (using the same recurrence)
    \[
        \{g_{i;v}\in M\mid i\in I_\mathrm{uf}\}.
    \]
    The difference caused by frozen directions is that the recurrence to define $g$-vectors only uses the rectangular exchange matrix $\widetilde B$ (and its mutations). Another option is to extend $I_\mathrm{uf}$ to $I$ per the construction in \Cref{rem: recover seed data}. Now we have a $|I|\times |I|$ square exchange matrix which can be used to define the dual $g$-vectors $\{g_{i;v}^*\in N\mid i\in I\}$ for any $v\in \mathbb T_{|I|}$. Then $\{g_{i;v}\in M\mid i\in I\}$ can be defined to be the dual basis in $M$ and we only take the ones with $i\in I_\mathrm{uf}$. 
    
    The vectors $\{g_{i;v}\mid i\in I_\mathrm{uf}\}$ are dual to $\{g_{i;v}^*\mid i\in I_\mathrm{uf}\}$ in the sense that $\langle g_{i;v}, g_{j;v}^*\rangle = \delta_{i,j}$. The cone $\mathcal C_v^+$ has the description
    \[
        \operatorname{span}_{\mathbb R_{\geq 0}}\{g_{i;v}\mid i\in I_\mathrm{uf}\} \oplus \bigoplus_{i\in I_\mathrm{fr}} \mathbb Re_{i}^*.
    \]
    Any $m\in \mathcal C_v^+\cap M$ is a sum of a nonnegative integral sum of $\{g_{i;v}\mid i\in I_\mathrm{uf}\}$ and a vector $m_\mathrm{fr}\in \bigoplus_{i\in I_\mathrm{fr}} \mathbb Z e_i^*$. In \cite{Mou}, the mutation of a seed is deterministic so we can generate a cluster structure from $\Sigma_{v_0} = \Sigma$, which can be extracted from \Cref{thm: cluster complex struct any rk}. In particular, by \cite[Proposition 8.1]{Mou}, the unfrozen cluster variable $x_{i;v}$ is given by
    \begin{equation}\label{eq: pop cluster var high}
        x_{i;v} = \frakp_{\gamma}(x^{g_{i;v}})\in \widehat{\Bbbk[\mndP]}
    \end{equation}
    for a regular path $\gamma$ going from $\mathcal C_v^+$ to $\mathcal C_{v_0}^+$.
    Any wall-crossing automorphism acts trivially on $x^{m_\mathrm{fr}}$, which is a Laurent monomial of frozen variables $\{x^{e_i^*}\mid i\in I_\mathrm{fr}\}$. Now we see that any cluster monomial (with frozen variables inverted) equals
    \[
        \frakp_{\gamma}(x^m)\in \widehat{\Bbbk[\mndP]}
    \]
    for $m$ in $\Delta_\mathbf s^+ \cap M$.

    To see that any $\frakp_\gamma(x^m)$ for $m\in \Delta_\mathbf s^+\cap M$ is a theta function, we first have by \Cref{prop: theta function properties}
    \[
        \vartheta_{Q_v, m} = x^m
    \]
    for any $m\in \mathcal C_{v}^+\cap M$ and any $Q_v$ in the interior of $\mathcal C_{v}^+$. Then using a higher-rank analogue of \Cref{thm: CPS consistency} (see \cite[Theorem 3.5]{GHKK} and \cite{CPS}), we have
    \begin{equation}\label{eq: change endpoint theta high}
        \frakp_\gamma(x^m) = \frakp_\gamma(\vartheta_{Q_v, m}) = \vartheta_{Q, m}.
    \end{equation}
    for a general $Q\in \mathcal C_{v_0}^+$.
\end{proof}

Now we show the positivity of Laurent phenomenon for cluster monomials in any generalized cluster structure (with possible frozen directions). See the statement in \Cref{thm: positive Laurent} and terminologies in \Cref{subsec: gen cluster struct}.

\begin{proof}[Proof of \Cref{thm: positive Laurent}]
    If suffices to show this when the (rectangular) exchange matrix $B$ has full rank. If not, we can add more frozen directions by adding rows to $B$ to get a full rank matrix. Evaluating added frozen variables at $1$ recovers original cluster monomials.

    Another simplification is to assume $p_{i,0} = 1$ for any $i\in I_\mathrm{uf}$ as discussed in \Cref{rem: pi0 is trivial}. Now with the injectivity assumption (\ref{eq: inj assumption}) ensured by the full rank of $B$, the generalized cluster structure is contained within the scattering diagram $\frakD_\mathbf s$ where $\mathbf s$ can be constructed from a seed $\Sigma = (\mathbf x, \mathbf p, B)$ as in \Cref{rem: recover seed data}.

    Now every cluster monomial is a theta function $\vartheta_{Q, m}$ by \Cref{thm: cluster monomial theta}. Now the positivity of the Laurent phenomenon follows from \Cref{thm: gen laurent CS} and that the weight $c(\gamma)$ of every broken line $\gamma$ is in $\mathbb N\mathbb P^\oplus$ due to the positivity of wall-functions \Cref{thm: positive wall-function}.
\end{proof}

There are variants of Chekhov--Shapiro's generalized cluster structures, for example \cite{Nak, GSV18, Mou}, where certain restrictions are imposed on coefficients so that coefficients are ``normalized'' and thus mutations are deterministic. These settings can nonetheless be viewed as special cases of \Cref{def: cluster structure} as pointed out in \emph{loc.~cit}. Therefore, the positivity \Cref{thm: positive Laurent} is still valid in those settings.

\subsection{Strong positivity of theta functions}\label{subsec: strong positivity}

We define structure constants for theta functions following \cite[Definition-Lemma 6.2]{GHKK}, with the difference that now they are not just integers but in the coefficient ring $\mathbb Z[\mathbb P^\oplus]$. The argument carries over and the positivity follows directly from \Cref{thm: positive wall-function}. Remind that we are given a generalized cluster scattering diagram $\frakD_\mathbf s$.

\begin{defnlem}\label{def: struct constant}
    Let $p_1, p_2, q\in M$. Let $z\in M_\mathbb R$ be chosen generally. There are at most finitely many pairs of broken lines $\gamma_1, \gamma_2$ with $m_0(\gamma_1) = p_1$, $m_0(\gamma_2) = p_2$ and both ending at $z$, and $m(\gamma_1) + m(\gamma_2) = q$. Define
    \[
        \alpha_z(p_1, p_2; q) = \sum_{(\gamma_1, \gamma_2)} c(\gamma_1)c(\gamma_2)
    \]
    where the finite sum is over the aforementioned pairs of broken lines. Hence $\alpha_z(p_1, p_2; q)$ is in $\mathbb N[\mathbb P^\oplus]$ by \Cref{thm: positive wall-function}.
\end{defnlem}

Now we can describe the product structure of theta functions using $\alpha_z(p_1, p_2; q)$. Let $Q$ be general in $\mathcal C_{v_0}^+$. Write $\vartheta_{m} = \vartheta_{m, Q}$ for any $m\in M$. The fact that the structure constants $\alpha(p_1, p_2; q)$ are in $\mathbb N[\mathbb P^\oplus]$ is what we refer to as the \emph{strong positivity}.

\begin{thm}[{cf. \cite[Proposition 6.4(3)]{GHKK}}]\label{thm: strong positivity}
    The theta functions $\{\vartheta_q\mid q\in M\}$ form a topological basis of certain formal completion of a variant $\mathcal U^\oplus$ of the upper cluster algebra $\mathcal U$. For any $p_1, p_2\in M$, we have
    \begin{equation}\label{eq: product struct theta}
        \vartheta_{p_1} \cdot \vartheta_{p_2} = \sum_{q\in M} \alpha(p_1, p_2; q)\vartheta_q \in \widehat{\Bbbk[\mndP]}
    \end{equation}
    where $\alpha(p_1, p_2, q) \coloneqq \alpha_{z(q)}(p_1, p_2; q)$ for $z(q)$ chosen sufficiently near $q$ and is independent of this choice.
\end{thm}

\begin{proof}
    The proof essentially follows that of \cite[Proposition 6.4(3)]{GHKK}. We first understand conceptually where the theta functions live. In fact, for each $v\in \mathbb T_v$, we have the algebra
    \[
        \Z[\mathbb P^\oplus][x_{i;v}^{\pm 1}\mid i\in I] \subseteq \mathcal F.
    \]
    Denote the intersection of all these algebras by $\mathcal U^\oplus$ (by $\mathcal U$ if taking the coefficient ring $\Z[\mathbb P]$ instead). It contains all cluster monomials. In fact, when expressing $x_{i;v}$ in another cluster $\mathbf x_{v'}$, it is the theta function $\vartheta_{Q_{v'}, g_{i;v}}$ for some general $Q_{v'}\in \mathcal C_{v'}^+$ by \eqref{eq: pop cluster var high} and \eqref{eq: change endpoint theta high}, so the coefficients are actually in $\N[\mathbb P^\oplus]$. Take the completion 
    \[
        \widehat{\mathcal U^\oplus}\coloneqq \lim_{\longleftarrow} \mathcal U^\oplus/I^{k}
    \]
    with respect to the ideal $I$ generated by $\{p_{i,s}\mid i\in I_\mathrm{uf},1\leq s\leq \ell_i\}$. For the same reason $\mathcal U$ contains cluster monomials, $\widehat{\mathcal U^\oplus}$ contains the theta functions $\{\vartheta_m\mid m\in M\}$.
    
    We note that $\mathcal U^\oplus/I = \Bbbk[M]$ since letting $p_{i,s} = 0$ makes any $\vartheta_{Q_{v'}, g_{i;v}}$ a single monomial $x^{g_{i;v}}$. This phenomenon is phrased in \cite[Corollary 5.3]{GHKK} as the central fiber of the family $\mathcal U_\mathrm{prin}^\mathbf s\rightarrow \mathbb A_{X_1, \dots, X_n}^n$ is the torus $T_N$. Here we are directly taking the regular functions instead of considering the relevant schemes. The family is in the ordinary case, thus recovered by restricting to all $p_{i,s} = 0$ except for the leading ones $\{p_{i, \ell_i}\mid i\in I_\mathrm{uf}\}$, which correspond to the coordinates $\{X_1, \dots, X_n\}$.

    Back to the product structure, now it suffices to show (\ref{eq: product struct theta}) in $\mathcal U^\oplus/I^{k+1}$ for each $k\geq 0$. Now the theta functions can be defined using the finite scattering diagram $(\frakD_{\mathbf s})_k$ over $\Bbbk[\mndP]/\mathfrak m^{k+1}$. As argued in the first paragraph of the proof of \cite[Proposition 6.4]{GHKK}, theta functions $\{\vartheta_q\mid q\in M\}$ form a $\Z[\mathbb P^\oplus]/I^{k+1}$-basis for $\mathcal U^\oplus/I^{k+1}$; it is clear when $k=0$ as $\vartheta_q = x^q$ and one can use Nakayama's lemma to lift the basis $\{x^q\mid q\in M\}$ to $\{\vartheta_q\mid q\in M\}$ in $\mathcal U^\oplus/I^{k+1}$. We thus say that $\{\vartheta_q\mid q\in M\}$ form a \emph{topological basis} of $\widehat{\mathcal U^\oplus}$.

    For a given $k\geq 1$, express $\vartheta_{p_1}\cdot \vartheta_{p_2}$ in the basis $\{\vartheta_q\mid q\in M\}$. Notice we can choose any general terminal point for this expression because for different terminal points, theta functions are transformed under a path-ordered product which does not change coefficients in $\Bbbk[\mathbb P^\oplus]/I^k$. For a given $q\in M$, choose a general terminal point $z$ close enough to $q$ so that they belong to the same chamber. By the definition of $\alpha_z$, we have
    \begin{equation}\label{eq: def struct const proof}
        \vartheta_{z, p_1}\cdot \vartheta_{z, p_2} = \sum_r \alpha_z(p_1, p_2;r) x^r.
    \end{equation}
    However, among all broken lines for all $\vartheta_{z, r}$, $r\in M$, there is only one broken line (whose image is $z + \mathbb R_{\geq 0}q$) with associated final monomial $x^q$, which gives the theta function $\vartheta_{z, q} = x^q$. Thus the coefficient of $\vartheta_{z, q}$ can be read off by the coefficient of $x^q$ on the right-hand side of (\ref{eq: def struct const proof}), namely $\alpha_z(p_1, p_2; q)$. This proves the product formula in order $k$. To see the independence of $z$ for all $k$, we observe exactly as in the proof of \cite[Definition-Lemma 6.2]{GHKK} that for given $p_1, p_2, q$ only finitely many wall-function terms in $\frakD$ are relevant to the possible pairs of broken lines computing $\alpha_z(p_1, p_2;q)$ (for any general $z$). Hence one can choose a big enough $k$ and use $\frakD_k$ to define $\alpha(p_1, p_2; q)$.
\end{proof}

In general, the cluster complex $\Delta_\mathbf s^+$ is only a proper subset of $M_\mathbb R$. Theta functions outside $\Delta_\mathbf s^+$ may not be Laurent but a formal sum. The following observation generalizes \cite[Proposition 7.1]{GHKK} immediately once we have the positivity of wall-crossing functions \Cref{thm: positive wall-function}.

\begin{prop}\label{prop: finite independent chamber}
    Let $m_0\in M$. If for some general point $Q\in \mathcal C_v^+$ for $v\in \mathbb T_n$ there are only finitely many broken lines for $m_0$ with endpoint $Q$, the same is true for any general point $Q \in \mathcal C_{v'}^+$ for any other $v'\in \mathbb T_n$. In this case, the theta function $\vartheta_{m}$ is a positive universal Laurent polynomial.
\end{prop}

We can now extend \cite[Theorem 7.5]{GHKK} to the generalized case. 

\begin{thm}\label{thm: middle cluster algebra}
    Let $\Theta\subseteq M$ be the collection of $m$ such that the theta function $\vartheta_{m}$ is a finite sum, that is, by \Cref{prop: finite independent chamber} a universal Laurent polynomial.
    \begin{enumerate}
        \item $\Delta_\mathbf s^+\cap M\subseteq \Theta$.
        \item For $p_1, p_2\in \Theta$,
        \[
            \vartheta_{p_1}\cdot \vartheta_{p_2} = \sum_r \alpha(p_1, p_2; r)\vartheta_r 
        \]
        is a finite sum with coefficients in $\mathbb N[\mathbb P^\oplus]$. If $\alpha(p_1, p_2; r)\neq 0$, then $r\in \Theta$.
        \item $\Theta$ is a saturated monoid in $M$.
        \item The $\Bbbk[\mathbb P]$-module
        \[
            \mathcal M \coloneqq \operatorname{span}_{\Z[\mathbb P]}\{\vartheta_m\mid m\in \Theta\}
        \]
        is an associative $\Z[\mathbb P]$-algebra with a $\Z[\mathbb P]$-basis given by $\{\vartheta_m\mid m\in \Theta\}$. We have natural inclusions
        \[
            \mathcal A\subseteq \mathcal M \subseteq \mathcal U.
        \]
    \end{enumerate}
\end{thm}

\begin{proof}
    The proof uses the exact same arguments as in \cite[Theorem 7.5]{GHKK} but relies on the generalized positivity \Cref{thm: positive wall-function}. We note that the inclusion $\mathcal A\subseteq \mathcal M$ sends a cluster monomial to the corresponding theta function given by \Cref{thm: cluster monomial theta} and the second inclusion $\mathcal M\subseteq \mathcal U$ is guaranteed by \Cref{prop: finite independent chamber}.
\end{proof}

\begin{rem}
    In the ordinary case, the algebra $\mathcal M$ is denoted as $\operatorname{mid}(\mathcal A)$ (where $\mathcal A$ denotes the cluster variety) and called the \emph{middle cluster algebra} in \cite{GHKK}. The set $\Theta$ a priori depends on the choices of $\ell_i$. In particular, when $\ell_i=1$ for any $i\in I_\mathrm{uf}$, it is the corresponding set for ordinary cluster algebras. It is interesting to ask whether the generalized and ordinary situations share the same $\Theta$.
\end{rem}

\subsection{Theta functions are atomic}\label{subsec: atomicity}

A non-zero element in the upper cluster algebra $\mathcal U$ is called \emph{universally positive} \cite{FG09} if it is a Laurent polynomial with coefficients in $\mathbb N[\mathbb P]$ in every cluster, that is, it is in the subset
\[
    \bigcap_{v\in \mathbb T_{|I_\mathrm{uf}|}} \mathbb N[\mathbb P][x_{i;v}^{\pm 1}\mid i\in I] \subseteq \mathcal U.
\]
Such an element is furthermore called \emph{indecomposable} or \emph{atomic} if it cannot be written as a sum of two universally positive elements. Any cluster monomial is universally positive by \Cref{thm: positive Laurent} and is atomic since it is a monomial in some cluster.

Theta functions are not necessarily in $\mathcal U$ but live in the completion $\widehat{\mathcal U^\oplus}$. We follow Mandel's definition \cite{Man17} of atomicity using the scattering diagram $\frakD_\mathbf s$.

\begin{defn}
    A non-zero element $f\in \widehat{\mathcal U^\oplus}$ is called \emph{universally positive} with respect to the \emph{scattering atlas} if regarded as an element in $\widehat{\Bbbk[\mndP]}$, for any general $Q\in \mathcal C_\mathbf s^+$, any general $Q'\in M_\mathbb R\setminus \operatorname{Supp}(\frakD_\mathbf s)$, and a regular path $\gamma$ from $Q$ to $Q'$, the element
    \[
        \frakp_{\gamma}(f)\in \widehat{\Bbbk[\mndP]}
    \]
    is a formal Laurent series with coefficients in $\mathbb N[\mathbb P^\oplus]$. It is in addition called \emph{atomic} if it cannot be written as a sum of two elements that are universally positive with respect to the scattering atlas.
\end{defn}

\begin{thm}[{cf. \cite[Theorem 1]{Man17}}]\label{thm: atomicity}
    The theta functions $\{\vartheta_q\mid q\in M\}$ (up to a scalar in $\mathbb P$) are exactly the atomic elements of $\widehat{\mathcal U^\oplus}$ with respect to the scattering atlas.
\end{thm}

\begin{proof}
    It follows directly from \Cref{thm: positive wall-function} and (\cite[Theorem 3.5]{GHKK} and \cite{CPS})
    \[
        \mathfrak p_{\gamma}(\vartheta_{Q, q}) = \vartheta_{Q', q}
    \]
    that theta functions are universally positive with respect to the scattering atlas. Then it remains to show that for any $f$ universally positive with respect to the scattering atlas, the expansion
    \[
        f = \sum_{q\in M} a_q \vartheta_q
    \]
    has $a_q\in \mathbb N[\mathbb P^\oplus]$ for any $q\in M$. However, for $r\in M$, we can take $Q'$ sufficiently close to $r$ as in the proof of \Cref{thm: strong positivity} so that in the expansion
    \[
        \frakp_{\gamma}(f) = \sum_{q\in M} c_q x^q,
    \]
    we have $a_r = c_r\in \mathbb N[\mathbb P^\oplus]$ by the universal positive assumption on $f$.
\end{proof}

\section{Positivity of wall-functions}\label{sec: proof of positivity}

This section is devoted to positivity of scattering diagrams. We prove in \Cref{subsec: full positivity rk 2} that in rank 2, any consistent scattering diagram is positive with respect to the initial coefficients. In \Cref{subsec: positivity higher rank}, we prove \Cref{thm: positive wall-function}, the positivity of higher-rank generalized cluster scattering diagrams.

\subsection{The non-standard case}\label{subsec: non-standard}

The scattering diagrams considered in \Cref{subsec: gen cluster sd rk 2} are regarded as \emph{standard}, as the directions of the two initial walls are standard unit vectors $e_1$ and $e_2$. The non-standard case can be treated by the \emph{change-of-lattice trick} already mentioned in the proof of \cite[Proposition C.13]{GHKK}.

Let $M = \mathbb Z^2$. Suppose that $m_1, m_2\in M$ are primitive, linearly independent, but may only generate a rank-2 proper sublattice $M^\circ \subsetneq M$. The dual $N^\circ \coloneqq \operatorname{Hom}_\mathbb Z(M^\circ, \mathbb Z)$ naturally contains $N \coloneqq \operatorname{Hom}_\mathbb Z(M, \mathbb Z)$ as a sublattice.

For any $m\in M$, denote $m^\perp \coloneqq \{n\in N\otimes \mathbb R \mid n\cdot m = 0\}$. We define the \emph{index} of $m$ to be
\[
    \operatorname{ind}(m) \coloneqq |(m^\perp \cap N^\circ)/(m^\perp \cap N)|, 
\]
that is, the order of the finite cyclic group $(m^\perp \cap N^\circ)/(m^\perp \cap N)$. For example, we have
\begin{equation}\label{eq: def of index}
    \operatorname{ind}(m_1) = \operatorname{ind}(m_2) = |M/M^\circ| \eqqcolon \ell.
\end{equation}
Notice that $\operatorname{ind}(m)$ only depends on the line $\mathbb Rm$. We then define $\operatorname{ind}(\frakd) \coloneqq \mathrm{ind}(m_\frakd)$,
where $m_\frakd$ is the direction of either a ray or a line $\frakd$.

Take $\mndP = M\oplus \mnd$ to be the monoid in \Cref{subsec: gen cluster sd rk 2} with the projection $r\colon \mndP \rightarrow M$. Let $\frakD_\mathrm{in} = \{\frakd_1, \frakd_2\}$ where
\[
    S_{\frakd_i} = \mathbb Rm_i \quad \text{and} \quad f_{\frakd_i} = f_i \coloneqq 1 + \sum_{k = 1}^{\ell_i} p_{i, k} x^{km_i}, \quad i = 1, 2.
\]
By \Cref{thm: rank 2 consistent}, we can assume that there is at most one added ray in $\frakD$ with support $\mathbb R_{\leq 0}(pm_1 + qm_2)$ for any positive coprime pair $(p, q)$. Assuming that $m_1, m_2$ form a basis of $M$, we can express the attached wall-function as
\begin{equation}\label{eq: standard commutator}
    f_{\mathbb R_{\leq 0}(pm_1 + qm_2)} = 1 + \sum_{d\geq 1} \lambda(dp, dq)x^{dpm_1 + dqm_2},
\end{equation}
where $\lambda(dp, dq)\in \N[\mnd]$ is computed in \Cref{thm: power shadow grading formula}.

\begin{defn}\label{def: scale by l}
    Define $\clco{dp}{dq}{\ell}$ for coprime $(p, q)\in \mathbb N^2$, $d\in \mathbb N$, $\ell\in \mathbb N$ by the following procedure. First we evaluate $\lambda(\ell dp, \ell dq)$ at $p_{i,k}=0$ for any $k$ not divisible by $\ell$. Then we substitute any $p_{i, \ell k}$ by $p_{i, k}$.
\end{defn}

\begin{exmp}
    Let $(p, q)=(1, 1)$, $d=1$, $\ell=2$. For $\ell_1 = \ell_2 = 2$, we compute $\clco{1}{1}{2}$. By \Cref{thm: power shadow grading formula}, we have 
    \[
        \lambda(2, 2) = p_{1,1}^2p_{2,2} + p_{2,1}^2p_{1,2} + 2p_{1,2}p_{2,2}.
    \]
    After the evaluation at $p_{1,1}=p_{2,1}=0$, we get $2p_{1,2}p_{2,2}$. Substituting $p_{i,2}$ by $p_{i,1}$, we obtain $\clco{1}{1}{2} = 2p_{1,1}p_{2,1}$.
\end{exmp}

The following formula computes wall-functions in the non-standard case.

\begin{thm}\label{thm: positivity rk 2 non-standard}
    The wall-function in $\mathfrak D$ on any ray $\mathbb R_{\leq 0}m^\circ$ with $m^\circ = pm_1 + qm_2$ and $(p, q)$ positive coprime is
    \[
        f_{\mathbb R_{\leq 0}m^\circ} = \left(1 + \sum_{d\geq 1} \clco{dp}{dq}{\ell} x^{dpm_1 + dqm_2}\right)^{\frac{\ell}{\mathrm{ind}(m^\circ)}}.
    \]
    where $\ell$ is defined in \eqref{eq: def of index} and $\ell/\operatorname{ind}(m^\circ)$ belongs to $\mathbb N$. Hence every coefficient in the expansion of $f_{\mathbb R_{\leq 0}m^\circ}$ is a polynomial in $p_{i, j}$ with nonnegative integer coefficients.
\end{thm}

\begin{proof}

Let $\frakd$ be a wall in $\mathfrak D$. The associated wall-crossing automorphism is
\[
    x^{m'} \mapsto x^{m'} f_\frakd^{\pm n \cdot m'},
\]
where $n$ is primitive in $N$ and orthogonal to $m_\mathfrak d$. Notice that the index of $\frakd$ is equivalently
\[
    \mathrm{ind}(\frakd) \coloneqq \max\{d\in \mathbb N \mid n/d \in N^\circ \}.
\]
Namely $n^\circ \coloneqq n/\mathrm{ind}(\frakd)$ is primitive in $N^\circ$.

Consider the scattering diagram
\[
    \mathfrak D^\circ \coloneqq \left\{\left(S_\frakd, f_\frakd^{\operatorname{ind}(\frakd)}\right)\,\middle\vert\, \frakd \in \mathfrak D \right\}.
\]
We claim that it is the consistent scattering diagram $\Scat(f_1^\ell, f_2^\ell)$ for the lattice $M^\circ$. In fact, for $m'\in M^\circ$, the wall-crossing automorphism associated to $(S_\frakd, f_\frakd^{\operatorname{ind}(\frakd)})$ can be expressed as
\[
    x^{m'} \mapsto x^{m'} \left( f_\frakd^{\mathrm{ind}(\frakd)} \right)^{\pm n^\circ\cdot m'} = x^{m'} f_\frakd^{\pm n \cdot m'}
\]
where $n$ and $n^\circ$ are respectively the primitive normal vectors of $S_\frakd$ in $N$ and $N^\circ$. It follows immediately that $\frakD^\circ$ is consistent because so is $\frakD$. Thus by the uniqueness, we have $\frakD^\circ = \Scat(f_1^\ell, f_2^\ell)$.

In retrospect, we can use $\frakD^\circ$ to represent $\mathfrak D$. This in particular implies that any wall-function $f$ of $\mathfrak D$ is always a power series of $x^{m^\circ}$ for some primitive $m^\circ$ in $M^\circ$. Precisely, consider $(p, q)$ positive coprime and $m^\circ = pm_1 + qm_2\in M^\circ$. The wall-function $f_{\mathbb R_{\leq 0} m^\circ}$ in $\mathfrak D$ is then recovered by
\begin{equation}\label{eq: reconstruct change lattice}
    f_{\mathbb R_{\leq 0} m^\circ} = \left(f_{\mathbb R_{\leq 0} m^\circ}^\circ \right)^{\frac{1}{\mathrm{ind}(m^\circ)}}
\end{equation}
where $f^\circ_{\mathbb R_{\leq 0} m^\circ}$ denotes the wall-function in $\frakD^\circ$. This is not so ideal as taking the power of a fraction may break the integrality and the positivity.

However, we can \emph{change the lattice} one more time. Consider the scattering diagram
\[
    \mathfrak D' = \Scat(f'_1, f'_2)
\]
again with respect to $M^\circ$ where for $i=1, 2$
\[
    f_i' \coloneqq f_i(x^{\ell m_i}) =  1 + \sum_{d\geq 1} p_{i, d}x^{\ell d m_i}.
\]
For the ray with support $\mathbb R_{\leq 0}m^\circ $ where $m^\circ = pm_1 + qm_2$ for $(p, q)$ positive coprime, denote the wall-function by $f'_{\mathbb R_{\leq 0}m^\circ}$. By \eqref{eq: standard commutator} and the substitution \Cref{def: scale by l}, we have
\[
    f'_{\mathbb R_{\leq 0}m^\circ} = 1 + \sum_{d\geq 1} \clco{dp}{dq}{\ell} x^{dpm_1 + dqm_2}.
\]

Now consider the wall-crossing automorphism associated to the above ray in $\mathfrak D'$ on $x^{a_1\ell m_1 + a_2 \ell m_2}$ with $a_i\in \mathbb Z$:
\[
    x^{a_1\ell m_1 + a_2 \ell m_2} \mapsto x^{a_1\ell m_1 + a_2 \ell m_2} (f'_{\mathbb R_{\leq 0}m^\circ})^{\pm \ell(pa_2 - qa_1)}.
\]
Observe that these wall-crossings can be viewed as in a scattering diagram for $M'$ generated by $\ell m_1$ and $\ell m_2$. Therefore, by turning every $f'_{\mathbb R_{\leq 0}m^\circ}$ into
\begin{equation}\label{eq: change lattice again}
    f^\circ_{\mathbb R_{\leq 0}m^\circ} \coloneqq \left(f'_{\mathbb R_{\leq 0}m^\circ} \mid x^{m_i} \leftarrow x^{\ell m_i}\right)^\ell,
\end{equation}
we obtain a consistent scattering diagram for $M^\circ$ with the initial wall-functions $f_1^\ell$ and $f_2^\ell$. By the uniqueness, this is exactly $\mathfrak D^\circ$, hence justifying the previously defined notation $f_\frakd^\circ$.

Now combining (\ref{eq: reconstruct change lattice}) with (\ref{eq: change lattice again}), the wall-function for $\mathfrak D$ can be recovered as
\[
    f_{\mathbb R_{\leq 0} m^\circ} = \left(1 + \sum_{d\geq 1} \clco{dp}{dq}{\ell} x^{\ell d (pm_1 + qm_2)}\right)^{\frac{\ell}{\mathrm{ind}(m^\circ)}},
\]
as stated in the theorem. Since for any $m\in M$, we have
\[
    \ell(m^\perp \cap N^\circ) \subseteq m^\perp \cap N \subseteq m^\perp \cap N^\circ,
\]
the index $\operatorname{ind}(m^\circ)$ divides $\ell$.
\end{proof}

\begin{cor}\label{cor: two lines any monoid}
    Let $\frakD_\mathrm{in} = \{\frakd_1, \frakd_2\}$ be a scattering diagram of two lines, where
    \[
        f_{\frakd_i} = 1 + \sum_{m\in \mndP(m_i)} c_{i,m} x^m \in \widehat{\Bbbk[\mndP]},\quad i = 1,2.
    \]
    Then there is a consistent scattering diagram in the equivalent class of $\Scat(\frakD_\mathrm{in})$ such that any coefficient $c_m\in \Bbbk$ of any wall-function
    \[
        f_\frakd = 1 + \sum_{m\in\mndP(m_\frakd)} c_{m} x^m
    \]
    can be expressed as a polynomial in $\{c_{i,m}\mid i=1,2, m\in \mndP\}$ with nonnegative integer coefficients.
\end{cor}

\begin{proof}
    Notice that we do not assume $\mndP = M\oplus \mnd$ and thus $\frakD_\mathrm{in}$ is not exactly of the form in \Cref{thm: positivity rk 2 non-standard}. It suffices to show the positivity at each order $k\geq 1$ for $\Scat(\frakD_\mathrm{in})_k$. For fixed $k$, we can choose large $\ell_1, \ell_2$ and consider a close variant $\frakD = \Scat(f_1, f_2)$ of the scattering diagram in \Cref{thm: positivity rk 2 non-standard} where
    \[
        f_i = 1 + \sum_{k=1}^{\ell_i}p_{i,k},    
    \]
    with the monoid $\mnd_{\ell_1, \ell_2}$ the monomials in $\{p_{i,k}\mid i = 1, 2, 1\leq k\leq \ell_i\}$ equipped with $r'\colon \mnd_{\ell_1, \ell_2}\rightarrow M$ such that $r'(p_{i,k})=km_i$. Then consider the algebra homomorphism
    \[
        \varphi\colon \widehat{\Bbbk[\mnd_{\ell_1,\ell_2}]} \rightarrow \widehat{\Bbbk[\mndP]},\quad \varphi(p_{i,k}) = \sum_{r(m)=km_i} c_{i,m}x^m.
    \]
    By the functoriality \eqref{eq: functoriality complete}, we have $\varphi(\frakD) \equiv \Scat(\varphi(f_1), \varphi(f_2))$. The positive integers $\ell_1, \ell_2$ are sufficiently large so that $f_{\frakd_i} \equiv \varphi(f_i)\mod \frakm^{k+1}$. Then the positivity of $\Scat(\frakD_\mathrm{in})_k \equiv \varphi(\frakD)_k$ follows from that of $\frakD$ by \Cref{thm: positivity rk 2 non-standard}.
\end{proof}

\subsection{Full positivity in rank 2}\label{subsec: full positivity rk 2} 

We now show that any rank-2 consistent scattering diagram is positive with respect to the initial coefficients.

Let $\frakD_\mathrm{in} = \{\frakd_s\mid s\in S\}$ be a scattering diagram (as defined in \Cref{subsec: sd in rk 2}) of lines, where $S$ is a countable index set. We do not require the lines to pass through the origin. The initial wall-functions are expressed as
\[
    f_{\frakd_s} = 1 + \sum_{m\in \mndP(m_{\frakd_s})} c_{s,m} x^m\in \widehat{\Bbbk[\mndP]}.
\]

\begin{thm}\label{thm: full positivity rk 2}
    There exists a consistent scattering diagram $\frakD$ in the equivalent class of $\Scat(\frakD_\mathrm{in})$ such that any coefficient $c_m\in \Bbbk$ of any wall-function
    \[
        f_\frakd = 1 + \sum_{m\in \mndP} c_m x^m
    \]
    can be expressed as a polynomial in $\{c_{s,m}\mid s\in S, m\in \mndP\}$ with nonnegative integer coefficients, which we call a positive polynomial.
\end{thm}

\begin{proof}
    We prove by inductively constructing a finite scattering diagram $\frakD_k$ over $\Bbbk[\mndP]/\frakm^{k+1}$ for each $k\geq 1$ such that $\frakD_k \equiv \Scat(\frakD_\mathrm{in})_k$ and that the positivity holds for $\frakD_k$. This is clearly true for $k = 1$ as we simply take $\frakD_1 = (\Din)_1$. We add to the induction hypothesis that $\frakD_k$ is a union of $(\Din)_k$ and a set of rays such that no two rays have the same support.

    \emph{Step 1: Reduction to the cone case.} 
    
    We first reduce to the case where every initial line passes through the origin, which we refer to as the \emph{cone case}.
    
    For a singular point $p\in \operatorname{Sing}(\frakD_k)$, define
    \[
        \frakD_k(p) \coloneqq \{\frakd\in \frakD_k\mid p\in S_\frakd\}.
    \]
    When viewed as over $\Bbbk[\mndP]/\frakm^{k+2}$, $\frakD_k(p)$ may fail to be consistent at $p$. Translating the support of every wall in $\frakD_k(p)$ by $-p$, we obtain a scattering diagram $\frakD_k^\mathrm{loc}(p)$ centered at the origin, which can be viewed as a local version of $\frakD_k(p)$ at $p$. Denote the set of lines in $\frakD_k^\mathrm{loc}(p)$ by $\frakD_{k, \mathrm{in}}^\mathrm{loc}(p)$.

    Assuming the inductive construction is valid in the cone case, we can modify $\frakD_k(p)$ as follows. We can add a finite collection $\frakD_k[p]$ of new rays based at $p$  whose wall-functions are in $1 + \frakm^{k+1}\setminus \frakm^{k+2}$ and modify existing rays $\frakd\mapsto \frakd'$ by only updating their wall-functions $f_\frakd$ by adding monomials in $\frakm^{k+1}\setminus \frakm^{k+2}$ where the coefficients of add terms are positive polynomials such that
    \[
        \widetilde \frakD_{k}(p)\coloneqq \frakD_k[p] \cup \{\frakd'\mid \frakd\in \frakD_k(p), p = \partial S_\frakd \}\cup\{\frakd\in \frakD_k(p)\mid p\in S_\frakd\cut \partial S_\frakd\}
    \]
    is consistent at $p$ over $\Bbbk[\mndP]/\frakm^{k+2}$. Notice that this means $\widetilde \frakD_{k}^\mathrm{loc}(p) \equiv \Scat(\frakD_{k,\mathrm{in}}^\mathrm{loc}(p)) \mod \frakm^{k+2}$.

    Globally, we construct
    \[
        \frakD_{k+1} \coloneqq (\frakD_\mathrm{in})_{k+1} \cup \bigcup_{p\,\in \,\operatorname{Sing}(\frakD_k)} \left(\frakD_k[p] \cup \{\frakd'\mid \frakd\in \frakD_k,\, p = \partial S_\frakd\} \right).
    \]
    We show that the induction hypothesis holds for $\frakD_{k+1}$.
    
    Clearly $\frakD_{k+1}$ is positive. To show $\frakD_{k+1}\equiv \Scat(\Din)_{k+1}$, it suffices to show the consistency for $\frakD_{k+1}$ because they share the same set of lines. Then we can check locally at each singular point $p\in \Sing(\frakD_{k+1})$. Suppose that $p\in \Sing(\frakD_k)$. For each $\frakd\in \frakD_{k+1}(p)$ such that $p\in S_\frakd\cut \partial S_\frakd$, we write
    \[
        f_{\frakd} = f'_\frakd \cdot (1 + \xi) \mod \frakm^{k+2},
    \]
    where $\xi$ is the part of $f_\frakd$ in $\frakm^{k+1}\cut\frakm^{k+2}$. The wall-crossing automorphism (and its inverse) associated to $1 + \xi$ commutes with others modulo $\frakm^{k+2}$. In the path-ordered product of a small loop around $p$, such wall-crossing automorphism and its inverse both appear once and thus cancel out. Therefore to check the consistency at $p$, we can replace $f_\frakd$ with $f_\frakd'$. After the replacement for every such $\frakd$, $\frakD_{k+1}(p)$ is equivalent to $\widetilde{\frakD}_k(p)$, hence consistent. A new singular point $p\in \Sing(\frakD_{k+1})\cut \Sing(\frakD_k)$ must be in the interior of some newly added ray. There is no ray based at $p$. For consistency at $p$, we can again ignore any newly added ray that contains $p$ in its interior. The rest walls containing $p$ are all in the same or the opposite direction, whose wall-crossing automorphisms commute and cancel out in the path-ordered product of a small loop around $p$. Thus $\frakD_{k+1}$ is consistent at $p$.

    Then the induction hypothesis is proven for $k+1$ by assuming that this can be done for the cone case.

    \emph{Step 2: Perturbation trick.} 
    
    We use the \emph{perturbation trick} to handle the cone case. Choose for each $\frakd\in \frakD_\mathrm{in}$ a sufficiently general vector $v_\frakd\in M_\mathbb R$ and consider the \emph{perturbed scattering diagram} $\frakD'_\mathrm{in}$ with supports $S_\frakd + v_\frakd$ but the same wall-functions. After the perturbation, at most two initial lines intersect at any given point.

    By the induction hypothesis, we have constructed positive $\frakD'_k$ equivalent to $\Scat(\frakD_\mathrm{in}')_k$, where $\frakD'_1 = (\frakD'_\mathrm{in})_1$. Because of the perturbation, for each $p\in \Sing(\frakD'_k)$, there are at most two lines in $\left(\frakD^{\prime, \mathrm{loc}}_{k,\mathrm{in}}(p)\right)_1$, that is, at most two lines are non-trivial modulo $\frakm^2$. We want to show $\widetilde \frakD'_k(p)$ is positive at each $p$. Notice that this is equivalent to show the induction hypothesis at $k+1$ for the initial scattering diagram $\frakD^{\prime, \mathrm{loc}}_{k,\mathrm{in}}(p)$. Once this is established, we can construct $\frakD'_{k+1}$ as in \emph{Step 1}. To get $\frakD_{k+1}$, we then translate every wall in $\frakD'_{k+1}$ back to pass through or based at the origin. Therefore it suffices to prove the induction step in the cone case with at most two non-trivial lines in $(\frakD_\mathrm{in})_1$

    \emph{Step 3: Reduction to two initial lines.} 

    By \emph{Step 2}, we assume in the cone case that only $\frakd_1$ and $\frakd_2$ are non-trivial in $(\frakD_\mathrm{in})_1$. We may assume that $S = \{1, \dots, N\}$ since only finitely many initial lines are non-trivial up to a sufficiently large order. We want to show the induction hypothesis at $k+1$, that is, on any ray in $\frakD_{k+1} = \Scat(\Din)_{k+1}$, every term $c_px^p$ in $\frakm^{k+1}\cut\frakm^{k+2}$ has positive polynomial coefficient $c_p$. Notice that
    \[
        p = \sum_{i, s} a_{i,s}p_{i,s}
    \]
    for $a_{i,s}\in \N$, where $p_{i,s}\in \mndP$ appears in $f_{\frakd_s}$. If any $a_{i,s}\neq 0$ for some $s\neq 1, 2$, then at most $k$ of $a_{i,s}$ are non-zero. The crucial point is that $c_px^p$ can be viewed as a term of order at most $k$ after the \emph{change of monoid trick} as follows. 
    Express each $f_{\frakd_s}$ as
    \[
        f_{\frakd_s} = 1 + \sum_{i=1}^{\ell_s} c_{i,s} x^{p_{i,s}} \mod \frakm^{k+2}.
    \]
    Consider the monoid
    \[
        \mnd = \bigoplus_{s=1}^N \bigoplus_{i=1}^{\ell_s} \N \cdot e_{i,s}
    \]
    with $r'\colon \mnd \rightarrow M$ such that $r'(e_{i,s}) = r(p_{i,s})$. Denote $\frakm' = \mnd\cut \{0\}$. Consider the scattering diagram
    \[
        \frakD' = \Scat\left(1 + \sum_{i=1}^{\ell_s} c_{i,s} x^{e_{i,s}}\,\middle \vert\, s=1, \dots, N\right).
    \]
    By the functoriality \eqref{eq: functoriality complete}, we can obtain $c_px^p$ by a wall-function term $c_px^{\sum_{i,s}a_{i,s}e_{i,s}}$ in $\frakD'$. Notice that now $\sum_{i,s} a_{i,s}e_{i,s}$ is in $(\frakm')^j\cut (\frakm')^{j+1}$ for some $j\leq k$. So we apply the induction hypothesis at $k$ to this situation to conclude that $c_p$ is positive with respect to the initial coefficients.
    
    Now we are left with $p$ as a linear combination of $p_{i,s}$ for $s=1, 2$ only. Notice that these same terms appear in $\Scat(\frakd_1, \frakd_2)$ as one can simply let $x^{p_{i,s}} = 0$ for $s\neq 1, 2$. Therefore we have reduced the problem to prove the induction hypothesis at $k+1$ for the case of two initial lines.

    \emph{Step 4.} When $\frakD_\mathrm{in}$ only consists of two initial lines, the positivity follows immediately from \Cref{cor: two lines any monoid}.

    Now the induction is complete.
\end{proof}

\subsection{Positivity in higher ranks}\label{subsec: positivity higher rank}

We proceed to prove \Cref{thm: positive wall-function} the positivity of wall-functions in $\frakD_\mathbf s$ as defined in \Cref{subsec: gen cluster sd higher rk}. The strategy is to construct a sequence of scattering diagrams $(\frakD_k)_{k\geq 1}$ such that each $\frakD_k$ is over $\widehat{\Bbbk[\mndP]}/\frakm^{k+1}$, has positive wall-functions, and is equivalent to $(\frakD_{\mathbf s})_k$.

We first realize the singular locus (\ref{eq: support and sing locus}) of a finite scattering diagram as a cone complex. The support of each wall, as a polyhedral cone, is a finite intersection of closed linear half-spaces. Consider the collection of these half-spaces and their other halves from all the walls in $\frakD$. Some of their intersections are contained in $\operatorname{Sing}(\frakD)$ and $\operatorname{Sing}(\frakD)$ is a union of these intersections. They give rise to a cone complex structure on $\operatorname{Sing}(\frakD)$.

\begin{defn}\label{def: cone complex sing}
    For any scattering diagram $\frakD$ of finitely many walls, the singular locus $\operatorname{Sing}(\frakD)$ can be realized as the union of the maximal cones of a cone complex of dimension $\dim M_\mathbb R - 2$ such that for each such cone $\mathfrak j$ and $x\in \operatorname{Int}(\mathfrak j)$ (the interior points of $\mathfrak j$), the set of walls 
    \[
        \{\frakd \in \mathfrak D \mid x\in S_\frakd \}
    \]
    is independent of $x$. With such a realization, the maximal cones are called \emph{joints}. In particular, any facet of $S_\frakd$ is a union of joints with disjoint interiors.
\end{defn}

The aforementioned cone complex structure on $\operatorname{Sing}(\frakD)$ clearly satisfies the requirement in the above definition, thus giving a choice of joints.

Let $\mathfrak D_1$ be $(\mathfrak D_{\mathrm{in}, \mathbf s})_1$. Concretely, the walls in $(\mathfrak D_{\mathrm{in}, \mathbf s})_1$ are still $\frakd_i$ for $i\in I_\mathrm{uf}$ with the support $S_{\frakd_i} = e_i^\perp$ and the wall-function 
\[
    f_{\frakd_i} = 1 + p_{i,1}x^{w_i} + \dots + p_{i, \ell_i} x^{\ell_iw_i}
\]
because each monomial is in $\frakm\setminus \frakm^2$. It is consistent as the wall-crossing automorphisms commute over $\Bbbk[\mndP]/\mathfrak m^2$.

\begin{exmp}
    For $\frakD_1$, we have
    \[
        \operatorname{Supp}(\frakD_1) = \bigcup_{i=1}^{n} e_i^\perp \quad \text{and} \quad \operatorname{Sing}(\frakD_1) = \bigcup_{i\neq j} e_i^\perp \cap e_j^\perp.
    \]
    The joints of $\mathfrak D_1$ are as follows. Let $H_k^\varepsilon$ for $\varepsilon\in \{+, -\}$ be the half space
    \[
        H_k^\varepsilon \coloneqq \{m\in M_\mathbb R\mid \mathrm{sgn}(\langle e_k, m\rangle) = \varepsilon \}.
    \]
    A joint is the closure of
    \[
        e_i^\perp \cap e_j^\perp \cap \bigcap_{k\neq i, j} H_k^{\varepsilon_k}
    \]
    for $i\neq j$ and $\varepsilon_k\in \{+, -\}$.
\end{exmp}

\begin{defn}\label{def: consistent at joint}
    A scattering diagram $\mathfrak D$ is called \emph{consistent} at a joint $\mathfrak j$ if the path-ordered product associated to a small loop around $\mathfrak j$ is identity. It is called \emph{consistent} if it is consistent at each of its joint. This is equivalent to the previous definition in \Cref{subsec: gen cluster sd higher rk} where the consistency means the path-ordered product of any regular path only depends on the endpoints.
\end{defn}

\begin{rem}
    By the definition of joints, the set of walls (whose support) containing a given joint is well-defined. There are two situations that a wall can contain a joint $\mathfrak j$: (1) $\mathfrak j$ is contained in $\partial S_\frakd$; (2)  $\mathrm{int}(\mathfrak j)$ is contained in $\mathrm{int}(S_\frakd)$. To check the consistency at $\mathfrak j$, the small loop around $\mathfrak j$ is chosen to cross a wall of the first kind once and a wall of the second kind twice.
\end{rem}

\begin{lem}[{\cite[Definition-Lemma C.2]{GHKK}}]
    Let $\frakD$ be a finite scattering diagram as in \Cref{subsec: gen cluster sd higher rk}. Let $\mathfrak j$ be a joint of $\frakD$. Then either every wall containing $\mathfrak j$ has direction tangent to $\mathfrak j$ or every wall containing $\mathfrak j$ has direction not tangent to $\mathfrak j$. In the first case we call the joint parallel; in the second case, perpendicular.
\end{lem}

Now we inductively construct for any $k\geq 1$, a finite scattering diagram $\frakD_k$ over $\Bbbk[\mndP]/\frakm^{k+1}$ to satisfy the following four assumptions.
\begin{enumerate}
    \item[(A1)] $\frakD_k$ is consistent over $\Bbbk[\mndP]/\mathfrak m^{k+1}$, that is, consistent at each joint.
    \item [(A2)] Any wall-function has positive coefficients, i.e., in $\mathbb N[\mndP]/\frakm^{k+1}$.
    \item [(A3)] Any non-initial wall $\frakd$ is outgoing and always has support of the form 
    \[
        S_\frakd = F_\frakd + \mathbb R_{\leq 0}\omega(-, n_\frakd)
    \]
    where $F_\frakd\subseteq \partial S_\frakd$ is a facet. In particular, $F_\frakd$ is the only facet not tangent to the direction of $\frakd$. We say that $\frakd$ or $S_\frakd$ is \emph{based} at $F_\frakd$.
    \item [(A4)] Any perpendicular joint $\mathfrak j$ is contained in the facet of at most one outgoing wall in each possible direction.
\end{enumerate}

The base case $\mathfrak D_1$ clearly satisfies all four assumptions.

Now we construct $\mathfrak D_{k+1}$ by essentially modifying existing walls in $\mathfrak D_k$ and adding new walls based at perpendicular joints of $\frakD_k$. Precisely, the construction comprises three steps.

\emph{Step 1: subdivision of outgoing walls.} 

By the assumption (A3), every outgoing wall $\frakd\in \frakD_k$ has support
\[
    S_\frakd = F_\frakd + \mathbb R_{\leq 0}\omega(-, n_\frakd)  
\]
where $F_\frakd$ is a facet not tangent to $\omega(-, n_\frakd)$. Being in the boundary of $S_\frakd$, the facet $F_\frakd$ inherits part of the cone complex realization of $\operatorname{Sing}(\frakD_k)$. Thus we can subdivide $F_\frakd$ into joints
\[
    F_\frakd = \mathfrak j_1\cup \cdots \cup \mathfrak j_s.    
\]
Note that these joints are all perpendicular. Accordingly $\frakd$ is subdivided into walls $\frakd_1, \dots, \frakd_s$ with smaller supports
\[
    S_{\frakd_1} = \mathfrak j_1 + \mathbb R_{\leq 0} \omega(-, n_\frakd), \dots, S_{\frakd s} = \mathfrak j_s + \mathbb R_{\leq 0} \omega(-, n_\frakd),
\]
and with wall-functions
\[
    f_{\frakd_1} = \dots = f_{\frakd_s} = f_\frakd.
\]
We replace $\frakD_k$ with an equivalent scattering diagram by implementing the above subdivision on every outgoing wall. Clearly it still satisfies the four assumptions. The set of joints remains unchanged. We also note that an initial wall never gets subdivided as it is incoming.

\emph{Step 2: rank-2 scattering at perpendicular joints.}

Consider at each perpendicular joint $\mathfrak j$ of $\mathfrak D_k$ the scattering diagram
\[
    \mathfrak D_{\mathfrak j} \coloneqq \{\frakd\in \mathfrak D_k\mid \mathfrak j\subseteq S_\frakd\},
\]
that is, the collection of walls whose support contains $\mathfrak j$. Let $\Lambda_\mathfrak j$ denote the tangent space of (any interior point of) $\mathfrak j$. It can thus be viewed as a $(\dim M_\mathbb R - 2)$-dimensional subspace of $M_\mathbb R$.
We next construct a rank-2 scattering diagram $\overline{\frakD}_\mathfrak j$ with respect to the lattice
\[
    M_{\mathfrak j} \coloneqq \omega(-, \Lambda_\mathfrak j^\perp) \cap M,
\]
living in the subspace $M_{\mathfrak j,\mathbb R} \coloneqq \omega(-, \Lambda_\mathfrak j^\perp) \subseteq M_\mathbb R$. These are of rank $2$ because of the injectivity assumption on the form $\omega(-, -)$.

The monoid for $\overline{\frakD}_\mathfrak j$ is 
\[
    \mndP_\mathfrak j \coloneqq \{p\in \mndP\mid r(p)\in M_\mathfrak j\} = M_\mathfrak j \oplus \mathbb P^\oplus = M_\mathfrak j \oplus \bigoplus_{i=1}^{|I_\mathrm{uf}|}\mathbb N^{\ell_i}.
\]
Let $\mathfrak m_{\mathfrak j}$ denote $\mndP_\mathfrak j \setminus \mndP_{\mathfrak j}^\times$ as well as the corresponding monomial ideal in $\Bbbk[\mndP_\mathfrak j]$. For each $\frakd\in \frakD_{\mathfrak j}$, define a wall $\overline{\frakd}$ with the support
\[
    S_{\overline{\frakd}} \coloneqq S_\frakd \cap M_{\mathfrak j, \mathbb R} 
\]
and the wall-function $f_{\overline{\frakd}} \coloneqq f_\frakd$.
Clearly by the particular form of the wall-function, the direction of $\overline{\frakd}$ is $m_{\overline{\frakd}} = m_\frakd \in \mathbb R_{\geq 0} \omega(-, n_\frakd)\cap M$ where $n_\frakd\in N^+$ is the normal vector of $\frakd$. One should carefully check that when $S_{\overline{\frakd}}$ is a ray, it must be of the form $\mathbb R_{\leq 0}m_{\overline{\frakd}}$ instead of $\mathbb R_{\geq 0}m_{\overline{\frakd}}$. By the assumption (A3), this is indeed the case.

Now we define the scattering diagram
\[
    \overline{\frakD}_\mathfrak j \coloneqq \{\overline{\frakd}\mid \frakd\in \frakD_\mathfrak j\}.
\]
The walls in this rank-2 scattering diagram are either lines or rays. By the assumption (A4), rays are in distinct directions (but may overlap with lines).

We then take the consistent completion $\Scat(\overline{\frakD}_\mathfrak j)$ over $\Bbbk[\mndP_\mathfrak j]/\mathfrak m_\mathfrak j^{k+2}$. This is equivalent to the consistent completion of only lines in $\overline{\frakD}_\mathfrak j$, as studied in \Cref{subsec: full positivity rk 2}. Notice that $\overline{\frakD}_\mathfrak j$ is already consistent over $\Bbbk[P_\mathfrak j]/{\mathfrak m}_{\mathfrak j}^{k+1}$ because by (A1) $\frakD_{\mathfrak j}$ is consistent at $\mathfrak j$. Therefore, by the positivity \Cref{thm: full positivity rk 2} we can obtain a representative of $\Scat(\overline{\frakD}_\mathfrak j)$ by 
\begin{enumerate}
    \item adding new rays $\frakd(n_0)$ with support $S_{\frakd(n_0)} = \mathbb R_{\leq 0}\omega(-, n_0)$ and wall-function $f_{\frakd(n_0)}$ with only non-trivial monomial terms in $\mathfrak m^{k+1}\setminus \mathfrak m^{k+2}$ with positive integer coefficients for distinct $n_0\in \Lambda_{\mathfrak j}^\perp \cap N^+$;
    \item for any existing ray $\overline{\frakd}$ in $\overline{\frakD}_\mathfrak j$, adding positive terms in $\mathfrak m^{k+1}\setminus \mathfrak m^{k+2}$ to $f_{\overline\frakd}$ (and calling it $f'_{\overline{\frakd}}$).
\end{enumerate}
The crucial point here is that the positivity is guaranteed by \Cref{thm: full positivity rk 2}.

\emph{Step 3: constructing $\mathfrak D_{k+1}$.}

For each new ray added in (1), define a wall in $M_\mathbb R$ with the support
\[
    \mathfrak j + \mathbb R_{\leq 0} \omega(-, n_0)
\]
and the same wall-function $f_{\frakd(n_0)}$. Since $\mathfrak j$ is perpendicular, such a wall is outgoing as the support must not contain $\omega(-,n_0)$. Denote this (finite) collection of walls by $\frakD(\mathfrak j)$.

Now we define
\[
    \frakD_{k+1} \coloneqq (\frakD_{\mathrm{in}, \mathbf s})_{k+1} \cup \bigcup_{\mathfrak j} \left(\frakD(\mathfrak j) \cup \{(S_\frakd, f'_{\overline\frakd})\mid \mathfrak j \subseteq \partial S_\frakd, \frakd\in \frakD_\mathfrak j \} \right),
\]
where $\bigcup_{\mathfrak j}$ is over all perpendicular joints in $\frakD_k$. We note that each existing outgoing $\frakd\in \mathfrak D_k$ appears exactly once in the union because it only contains a unique perpendicular joint after the subdivision in \emph{Step 1}. We have finished the construction of $\mathfrak D_{k+1}$.

We next elaborate on why $\mathfrak D_{k+1}$ satisfies assumptions (A2) to (A4). We then discuss the consistency assumption (A1), which will be proven in \Cref{lem: consistency overall}.

For (A2), as pointed out earlier, the positivity is preserved by \Cref{thm: full positivity rk 2}. Any wall in $\frakD_{k+1}$ (except the initial ones) is either in some $\frakD(\mathfrak j)$ or of the form $(S_\frakd, f_{\overline{\frakd}}')$ arising from some $\frakd\in \frakD_\mathfrak j$ such that $\mathfrak j\subseteq \partial S_\frakd$. Clearly both kinds are outgoing. Any such $\frakd$ has support of the form
\[
    \mathfrak j + \mathbb R_{\leq 0}\omega(-, n_\frakd)
\]
where $\mathfrak j$ is regarded as $F_\frakd$. So (A3) is satisfied by $\frakD_{k+1}$.

The statement of (A4) involves joints in $\mathfrak D_{k+1}$ upon a cone complex realization of $\operatorname{Sing}(\mathfrak D_{k+1})$. Since $\operatorname{Sing}(\mathfrak D_{k+1})$ contains $\operatorname{Sing}(\mathfrak D_k)$, we can require that the cone complex of $\operatorname{Sing}(\mathfrak D_{k+1})$, when restricted on $\operatorname{Sing}(\mathfrak D_k)$, is a refinement of the cone complex of $\operatorname{Sing}(\mathfrak D_k)$. We then have the following two cases for any perpendicular joint $\mathfrak j'$ in $\mathfrak D_{k+1}$.

\emph{Case 1.} Suppose that $\mathfrak j'$ is contained in a perpendicular joint $\mathfrak j$ of $\mathfrak D_k$. Notice that the outgoing walls in $\mathfrak D(\mathfrak j)$ and the outgoing walls of the form $(S_\frakd, f'_{\overline{\frakd}})$ are all in distinct directions. Thus we only need to show if there is a wall $\frakd'\in \mathfrak D_{k+1}$ based at another perpendicular joint $\mathfrak j_1$ (of $\frakD_k$) such that $S_{\frakd'}$ contains $\mathfrak j'$, the only case can happen is
\[
    \mathrm{int}(\mathfrak j') \subseteq \mathrm{int}(S_{\frakd'}).
\]
Otherwise assume that $\mathfrak j'$ is contained in $\partial S_{\frakd'}$. Notice that we have
\[
    S_{\frakd'} = \mathfrak j_1 + \mathbb R_{\leq 0}\omega(-, n_{\frakd'}),
\]
where $\mathfrak j_1$ is the only facet not tangent to the direction $\omega(-, n_{\frakd'})$. Thus $\mathfrak j'$ is contained in $\mathfrak j_1$. Then $\mathfrak j_1$ and $\mathfrak j$ (both as joints of $\frakD_k$) have maximal dimensional intersection, a contradiction.

\emph{Case 2.} Suppose that $\mathfrak j'$ is not contained in any perpendicular joint of $\frakD_k$. Assume that $\mathfrak j'$ is contained in a facet of a wall $\frakd\in \mathfrak D_{k+1}$. Then it must be contained in the only possible facet that is not tangent to the direction. However, such a facet is always a perpendicular joint of $\mathfrak D_k$, a contradiction. This argument also implies that any such perpendicular joint $\mathfrak j'$ of $\mathfrak D_{k+1}$ is not in the boundary of any wall.

\begin{lem}\label{lem: consistency perp joint}
    $\mathfrak D_{k+1}$ is consistent at each perpendicular joint $\mathfrak j'$ over $\Bbbk[\mndP]/\mathfrak m^{k+2}$.
\end{lem}

\begin{proof}
    The analysis proceeds by considering the same two cases on whether $\mathfrak j'$ is contained in a perpendicular joint $\mathfrak j$ of $\frakD_k$, as when we show (A4). 
    
    When $\mathfrak j'$ is contained in some perpendicular $\mathfrak j$ of $\mathfrak D_k$, we check all walls $\frakd$ containing $\mathfrak j'$, namely the scattering diagram $\frakD_{\mathfrak j'}$. The idea is then to derive the consistency of $\frakD_{\mathfrak j'}$ from $\Scat(\overline{\frakD}_\mathfrak j)$. There are two kinds of $\frakd$:
    \[
        \mathfrak j'\subseteq \partial S_\frakd \quad \text{or} \quad \mathrm{int}(\mathfrak j')\subseteq \mathrm{int}(S_\frakd).
    \]
    The walls of the first kind are exactly those in 
    \[
        \frakD(\mathfrak j) \cup \{(S_\frakd, f'_{\overline\frakd})\mid \mathfrak j \subseteq \partial S_\frakd, \frakd\in \frakD_\mathfrak j \}.
    \]
    For the second kind, we can get rid of those whose wall-functions are purely in order $k+1$ since their wall-crossing automorphisms commute with others over $\Bbbk[\mndP]/\mathfrak m^{k+2}$. The remaining walls of the second kind are from walls in $\mathfrak D_k$. The possible difference on the wall-function is only in order $k+1$, as mentioned in the construction, namely $f_{\overline{\frakd}}'$ versus $f_{\overline{\frakd}}$. Notice that
    \[
        f_{\overline{\frakd}}' = f_{\overline{\frakd}} \cdot (1 + f_{\overline{\frakd}}'-f_{\overline{\frakd}})\mod \mathfrak m^{k+2}.
    \]
    The wall-crossing automorphism associated to the function $1 + f_{\overline{\frakd}}'-f_{\overline{\frakd}}$ commute with others over $\Bbbk[\mndP]/\mathfrak m^{k+2}$. Therefore, we can replace $f'_{\overline{\frakd}}$ with  $f_{\overline{\frakd}} = f_{\frakd}$ and such a wall rolls back to the version in $\mathfrak D_k$. Now we see that the consistency of $\mathfrak D_{\mathfrak j'}$ is equivalent to that of $\Scat(\overline{\frakD}_\mathfrak j)$.

    The case where $\mathfrak j'$ is not contained in any perpendicular joint of $\mathfrak D_k$ is simpler. We have understood that $\mathfrak j'$ is not in the boundary of any wall. Such a joint is created only because of the intersection with new walls, i.e., the ones in $\frakD(\mathfrak j)$ for some $\mathfrak j$. Those walls always contain $\mathfrak j'$ of the form
    \[
        \mathrm{int}(\mathfrak j') \subseteq \mathrm{int}(S_\frakd).
    \]
    Thus as argued earlier, we can get rid of those walls and are left with only existing walls in $\frakD_k$ containing $\mathfrak j'$. The only possibility is that the walls
    \[
        \{\frakd\in \mathfrak D_k\mid \mathfrak j'\subseteq S_\frakd\}
    \]
    are all in the same direction; otherwise $\mathfrak j'$ would have been in a perpendicular joint of $\mathfrak D_k$. Since $\mathfrak j'$ is not in the boundary of any wall, these walls all pass through $\mathfrak j'$. The consistency around $\mathfrak j'$ is then clear.
\end{proof}

We apply \Cref{lem: consistency perp joint} to show

\begin{lem}\label{lem: consistency overall}
    $\mathfrak D_{k+1}$ is consistent over $\Bbbk[\mndP]/\mathfrak m^{k+2}$.  
\end{lem}

\begin{proof}
    The strategy is to compare $\mathfrak D_{k+1}$ with $(\frakD_\mathbf s)_{k+1}$ and to show they are equivalent. The scattering diagram $(\frakD_\mathbf s)_{k+1}$ is over $\Bbbk[\mndP]/\mathfrak m^{k+2}$, consistent and contains finitely many outgoing walls besides the initial walls. We note that the existence of $(\frakD_\mathbf s)_{k+1}$ does not depend on the constructive process described here. See \Cref{thm: gen cluster sd} and also \cite[Theorem 1.21]{GHKK}, which are both incarnation of \cite[Theorem 2.1.6]{KSwcs}.

    Firstly, these two scattering diagrams are equivalent over $\mathfrak m^{k+1}$, for they are consistent and share the same collection of incoming walls. It is convenient to consider their difference $\mathfrak D'$, as a scattering diagram. In fact, one can subdivide walls in $\mathfrak D_{k+1}$ and $(\frakD_\mathbf s)_{k+1}$ and add walls with trivial wall-functions so that they have walls labeled by the same index set $S$
     \[
        \mathfrak D_{k+1} = \{\frakd_i\mid i\in S\} \quad \text{and} \quad (\frakD_\mathbf s)_{k+1} = \{\mathfrak c_i\mid i\in S\}
     \]
     such that $S_{\frakd_i} = S_{\mathfrak c_i}$ for every $i\in S$. Their difference is $\frakD' = \{\frakd'_i\mid i\in S\}$ where
     \[
        S_{\frakd'_i} \coloneqq S_{\frakd_i} \quad \text{and} \quad f_{\frakd'_i} \coloneqq f_{\frakd_i}\cdot f_{\mathfrak c_i}^{-1}.
     \]
     The wall-function $f_{\frakd'_i}$ must be of the form $1+\xi$ where $\xi\in \mathfrak m^{k+1}\setminus \mathfrak m^{k+2}$ since it is trivial modulo $\mathfrak m^{k+1}$.

     Take another representative of $\mathfrak D'$ such that there are only non-trivial walls and they intersect only at codimension two. Notice that in this representative, walls are all outgoing since both $\mathfrak D_{k+1}$ and $(\frakD_\mathbf s)_{k+1}$ have the exact same collection of incoming walls.

     Assume that there are non-trivial walls in a hyperplane $n_0^\perp$. Let $\mathfrak j$ be a joint (for $\mathfrak D'$) in $n_0^\perp$. If it is not tangent to $\omega(-, n_0)$, we claim that there are walls in $n_0^\perp$ on both sides of $\mathfrak j$ with the same wall-function. In fact, by \Cref{lem: consistency perp joint}, $\frakD_{k+1}$ is consistent at every perpendicular joint modulo $\mathfrak m^{k+2}$. Hence $\frakD'$ is also consistent at $\mathfrak j$ (as $(\frakD_\mathbf s)_{k+1}$ is consistent overall). Since the wall-crossing automorphisms associated to those $f_{\frakd'_i}$ all commute, the claim follows. Now take the union $U$ (which is a polyhedral cone but not necessarily convex) of all walls in $n_0^\perp$. Each facet of $U$ is tangent to $\omega(-, n_0)$. This forces the whole line $\mathbb R\omega(-, n_0)$ to lie in the boundary of $U$, giving the presence of a non-trivial incoming wall in $\mathfrak D'$, a contradiction. We can then conclude that $\frakD'$ is trivial and therefore $\frakD_{k+1}$ and $(\frakD_\mathbf s)_{k+1}$ are equivalent.
\end{proof}

Now we have obtained a sequence $(\mathfrak D_k)_{k\geq 1}$ of consistent scattering diagram over $\Bbbk[\mndP]/\mathfrak m^{k+1}$ with finitely many walls without any incoming walls besides the initials. For the final step, we ensure that the limit when $k\rightarrow \infty$ exists as a scattering diagram over $\widehat{\Bbbk[\mndP]}$. Notice that in this order by order construction, we either add a new outgoing wall or add higher order terms to the wall-function of an existing outgoing wall. The limit $\mathfrak D_\infty$ is taken to consist of $\mathfrak D_{\mathrm{in}, \mathbf s}$ and any outgoing wall that first appears in some $\frakD_k$ with the wall-function being the limit of the wall-functions on the (support of the) same wall in $\frakD_i$, $i\geq k$. Evidently for each $k\geq 1$ we have
\[
    \left(\mathfrak D_\infty \right)_k = \mathfrak D_k.
\]
It follows that $\mathfrak D_\infty$ is consistent and has every wall-function in $\widehat{\mathbb N[\mndP]}$. By the uniqueness in \Cref{thm: gen cluster sd}, $\mathfrak D_\infty$ is equivalent to $\frakD_\mathbf s$. This is the end of the proof of \Cref{thm: positive wall-function}.

\appendix

\section{Proof of the generalized greedy element formula}\label{app: generalized greedy element formula}

In this appendix, we prove the expansion formula for generalized greedy elements in terms of compatible gradings given by \Cref{thm: generalized greedy compatible gradings}.  The majority of the proof closely follows that of \cite[Theorem 2.15]{Rupgengreed}, and we discuss only those portions that require modification here.  Namely, we generalize some of the necessary lemmas to the case where $P_1$ and $P_2$ are not necessarily palindromic.

\subsection{Cluster variables up to a scalar}

Recall the notations in \Cref{subsec: gen cluster sd rk 2} and the definition of cluster variables $x_i$ from \Cref{defn: cluster variables}. We define elements $\{\widetilde x_k\mid k\in \Z\}\subseteq \mathcal F = \operatorname{Frac}(\Bbbk[\mndP])$ and will show $x_k/\widetilde x_k \in \mathbb P$.
These \emph{scaled cluster variables} will be convenient for our proofs in \Cref{subsec: strong Laurent,subsec: proof of cg formula}.

Let $\widetilde{x}_1=x_1$ and $\widetilde{x}_2=x_2$. 
We inductively define $\widetilde{x}_k\in \calF$ for $k\in\mathbb{Z}$ by the rule:
$$
\widetilde{x}_{k+1}\widetilde{x}_{k-1}=\left\{\begin{array}{ll}
P_1(\widetilde{x}_k)\quad &\text{if }k\equiv 1 \mod 4,\\
P_2(\widetilde{x}_k)\quad &\text{if }k\equiv 2 \mod 4,\\
\overline{P}_1(\widetilde{x}_k)\quad &\text{if }k\equiv 3 \mod 4,\\
\overline{P}_2(\widetilde{x}_k)\quad &\text{if }k\equiv 0 \mod 4.\\
\end{array}\right.
$$

\begin{lem}\label{lem: scaled cluster var}
For any $k\in\mathbb{Z}$, the quotient $x_k/\widetilde{x}_k$ is a Laurent monomial in $p_{1,\ell_1}$ and $p_{2,\ell_2}$.  More precisely, we have
\[
    x_k/\widetilde x_k = \begin{cases}
        W(\dvect_k-e_2) \quad \text{if $k\equiv 0\mod 4$}, \\
        W(\dvect_k+e_1) \quad \text{if $k\equiv 1\mod 4$}, \\
        W(\dvect_k+e_2) \quad \text{if $k\equiv 2\mod 4$}, \\
        W(\dvect_k-e_1) \quad \text{if $k\equiv 3\mod 4$}.
    \end{cases}
\]
\end{lem}
\begin{proof}
It suffices to show that $x_k/W(\dvect_k\pm e_i)$ satisfy the exchange relations of $\widetilde x_k$. We check for $k\equiv 1\mod 4$ the following relation and the remaining cases are similar. Notice that by \Cref{lem: exchange relation rk 2}, we have
\begin{align*}
    x_{k-1}x_{k+1} & = \bar \zeta_n(x_k) = \frac{W(\dvect_k + e_1)^{\ell_1}}{p_{1,\ell_1}} P_1(x_k/W(\dvect_k+e_1)), \\
    \widetilde x_{k-1}\widetilde x_{k+1} & = \frac{W(\dvect_k + e_1)^{\ell_1}}{p_{1,\ell_1}W(\dvect_{k-1}-e_2)W(\dvect_{k+1}+e_2)} P_1(\widetilde x_k).
\end{align*}
Now the desired relation follows from $W(\dvect_k + e_1)^{\ell_1} = p_{1,\ell_1}W(\dvect_{k-1}-e_2)W(\dvect_{k+1}+e_2)$, which is result of
\[
    \dvect_{k-1} + \dvect_{k+1} = \ell_1 \dvect_k\,. \qedhere
\]
\end{proof}

\subsection{The strong Laurent phenomenon}\label{subsec: strong Laurent}
We begin by establishing the strong Laurent phenomenon for the rank-$2$ generalized cluster algebras $\mathcal A(P_1, P_2)$, which is the $\rngB$-subalgebra of $\mathcal F$ generated by all scaled cluster variables $\widetilde x_k$. By \Cref{lem: scaled cluster var}, this notion coincides with the $\mathcal A(P_1, P_2)$ defined in \Cref{subsec: theta universal Laurent}.

\begin{thm}[{following \cite[Theorem 2.2]{Rupgengreed}}]\label{thm: cluster algebra four generators}
  For any $k \in \Z$, we have 
  $$\calA(P_1,P_2) = \rngB[x_{k-1}, x_{k}, x_{k+1}, x_{k+2}]\,.$$
\end{thm}
\begin{proof}
  By \Cref{lem: scaled cluster var}, it is enough to work with cluster pre-variables instead of cluster variables.  We show via induction on $m$ that $\cpv_m \in \rngB[x_{k-1}, x_{k}, x_{k+1}, x_{k+2}]$, hence $\calA(P_1,P_2)\subseteq \rngB[x_{k-1}, x_{k}, x_{k+1}, x_{k+2}]$. 
  Consider the element $\cpv_{m+1}^{\ell_1}\cpv_{m+4}$, where without loss of generality we assume $m \equiv 2 \mod 4$.
  We have
  \begin{align*}
    \cpv_m - \frac{\cpv_{m+1}^{\ell_1}\cpv_{m+4}}{p_{1,\ell_1}} &= p_{1,\ell_1}\cpv_m - \frac{\cpv_{m+1}^{\ell_1}P_1(\cpv_{m+3})}{p_{1,\ell_1}\cpv_{m+2}} \\
    &= \cpv_m - \frac{\overline{P_1}(\cpv_{m+1})}{\cpv_{m+2}} + \frac{p_{1,\ell_1} \overline{P_1}(\cpv_{m+1}) - \cpv_{m+1}^{\ell_1}P_1(\cpv_{m+3})}{p_{1,\ell_1}\cpv_{m+2}} \\
    &= \frac{p_{1,\ell_1} \overline{P_1}(\cpv_{m+1}) - \cpv_{m+1}^{\ell_1}P_1(\cpv_{m+3})}{p_{1,\ell_1}\cpv_{m+2}} \\
    &= \frac{\sum_{t = 0}^{\ell_1} p_{1,t} (\cpv_{m+1}^t\cpv_{m+3}^t - 1)\cpv_{m+1}^{\ell_1 - t}}{p_{1,\ell_1}\cpv_{m+2}}\,.
  \end{align*}
  
   Using that $\overline{P_2}$ has constant term $1$, we have $\cpv_{m+1}^t\cpv_{m+3}^t - 1 = (\overline{P_2}(\cpv_{m+2}))^t - 1 \in \cpv_{m+2} \rngB[x_{m+2}]$ for every $t \geq 0$. Thus, 
  $$\frac{\sum_{t = 0}^{\ell_1} p_{1,t} (\cpv_{m+1}^t\cpv_{m+3}^t - 1)\cpv_{m+1}^{\ell_1 - t}}{\cpv_{m+2}} \in \rngB[x_{m+1}, x_{m+2}]\,.$$

We then can see that $\cpv_m \in \rngB[x_{m+1}, x_{m+2}, x_{m+3}, x_{m+4}]$, and using a similar argument we obtain $\cpv_m \in \rngB[x_{m-4}, x_{m-3}, x_{m-2}, x_{m-1}]$. We can then conclude that 
\[
    \cpv_m \in \rngB[x_{k-1}, x_{k}, x_{k+1}, x_{k+2}] \subseteq \calT_k
\]
for every $k \in \Z$.  Since we have by definition that $\rngB[x_{k-1}, x_{k}, x_{k+1}, x_{k+2}] \subseteq \calA(P_1,P_2)$, the desired equality follows.
\end{proof}

We now show that the cluster algebra $\calA(P_1,P_2)$ equals the corresponding upper cluster algebra $\mathcal U$, as defined in \Cref{subsec: theta universal Laurent}, and moreover that it is the intersection of the Laurent polynomial ring of three consecutive clusters.

\begin{thm}[{Strong Laurent Phenomenon, following \cite[Theorem 2.5]{Rupgengreed}}]\label{thm: strong laurent}
  For any $m \in \Z$ we have
  $$\calA(P_1,P_2) = \bigcap_{k \in \Z} \calT_k = \bigcap_{k = m-1}^{m+1} \calT_k\,.$$
\end{thm}
\begin{proof}
   By \Cref{thm: cluster algebra four generators}, it is enough to show that
  $$\rngB[x_{k-1}, x_{k}, x_{k+1}, x_{k+2}] = \calT_{m-1} \cap \calT_m \cap \calT_{m+1}\,.$$
  The containment of $\rngB[x_{k-1}, x_{k}, x_{k+1}, x_{k+2}]$ into $\calT_{m-1} \cap \calT_m \cap \calT_{m+1}$ follows from \Cref{prop: theta universal Laurent}. We show the other containment by taking the decomposition 
  $$\calT_m = \rngB[x_{m}^{-1}, x_{m+1}^{-1}] + \rngB[x_{m}, x_{m+1}^{\pm1}] + \rngB[\cpv_{m}^{\pm1}, \cpv_{m+1}]$$
  and proving containment of the intersection of each summand with $\calT_{m-1} \cap \calT_{m+1}$.

  We begin by establishing the following:
  \begin{equation}\label{eq: ring equality} \calT_{m-1} \cap \rngB[x_{m}^{-1}, x_{m+1}^{-1}] \cap \calT_{m+1} = \begin{cases} \rngB & \text{ if } \ell_1\ell_2 \geq 2\\ \rngB[x_{m-1}x_{m+2}] & \text{ if } \ell_1 = \ell_2 = 1\\ \rngB[x_{m-1},x_{m-1}^{\ell_2}x_{m+2}] & \text{ if } \ell_1 = 0\\ \rngB[x_{m-1}x_{m+2}^{\ell_1}, x_{m+2}] &\text{ if $\ell_2 = 0$}\end{cases}\end{equation}

  Without loss of generality, we assume $m \equiv 1 \mod 4$.  In order to simplify the exchange relations, we often work with the scaled cluster pre-variables rather than cluster variables due to \Cref{lem: scaled cluster var}.

   We first show that the ring on the right of \eqref{eq: ring equality} is contained in the intersection on the left of $\eqref{eq: ring equality}$. By \Cref{prop: theta universal Laurent}, each ring on the right is contained in $\calT_{m-1} \cap \calT_{m+1}$, so it remains to show containment in $\rngB[x_m^{-1},x_{m+1}^{-1}]$.  This containment is clear for $\ell_1\ell_2 \geq 2$.  When $\ell_1 = \ell_2 = 1$, we have
  $$\cpv_{m-1}\cpv_{m+2} = P_1(\cpv_m)\cpv_{m+1}^{-1}P_2(\cpv_{m+1})\cpv_m^{-1} \in \rngB[x_m^{-1},x_{m+1}^{-1}]\,.$$
  When $\ell_1 = 0$ or $\ell_2 = 0$, where without loss of generality we assume the former, we have $\cpv_{m-1} = \cpv_{m+1}^{-1}$ and
  $$\cpv_{m-1}^{\ell_2}\cpv_{m+2} = \cpv_{m+1}^{-\ell_2}P_2(\cpv_{m+1})\cpv_m^{-1} = p_{2,\ell_2}\overline{P_2}(\cpv_{m+1}^{-1})\cpv_{m}^{-1} \in \rngB[x_m^{-1},x_{m+1}^{-1}]\,.$$
    In each of these cases, we can now readily see that the generators of the ring on the right of \eqref{eq: ring equality} are contained in $\rngB[x_m^{-1},x_{m+1}^{-1}]$.

  We now handle the other containment of \eqref{eq: ring equality} case-by-case. Consider $y \in \rngB[x_m^{-1},x_{m+1}^{-1}]$, and express $y$ as 
  $$y = \sum_{s,t \geq 0} c_{s,t}\cpv_m^{-s}\cpv_{m+1}^{-t}\,,$$
  with finitely many nonzero coefficients $c_{s,t} \in \rngB$. We then have
  $$y = \sum_{t \geq 0} \frac{\sum_{s \geq 0} c_{s,t}\cpv_m^{-s}}{P_1^t(\cpv_m)} \cpv_{m-1}^t = \sum_{t \geq 0} \frac{\sum_{s \geq 0} c_{s,t}\cpv_m^{-s}}{\overline{P_1}^t(\cpv_m^{-1})} p_{1,\ell_1}^t \cpv_m^{\ell_1 t} \cpv_{m-1}^t\,.$$
  If $y \in \calT_{m-1}$, then for all $t \geq 0$ we have $\frac{\sum_{s \geq 0} c_{s,t}\cpv_m^{-s}}{\overline{P_1}^t(\cpv_m^{-1})} p_{1,\ell_1}^t \in \rngB[x_m^{-1}]$, where we denote this polynomial by $C_t(x_m^{-1})$. So we can now express $y$ as 
  \begin{align}\label{eq: first y expansion} y &= \sum_{t \geq 0} C_t(\cpv_m^{-1})\cpv_m^{-\ell_1 t}\cpv_{m-1}^t\\ \nonumber &= \sum_{t \geq 0} C_t(\cpv_m^{-1})P_1(\cpv_m)^t\cpv_m^{-\ell_1 t}\cpv_{m+1}^{-t}\\ \nonumber &= \sum_{t \geq 0} C_t(\cpv_m^{-1})p_{1,\ell_1}^t\overline{P_1}^t(\cpv_m^{-1})\cpv_{m+1}^{-t}\,.\end{align}
  Similarly, if $y \in \calT_{m+1}$, we have
  \begin{equation} \label{eq: second y expansion} y = \sum_{t \geq 0} C'_s(\cpv_{m+1}^{-1})\cpv_{m+1}^{-\ell_2 s} \cpv_{m+2}^{s} = \sum_{t \geq 0} C'_s(\cpv_{m+1}^{-1})p_{2,\ell_2}^s\overline{P_2}^s(\cpv_{m+1}^{-1}) \cpv_{m}^{s} \end{equation}
  for some polynomials $C_s'(\cpv_{m+1}^{-1}) \in \rngB[x_{m+1}^{-1}]$.

  From this point, the proof follows essentially verbatim from Rupel's work \cite[Theorem 2.5]{Rupgengreed}, but we include the remainder of the argument here for completeness.

  If $\atwo = 0$, then $\cpv_m^{-1} = \cpv_{m+2}$ and \eqref{eq: first y expansion} becomes
  $$y = \sum_{t \geq 0} C_t(\cpv_{m+2})\cpv_{m+2}^{\ell_1 t} \cpv_{m-1}^t \in \rngB[x_{m-1}x_{m+2}^{\ell_1}, x_{m+2}]\,.$$

  Similarly, if $\aone = 0$ then 
  $$y = \sum_{s \geq 0} C_s'(\cpv_{m-1})\cpv_{m-1}^{\ell_2 s} \cpv_{m+2}^s \in \rngB[x_{m-1},x_{m-1}^{\ell_2} x_{m+2}]\,.$$

  We now consider when $\ell_1\ell_2 \neq 0$. Let $\sigma(t)$ be maximal such that $c_{\sigma(t),t} \neq 0$, and using the last equality in \eqref{eq: first y expansion} we must have $\sigma(t) \geq \ell_1 t$. Similarly, let $\tau(s)$ maximal such that $c_{s,\tau(s)} \neq 0$, and we have $\tau(s) \geq \ell_2 s$. 

  Suppose $t$ is maximal such that $c_{s,t} \neq 0$ for some choice of $s$. Then $c_{\sigma(t),\tau(\sigma(t)))} \neq 0$, and $\tau(\sigma(t)) \geq \ell_2 \sigma(t) \geq \ell_1 \ell_2 t$. For $\ell_1 \ell_2 \geq 2$ and $t > 0$ this implies $\tau(\sigma(t)) > t$, which contradicts the maximality of $t$. Thus we must have $t = 0$ and, similarly, $s = 0$ when $\ell_1 \ell_2 \geq 2$. Hence $y \in \rngB$.

  Now suppose $\ell_1 = \ell_2 = 1$. Then we have $\tau(\sigma(t)) \geq \sigma(t) \geq t$, and by the maximality of $t$ these are equalities. Using induction on the maximal value $t$ such that $c_{t,t}\neq 0$, we can see that $y \in \rngB[x_{m-1}x_{m+2}]$. The base case $t = 0$ is clear. When $t > 0$, the element $y - c_{t,t}\cpv_{m-1}^t\cpv_{m+2}^t$ is contained in $\rngB[x_{m-1}x_{m+2}]$, hence $y$ is also contained in $\rngB[x_{m-1}x_{m+2}]$. Thus \eqref{eq: ring equality} is proved.

  It then remains to show the equality
  $$\calT_{m-1} \cap \rngB[x_m,x_{m+1}^{\pm 1}] = \rngB[x_{m-1},x_m,x_{m+1}] \subset \calT_{m+1}\,.$$
  The inclusion of $\rngB[x_{m-1},x_m,x_{m+1}]$ into $\calT_{m-1} \cap \rngB[x_m,x_{m+1}^{\pm 1}]$ is clear. For the other inclusion, note that any $y \in \rngB[x_m,x_{m+1}^{\pm 1}]$ can be written in the form
  $$y = \sum_{t = -N}^N c_tx_{m+2}^t$$
  for some positive integer $N$ and coefficients $c_t \in \rngB[x_m]$. We then have
  $$y = \sum_{t = 0}^N c_t P_1^t(\cpv_m) \cpv_{m-1}^{-t} + \sum_{t = 1}^N \frac{c_{-t}}{P_1^t(\cpv_m)} \cpv_{m-1}^t\,.$$
  If $y \in \calT_{m-1}$, then we must have $\frac{c_{-t}}{P_1^t(\cpv_m)} \in \rngB[x_m^{\pm 1}]$ for all $1 \leq t \leq N$. Since $P_1(\cpv_m)$ has constant term $1$, this implies $\frac{c_{-t}}{P_1^t(\cpv_m)} \in \rngB[x_m]$. Thus we have
  $$y = \sum_{t = 0}^N c_t \cpv_{m+1}^t + \sum_{t = 1}^N \frac{c_{-t}}{P_1^t(\cpv_m)} \cpv_{m-1}^t \in \rngB[x_{m-1},x_m,x_{m+1}]\,.$$
  We can similarly obtain $\rngB[x_m^{\pm 1},x_{m+1}] \cap \calT_{m+1} = \rngB[x_m,x_{m+1},x_{m+2}] \subset \calT_{m-1}$. 
  Combining these inclusions yields
  \begin{align*}
    \calT_{m-1} \cap \calT_m \cap \calT_{m+1} &= \calT_{m-1} \cap \left(\rngB[x_{m}^{-1}, x_{m+1}^{-1}] + \rngB[x_{m}, x_{m+1}^{\pm1}] + \rngB[x_{m}^{\pm1}, x_{m+1}] \right) \cap \calT_{m+1}\\
    &= \calT_{m-1} \cap \rngB[x_{m}^{-1}, x_{m+1}^{-1}] \cap \calT_{m+1} + \rngB[x_{m-1}, x_{m}, x_{m+1}]\\
    &\qquad+ \rngB[x_{m}, x_{m+1}, x_{m+2}]\\
    &\subset \rngB[x_{k-1}, x_{k}, x_{k+1}, x_{k+2}]\,,
  \end{align*}
  as desired.
\end{proof}

\subsection{Proof of the compatible grading formula}\label{subsec: proof of cg formula}

This section closely follows the approach of Rupel in \cite{Rupgengreed}, with necessary adjustments made to extend the argument to the case when $P_1$ and $P_2$ are not palindromic. In the compatible grading formula given by \Cref{thm: generalized greedy compatible gradings}, the horizontal edges of the Dyck path are weighted by the coefficients of $P_2$ and the vertical edges are weighted by the coefficients of $P_1$.  We introduce auxiliary variants of this weighting, replacing $P_i$ with $\overline{P_i}$.

\begin{defn}
  For a grading $\omega:E(\aone ,\atwo )\rightarrow\mathbb{Z}_{\ge0}$, we define
  \begin{align*}
    \text{wt}'(\omega) &=\prod_{i=1}^{\aone} p_{2,\ell_2 - \omega(u_i)} \prod_{j=1}^{\atwo} p_{1, \omega(v_j)}\in \Bbbk[\mathbb Z^{\ell_1}\oplus \mathbb Z^{\ell_2}]\,, \text{ and}\\
    \text{wt}''(\omega) &=\prod_{i=1}^{\aone} p_{2,\omega(u_i)} \prod_{j=1}^{\atwo} p_{1,\ell_1 -\omega(v_j)} \in \Bbbk[\mathbb Z^{\ell_1}\oplus \mathbb Z^{\ell_2}]\,.
  \end{align*}
\end{defn}

We need to introduce the following auxiliary variants of the generalized greedy elements.

\begin{defn}
We define $x'[\aone,\atwo ] \coloneqq x_1^{-\aone }x_2^{-\atwo }\sum_{(S_1,S_2)} \wt'(\omega) x_1^{\omega(\calP_\north)}x_2^{\omega(\calP_\east)}$, and similarly let $x''[\aone,\atwo ] \coloneqq x_1^{-\aone }x_2^{-\atwo }\sum_{(S_1,S_2)} \wt''(\omega) x_1^{\omega(\calP_\north)}x_2^{\omega(\calP_\east)}$.
\end{defn}

We now consider two automorphisms of the field of rational functions $\calF$.  These maps are defined by their behavior on a generating set and given by $\sigma_1: (x_1,x_2) \mapsto (x_1,x_0)$ and $\sigma_2: (x_1,x_2) \mapsto (x_3,x_2)$.  In the case that $P_1$ and $P_2$ are palindromic, these are the usual reflection automorphisms $\sigma_j$, where $\sigma_j(x_m) = x_{2j-m}$.  We show that applying these automorphisms to a varient of some generalized greedy element results in a ``mutated'' generalized greedy element.

\begin{prop}[{following \cite[Proposition 2.17]{Rupgengreed}}]\label{prop: reflection greedy}  For any $(\aone, \atwo) \in \Z^2$ the elements $x[\aone,\atwo]$ defined by \eqref{eq: greedy expansion} satisfy  
$$\sigma_1(x'[\aone ,\atwo ]) = x[\aone, \ell_1[\aone ]_+ - \atwo ] \qquad \text{ and } \qquad\sigma_2(x''[\aone ,\atwo ]) = x[\ell_2[\atwo ]_+ - \aone, \atwo]\,.$$
\end{prop}
\begin{proof}
By symmetry, it is enough to establish the second identity. We do so by handling several cases.

 First, suppose $\aone ,\atwo \leq 0$. Then $\calP([\aone ]_+,[\atwo ]_+) = \calP(0,0)$ is a single point, so we have 
  $$\sigma_2(x''[\aone ,\atwo ]) = \sigma_2(x_1^{-\aone }x_2^{-\atwo }) = x_3^{-\aone }x_2^{-\atwo } = x_1^{\aone }x_2^{-\atwo }P_2(x_2)^{-\aone }\,.$$
  We can also directly compute $x[\ell_2[\atwo ]_+ - \aone, \atwo] = x[-\aone , \atwo]$ using \Cref{thm: generalized greedy compatible gradings}. Since $\calP_{[-\aone ]_+,[\atwo ]_+} = \calP_{-\aone ,0}$, we have $x[-\aone , \atwo] = x_1^{\aone }x_2^{-\atwo }P_2(x_2)^{-\aone }$, as desired.
  
  Next, suppose $\aone  \leq 0 < \atwo $. Then $\calP([\aone ]_+,[\atwo ]_+)$ consists of $\atwo $ vertical edges. Similar to the previous case, we have $x''[\aone ,\atwo ] = x_1^{-\aone }x_2^{-\atwo }p_{1,\ell_1}^{\atwo} \overline{P_1}(x_1)^{\atwo }$. Applying $\sigma_2$, we have
  \begin{align*}
  \sigma_2(x''[\aone ,\atwo ]) &= x_3^{-\aone }x_2^{-\atwo }p_{1,\ell_1}^{\atwo} \overline{P_1}(x_3)^{\atwo }\\
  &= x_1^{\aone }x_2^{-\atwo }p_{1,\ell_1}^{\atwo}  P_2(x_2)^{-\aone }\overline{P_1}(P_2(x_2)x_1^{-1})^{\atwo }
  \end{align*}
  Note that every vertical edge in $\calD_{\ell_2a_2-\aone ,\atwo }$ has at least $\ell_2$ vertical edges in front of it. Thus any compatible grading with vertical weights bounded by $\ell_2$ can be constructed by taking a set of vertical edges and any horizontal edges not contained in the shadow of the vertical edges. Focusing just a vertical edge and the edges in its shadow, which is identical to $\calD_{\ell_2,1}$, we can see that 
  $$\sum_{\omega \in \calC_{\ell_2,1}} \wt(\omega) x_1^{\omega(\calP_\north)}x_2^{\omega(\calP_\east)} = p_{1,\ell_1}\overline{P_1}(P_2(x_2)x_1^{-1}) x_1^{\ell_2}\,.$$ We can view compatible gradings on $\calD_{\ell_2\atwo -\aone ,\atwo }$ as a collection of $\atwo $ compatible gradings on $\calD_{\ell_2,1}$ along with compatible gradings on $-\aone $ additional horizontal edges. So, using the compatible grading definition, we have
  \begin{align*}x[\ell_2\atwo - \aone, \atwo ] &= x_1^{-\ell_2\atwo+ \aone}x_2^{-\atwo } \left(p_{1,\ell_1}\overline{P_1}(P_2(x_2)x_1^{-1}) \right)^{\atwo } P_2(x_2)^{-\aone }\\
  &= x_1^{\aone }x_2^{-\atwo }p_{1,\ell_1}^{\atwo} P_2(x_2)^{-\aone }\overline{P_1}(P_2(x_2)x_1^{-1})^{\atwo }\,.
  \end{align*}

  Lastly, suppose $0 < \aone,\atwo $. Then we have $\displaystyle x''[\aone ,\atwo ] = x_1^{-\aone }x_2^{-\atwo } \sum_{\omega \in \calC_{\aone ,\atwo }} \wt''(\omega) x_1^{\omega(\calP_\north)}x_2^{\omega(\calP_\east)}$. Hence
  \begin{align*}\sigma_2(x''[\aone ,\atwo ]) &= x_3^{-\aone }x_2^{-\atwo } \sum_{\omega \in \calC_{\aone ,\atwo }} \wt''(\omega) x_3^{\omega(\calP_\north)}x_2^{\omega(\calP_\east)}\\
  &= x_1^{-\ell_2\atwo  + \aone}x_2^{-\atwo } \sum_{\omega \in \calC_{\aone ,\atwo }} \wt''(\omega) x_1^{\ell_2 \atwo - \omega(\calP_\north)} P_2(x_2)^{\omega(\calP_\north) - \aone}x_2^{\omega(\calP_\east)}\,.
  \end{align*}
  
  Let $\calC^{\sh}_{\aone,\atwo}$ denote the set of compatible gradings on $\calP(\aone,\atwo)$ such that $S_\east \subseteq \sh(S_\north)$, where $S$ is the set of edges with nonzero value in the grading.   Given a grading $\omega \in \calC^{\sh}_{\aone,\atwo}$, define the grading $\theta(\omega) \coloneqq (\Omega(S_\east), \phi_{\ell_2}^*S_\north)$ on $\calD_{\ell_2\atwo -\aone ,\atwo }$ as in \cite[Section 4.2]{Rupgengreed}. By definition, we have $\wt(\theta(\omega)) = \wt''(\omega)$.  Then by \cite[Lemma 4.4 and Proposition 4.20]{Rupgengreed}, we have
  \begin{align*}
       x[\ell_2\atwo -\aone ,\atwo] &=  x_1^{-\ell_2 \atwo + \aone} x_2^{-\atwo}\sum_{\omega' \in \calC_{[\ell_2\atwo -\aone]_+,\atwo}} \wt(\omega') x_1^{\omega'(\calP_\north)}x_2^{\omega'(\calP_\east)}\\
       &=  x_1^{-\ell_2 \atwo + \aone} x_2^{-\atwo} \sum_{\omega \in \calC^{\sh}_{\aone,\atwo}} \wt(\theta(\omega)) x_1^{\ell_2\atwo - \omega(\calP_\north)}x_2^{\omega(\calP_\east)} \\
      &= x_1^{-\ell_2 \atwo + \aone} x_2^{-\atwo} \sum_{\omega \in \calC_{\aone,\atwo}} P_2(x_2)^{\omega(\calP_\east) - \aone}\wt''(\omega) x_1^{\ell_2\atwo - \omega(\calP_\north)}x_2^{\omega(\calP_\east)}.
      \end{align*}
Thus in all cases, we can conclude $\sigma_2(x''[\aone,\atwo]) = x[\ell_2\atwo -\aone ,\atwo]$.
\end{proof}

\begin{cor}[{following \cite[Corollary 2.18]{Rupgengreed}}]\label{cor: gen greedy containment}  For any $(\aone, \atwo) \in \Z^2$ the generalized greedy element $x[\aone,\atwo] \in \calT_1$ defined by \eqref{eq: greedy expansion} is contained in $\calA(P_1,P_2)$.
\end{cor}
\begin{proof}
  First note that $x'[d_1,d_2]$ and $x''[d_1,d_2]$ are contained in $\calT_1$ for any $d_1,d_2 \in \Z^2$.  By definition, for any $z \in \calT_1$, we have $\sigma_1(z) \in \calT_0$ and $\sigma_2(z) \in \calT_2$.  The result then follows directly from \Cref{thm: strong laurent} and \Cref{prop: reflection greedy}.
\end{proof}

\begin{proof}[Proof of \Cref{thm: generalized greedy compatible gradings}]
This can be proved in the same way as \cite[Theorem 2.15]{Rupgengreed}, except with the above results replacing their more specialized analogues in \cite{Rupgengreed}.  In particular, we replace references to \cite[Proposition 2.17]{Rupgengreed} with \Cref{prop: reflection greedy} and \cite[Corollary 2.18]{Rupgengreed} with \Cref{cor: gen greedy containment}, and the remainder of the argument proceeds unchanged.
\end{proof}

\section{Every wall coefficient in the Badlands is non-trivial}\label{app: Badlands nonvanishing}

As a direct application of our results, we can prove that for any rank-$2$ generalized cluster scattering diagram, every wall-function coefficient of every wall in the Badlands is strictly positive. The Badlands of the scattering diagram $\frakD \coloneqq \Scat(P_1,P_2)$ is the same region as the Badlands of $\frakD_{\ell_1,\ell_2}$.

The density of the Badlands in $\frakD_{\ell_1,\ell_2}$ was suggested by various sources, including \cite[Example 1.15]{GHKK}, \cite[Example 1.6]{GPS}, and \cite[Example 1.4]{GS}, before recently being proved by Gr{\" a}fnitz and Luo \cite{GL2023dense}. In fact, Gr{\" a}fnitz and Luo prove the stronger result that every wall-function coefficient of walls in the Badlands is strictly positive (except those that trivially vanish by the choice of $\ell_1,\ell_2$). In the skew-symmetric case, density was proven in 2010 by Gross--Pandharipande \cite{GP} and strict positivity of every wall-function coefficient was proven in 2021 by Davison--Mandel \cite{DM}. All the previous approaches use either representation theory or deformations of scattering diagrams. Using our tight grading formula, we give a lower bound on every wall-function coefficient in the Badlands of generalized scattering diagrams, which in particular establishes strict positivity. As a consequence, we recover the results of \cite{GP, DM, GL2023dense} using elementary combinatorics.

\begin{thm}\label{thm: Badlands coefficients bound}
Suppose the coefficients of $P_1$ and $P_2$ are contained in $\R_{\geq 0}$. Fix $a,b$ coprime such that the wall of slope $\frac{b}{a}$ is in the Badlands of $\frakD \coloneqq \Scat(P_1,P_2)$, and fix $k \in \Z_{> 0}$ such that $\ell_1$ divides $ka$ and $\ell_2$ divides $kb$. Then the wall-function coefficient $\lambda(ka,kb)$ in $\frakD$ is at least $k\cdot p_{1,\ell_1}^{ka/\ell_1} p_{2,\ell_2}^{kb/\ell_2}$.
\end{thm}
\begin{proof}
We proceed by using the tight grading formula \Cref{thm: power shadow grading formula} (with $m = 1$) and constructing $k$ distinct tight gradings contributing to the coefficient $\lambda(ka,kb)$. Fix $c_1,c_2$ such that $ka = c_1 \ell_1$ and $kb = c_2 \ell_2$, and choose $(\aone,\atwo)$ according to \Cref{thm: power shadow grading formula}. Fix $i \in \{0,\dots,k - 1\}$, and define the origin to be the minimal $(x,y) \in \N^2$ such that $\atwo x - \aone y = i$. We then work on the $(x + y)$-th cyclic shift of $\calP(\aone ,\atwo )$, which cyclically shifts the leftmost $x + y$ edges of $\calP(\aone ,\atwo )$ to the right end. We then label the horizontal (resp. vertical) edges by $u_j$ (resp. $v_j$) from left-to-right (resp. bottom-to-top). Let $\omega$ be the grading that assigns value $c$ to horizontal edges $u_{j}$ for $1 \leq j \leq c_2$, value $b$ to vertical edges $v_{j}$ for $c_2 \ell_2 - c_1 + 1 \leq j \leq c_2 \ell_2$, and value $0$ to all other edges.

Consider the top left point of the vertical blocks in each column (call the set of these points $R$), and the top left point of the horizontal blocks in each row (call the set $B$). If the Dyck path starts at the origin $(0,0)$, then the points in $R$ lie to the left of the line
$$y = \left(-\ell_2 + \frac{\atwo}{\aone}\right) x + c_2\ell_2 - \frac{k-1}{\aone}\,.$$
Similarly, the points in $B$ lie to the right of the line
$$y = \left(\frac{1}{\frac{\aone}{\atwo} - \ell_1}\right) x + \frac{\ell_1(c_1 - c_2\ell_2) - \frac{k-1}{n}}{\frac{\aone}{\atwo}-\ell_1}\,.$$
In order to show that points in $R$ are to the left of the point in $B$ of the same height, it is sufficient to show that the intersection point of these lines is to the left of the $y$-axis.
Let $\alpha = \frac{\aone}{\atwo}$. These lines intersect when 
$$x = \frac{\alpha(\alpha c_2 \ell_2 - c_1 \ell_1) + \frac{\ell_1 k}{\atwo}}{\ell_2 \alpha^2 - \ell_1 \ell_2 \alpha + \ell_1}\,.$$
The numerator is always positive by our choice of $\alpha$. Thus, it is sufficient to have $\ell_2 \alpha^2 - \ell_1 \ell_2 \alpha + \ell_1 < 0$. But the roots of this quadratic are precisely 
$$\frac{\ell_1 \ell_2 \pm \sqrt{\ell_1^2 \ell_2^2 - 4\ell_1 \ell_2}}{2\ell_2}\,,$$
The greater of these two roots is the slope of the right boundary of the Badlands.

Repeating the same process swapping $b$ and $c$ handles all walls on the interior of the Badlands. For $\ell_1,\ell_2$ such that $\ell_1\ell_2 \neq 4$, the boundary of the Badlands will have irrational slope. Thus the result is proved for the entirety of the Badlands in all cases except the affine cases. In the affine cases, where $\ell_1\ell_2 = 4$, constructions of tight gradings corresponding to each wall-function coefficient are given in \Cref{exmp: affine 2 2 gradings} and \Cref{exmp: affine 4 1 gradings}.
\end{proof}

By specializing \Cref{thm: Badlands coefficients bound}, we recover the results of \cite[Theorem 1]{GL2023dense}, \cite[Section 4.7]{GP}, and \cite[Example 4.10]{DM}. 

In the case when $P_1$ and $P_2$ have positive coefficients (in an ordered field), we can then guarantee that every wall appears within the Badlands.

\begin{cor}\label{cor: generalized dense badlands}
Let $\Bbbk$ be an ordered field, and suppose the coefficients of $P_1$ and $P_2$ are nonnegative elements of $\Bbbk$. Then for $\frakD \coloneqq \Scat(P_1,P_2)$, there are walls of every rational slope within the Badlands. 
\end{cor}
\begin{proof}
As is shown in \Cref{thm: Badlands coefficients bound}, there are tight gradings of non-zero weight corresponding to each wall-function coefficient in the Badlands.  By our assumption on the coefficients of $P_1$ and $P_2$, the weights of these tight gradings are in fact positive. Thus, by \Cref{thm: generalized greedy compatible gradings}, the wall-function coefficient are sums of positive elements and hence are positive.
\end{proof}

An alternative proof of \Cref{cor: generalized dense badlands} can be  obtained using the results of Gr{\" a}fnitz and Luo for cluster scattering diagrams and applying our positivity result \Cref{thm: positive wall-function}. 

\begin{rem}
If we allow the coefficients of $P_1$ and $P_2$ to be negative, the results of \Cref{thm: Badlands coefficients bound}.  While the tight gradings constructed in \Cref{thm: Badlands coefficients bound} still contribute to the wall-function coefficients, the weights of the tight gradings may cancel out.  An explicit example of this phenomenon is given below.  We do not know if \Cref{cor: generalized dense badlands} can also fail in this setting. 
\end{rem}

\begin{exmp}\label{exmp: vanishing wall coefficient}
Using our combinatorial formula for wall-function coefficients, we can show that certain coefficients vanish.  For example, suppose $P_1(x) = 1 + p_{1,3}x^3$ and $P_2(y) = 1 + p_{2,1}y + p_{2,2}y^2$, where we assume $p_{1,3}$, $p_{2,1}$, and $p_{2,2}$ are nonzero.  In order to calculate the degree-$3$ coefficient of $f_{\R_{\leq 0}(1,1)}$ in $\Scat(P_1,P_2)$, which we denote by $\lambda(3,3)$, we can take a weighted sum over tight gradings using \Cref{thm: power shadow grading formula}.  In particular, it is the weighted sum of the seven tight gradings on $\calP(4,3)$ below.  

\begin{figure}[ht]
\centering
\begin{tikzpicture}[scale=.42]
\draw[step=1,color=gray] (0,0) grid (4,3);
\vrec{0}{0}{1}{1}
\vrec{0}{1}{1}{2}
\vrec{1}{0}{1}{2}
\hrec{4}{2}{1}{3}
\draw[line width=2,color=orange] (0,0)--(2,0)--(2,1)--(3,1)--(3,2)--(4,2)--(4,3);
\draw[line width=2pt] (0,0)--(2,0);
\draw[line width=2pt] (4,2)--(4,3);

\begin{scope}[shift={(5.5,0)}]
\draw[step=1,color=gray] (0,0) grid (4,3);
\vrec{0}{0}{1}{2}
\vrec{0}{2}{1}{1}
\vrec{1}{0}{1}{1}
\hrec{4}{2}{1}{3}
\draw[line width=2,color=orange] (0,0)--(2,0)--(2,1)--(3,1)--(3,2)--(4,2)--(4,3);
\draw[line width=2pt] (0,0)--(2,0);
\draw[line width=2pt] (4,2)--(4,3);
\end{scope}

\begin{scope}[shift={(11,0)}]
\draw[step=1,color=gray] (0,0) grid (4,3);
\vrec{0}{0}{1}{2}
\vrec{0}{2}{1}{1}
\vrec{2}{1}{1}{1}
\hrec{4}{2}{1}{3}
\draw[line width=2,color=orange] (0,0)--(2,0)--(2,1)--(3,1)--(3,2)--(4,2)--(4,3);
\draw[line width=2pt] (0,0)--(1,0);
\draw[line width=2pt] (2,1)--(3,1);
\draw[line width=2pt] (4,2)--(4,3);
\end{scope}

\begin{scope}[shift={(16.5,0)}]
\draw[step=1,color=gray] (0,0) grid (4,3);
\vrec{0}{0}{1}{2}
\vrec{0}{2}{1}{1}
\vrec{1}{1}{1}{1}
\hrec{4}{2}{1}{3}
\draw[line width=2,color=orange] (0,0)--(1,0)--(1,1)--(3,1)--(3,2)--(4,2)--(4,3);
\draw[line width=2pt] (0,0)--(1,0);
\draw[line width=2pt] (1,1)--(2,1);
\draw[line width=2pt] (4,2)--(4,3);
\end{scope}

\begin{scope}[shift={(22,0)}]
\draw[step=1,color=gray] (0,0) grid (4,3);
\vrec{0}{0}{1}{2}
\vrec{0}{2}{1}{1}
\vrec{2}{1}{1}{1}
\hrec{4}{2}{1}{3}
\draw[line width=2,color=orange] (0,0)--(1,0)--(1,1)--(3,1)--(3,2)--(4,2)--(4,3);
\draw[line width=2pt] (0,0)--(1,0);
\draw[line width=2pt] (2,1)--(3,1);
\draw[line width=2pt] (4,2)--(4,3);
\end{scope}

\begin{scope}[shift={(27.5,0)}]
\draw[step=1,color=gray] (0,0) grid (4,3);
\vrec{0}{0}{1}{2}
\vrec{0}{2}{1}{1}
\vrec{1}{1}{1}{1}
\hrec{4}{2}{1}{3}
\draw[line width=2,color=orange] (0,0)--(1,0)--(1,1)--(2,1)--(2,2)--(4,2)--(4,3);
\draw[line width=2pt] (0,0)--(1,0);
\draw[line width=2pt] (1,1)--(2,1);
\draw[line width=2pt] (4,2)--(4,3);
\end{scope}

\begin{scope}[shift={(33,0)}]
\draw[step=1,color=gray] (0,0) grid (4,3);
\vrec{0}{0}{3}{1}
\vrec{1}{0}{1}{1}
\vrec{2}{1}{1}{1}
\hrec{4}{2}{1}{3}
\draw[line width=2,color=orange] (0,0)--(2,0)--(2,1)--(3,1)--(3,2)--(4,2)--(4,3);
\draw[line width=2pt] (0,0)--(2,0);
\draw[line width=2pt] (2,1)--(3,1);
\draw[line width=2pt] (4,2)--(4,3);
\end{scope}
\end{tikzpicture}
\caption{For $P_1(x)$ and $P_2(y)$ chosen as in \Cref{exmp: vanishing wall coefficient}, these are the seven tight gradings on $\calP(4,3)$ with both total vertical and horizontal grade equal to $3$ that have nonzero weight.}
\end{figure}
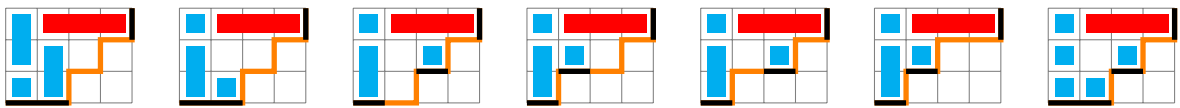

The leftmost six gradings have weight $p_{1,3}p_{2,1}p_{2,2}$ and the rightmost grading has weight $p_{1,3}p_{2,1}^3$.  Thus we have $\lambda(3,3) = 6p_{1,3}p_{2,1}p_{2,2} + p_{1,3}p_{2,1}^3$, so we can conclude that this wall-function coefficient vanishes whenever $p_{2,2} = -\frac{p_{2,1}^2}{6}$. 
\end{exmp}

\bibliographystyle{amsplain}
\bibliography{bibliography.bib} 
\addtocontents{toc}{\protect\setcounter{tocdepth}{1}}

\end{document}